\documentclass[11pt]{article}
\usepackage[utf8]{inputenc}
\usepackage[margin=1in]{geometry}
\usepackage{amsmath,amssymb,amsthm}
\usepackage{mathtools}
\usepackage[authoryear]{natbib}
\usepackage[hidelinks]{hyperref}
\renewcommand{\theenumi}{(\roman{enumi})}

\newtheorem{theorem}{Theorem}[section]
\newtheorem{lemma}[theorem]{Lemma}
\newtheorem{proposition}[theorem]{Proposition}
\newtheorem{corollary}[theorem]{Corollary}
\theoremstyle{definition}
\newtheorem{definition}[theorem]{Definition}
\newtheorem{convention}[theorem]{Convention}
\newtheorem{assumption}[theorem]{Assumption}
\newtheorem{conjecture}[theorem]{Conjecture}
\newtheorem{example}[theorem]{Example}
\theoremstyle{remark}
\newtheorem{remark}[theorem]{Remark}
\newcommand{\R}{\mathbb{R}}
\newcommand{\E}{\mathbb{E}}
\newcommand{\PP}{\mathbb{P}}
\newcommand{\eps}{\varepsilon}
\newcommand{\supp}{\operatorname{supp}}
\newcommand{\Law}{\operatorname{Law}}
\newcommand{\midW}{\mathrm{mid}_W}
\newcommand{\Mreg}{\mathcal{M}_{\mathrm{reg}}}
\newcommand{\Mstar}{\mathcal{M}_{\star}}
\newcommand{\Mclass}{\mathcal{M}}
\newcommand{\Pclass}{\mathcal{P}}
\newcommand{\clo}{\underline{c}}
\newcommand{\chigh}{\overline{c}}
\newcommand{\Beta}{B}
\begin{document}
\title{\bf Shape without scale: an identifiability dichotomy\\
for a bounded tail observed through a non-additive measurement kernel}
\author{Jiarui Qi\thanks{University of Washington. Email: \texttt{jiaruiqi@uw.edu}.}}
\date{}

\maketitle

\begin{abstract}
A latent severity has a bounded lower tail with density of shape $\alpha$ and scale $L$. It is observed only through a fixed Markov kernel $K$ that is biased and non-additive. The relative conditional spread of $K$ diverges at the endpoint. Our sample is i.i.d.\ from the marginal $Q$ alone, with no anchoring covariate or instrument. We prove a dichotomy. The shape index $\alpha$ is identifiable: for every admissible choice of the class constants, any two observationally equivalent members of a lean class share $\alpha$, determined by a near-endpoint expansion of $Q$. The rate, namely $L$ and the fixed-scale exceedance $p_\tau$, does not survive. There exist admissible shared class constants and two members of a smaller regularity class whose observed laws coincide exactly. Across the pair $\alpha$ agrees, whereas $L$ and $p_\tau$ move. A degenerate Le~Cam two-point bound excludes any uniformly consistent estimator of either, and pointwise consistency fails at one member. Only the rate needs an anchor. We conjecture that a known kernel family with known edge map identifies the rate fiber by fiber if and only if the family satisfies a fixed-scale injectivity clause, and we prove the sufficiency direction. In surrogate safety, uncalibrated conflict data give the shape of near-crash risk, not its absolute rate.
\end{abstract}



\section{Introduction}\label{intro:sec}
\subsection{A resolution floor that does not shrink with the signal}

Consider a latent scalar $S\ge0$ that measures how close a system came to a catastrophe, where $S=0$ is the boundary and small values of $S$ are dangerous. The variable $S$ is never observed. Instead, a measurement chain reports each true severity $s$ as a contaminated value $\tilde S$ drawn from a conditional law $K(\cdot\mid s)$, and we observe an i.i.d.\ sample from the resulting \emph{marginal} law $Q$ of $\tilde S$ alone. No covariate, instrument, pairing, or repeated measurement is available. The objects of interest are the lower-tail features of $S$ at the endpoint $0$: its \emph{shape} and the \emph{rate} of dangerous, small-severity outcomes.
Under a regularly varying tail $f_S(s)=L\,s^{\alpha}(1+o(1))$, the shape index $\alpha\in[0,\infty)$ measures how fast the latent density vanishes at the boundary. The scale $L$ fixes the near-endpoint mass, while the exceedance $p_\tau=F_S([0,\tau])$ at a fixed threshold $\tau$ is the probability of a dangerous outcome. This tail condition is equivalent to $-S$ lying in the Weibull max-domain of attraction with extreme-value index $\xi=-1/(\alpha+1)$ (\citealp[Thm.~1.2.1(2)]{deHaanFerreira2006}; \citealp{Coles2001}).

The main difficulty comes from a single structural feature of the measurement: \emph{the resolution floor of the detector does not scale down with the signal}. While the absolute sensor error is fixed, the score it corrupts shrinks toward $0$. Hence the relative conditional spread of $K(\cdot\mid s)$ does not vanish as $s\downarrow0$ but diverges there, on the band of severities that carries the estimand. The motivating instance is surrogate-safety estimation in traffic-conflict analysis \citep{Joo2024}, where the rate of rare crashes is estimated without observing any crash \citep{SongchitruksaTarko2006,Zheng2014}. One scores near-miss conflicts by a continuous severity, canonically time-to-collision (the encounter-minimum of the ratio of inter-agent gap to closing speed), and extrapolates the tail of the scores to the contact boundary $S=0$. Fixed sensing noise (below ten centimeters, \citealp{Krajewski2018}) is propagated through this min-of-ratio functional, whose denominator is itself a small difference of measured quantities. The result is a state-dependent, non-additive distortion that grows without bound at the dangerous, small-margin extreme. Two questions organize the paper: whether the \emph{shape} $\alpha$ survives in $Q$, and whether the \emph{rate}, namely the scale $L$ and the exceedance $p_\tau$, survives with it.

\subsection{Model and results}

A member is a pair $(F_S,K)$, where $F_S$ is a latent law on $[0,\infty)$ with endpoint $0$ and $K$ is a Markov kernel. Such a pair is observed only through the mixture $Q=KF_S=\int K(\cdot\mid s)\,F_S(ds)$, and the estimand is $\psi=(\alpha,L,p_\tau)$ (Definition~\ref{def:psi}). Three axioms carry the argument. A1 fixes the latent tail $f_S(s)=L\,s^{\alpha}(1+o(1))$, and A5 makes the lower edge $\ell(s)$ of the score move at least linearly with $s$. Under A6, the conditional mass within a width-$h$ strip above that edge vanishes like $h^{\beta+1}$. Here $\beta\ge0$ is a \emph{class constant}, one of several shared by every member, whereas $\alpha$ is member-specific. The regularity class $\Mreg$ adds the bias, bounded-support, spread and non-additivity axioms (A2--A4 and NA). The lean class $\Mstar$, by contrast, keeps only A1, A6 and a weakened edge axiom A5${}^-$, which gives the inclusion $\Mreg\subseteq\Mstar$ (Section~\ref{lean:sec}).

The shape survives. For \emph{every} admissible choice of the class constants, any two observationally equivalent members of $\Mstar$ share the shape index $\alpha$ (Theorem~\ref{thm:I2}), and identifiability over $\Mreg$ follows by inclusion (Corollary~\ref{cor:Mreg}). The main tool is a two-sided near-endpoint expansion (Lemma~\ref{lem:exponent}). Integrating the latent density $s^{\alpha}$ against the edge-mass profile $h^{\beta+1}$ over the width-$O(h)$ band of severities reaching $[a,a+h]$ gives
\[
Q\bigl([a,a+h]\bigr)\;\asymp\;L\,h^{\alpha+\beta+2}\qquad(h\downarrow0),
\]
with $a=\ell(0+)=\inf\supp Q$, where $\asymp$ means bounded above and below by fixed positive multiples. One $+1$ in the exponent $\alpha+\beta+2$ comes from the band and the other from the edge-mass profile. The exponent is a functional of $Q$ alone, so with $\beta$ a known class constant it determines $\alpha$.

The rate does not survive. There \emph{exist} admissible shared class constants and two members of $\Mreg$ forming a scale-conjugation pair, in which the latent law is dilated, $S\mapsto\kappa S$, while the kernel's edge slope is contracted, $c\mapsto c/\kappa$. Their observed laws coincide \emph{exactly}, $KF_S=K'F_S'$, and their shape indices agree, $\alpha'=\alpha$, yet both $L$ and $p_\tau$ move by the factor $\kappa^{-(\alpha+1)}\neq1$ (Theorem~\ref{thm:I1}). A degenerate Le~Cam two-point bound (Lemma~\ref{lem:bridge}) turns this into an impossibility result. Since both members induce the identical law of the data, no \emph{uniformly} consistent estimator of $L$ or of $p_\tau$ exists over $\Mreg$, and pointwise consistency fails at one of the two members. For example, with $\kappa=\tfrac34$ and $\alpha=1$ the scale moves $L:\tfrac29\to\tfrac{32}{81}$ and the exceedance $p_\tau:\tfrac1{36}\to\tfrac{4}{81}$, each by the factor $\kappa^{-2}=\tfrac{16}{9}$, a $77.8\%$ increase. The two observed laws are \emph{exactly identically distributed} (Lemma~\ref{lem:laws}), and a machine check of the certified instance reproduces the identity ($12/12$ assertions).

\subsection{Mechanism and the anchor}

A leading-order heuristic explains the dichotomy. Near the endpoint the bias acts on severity as an \emph{affine} map $S\mapsto cS+a$ with affine scale $c\neq1$ (Section~\ref{sec:app}), that is, as a two-parameter group on (endpoint,\,scale). The exponent $\alpha$ is an \emph{invariant} of this group, since a power law reparametrized affinely keeps its exponent, and the shape survives with it. The free parameters $(a,c)$, by contrast, absorb same-$\alpha$ differences in $(L,p_\tau)$, so the rate does not. The theorems rest on the exact near-endpoint expansion rather than on this heuristic.

The heuristic also suggests the remedy: an anchor that pins $(a,c)$ should restore identifiability of $(L,p_\tau)$. We formalize it as Anchor~C (Assumption~\ref{ass:C} of Section~\ref{sec:dichotomy}), which has \emph{two} clauses. First, the kernel is known to lie in a structured family $\{K_\theta\}$ whose lower edge map $\ell$ is known, so that the apparent endpoint and the edge slope are pinned. Second, a \emph{fixed-scale injectivity} clause is imposed: within the family, the map $F\mapsto K_\theta F$ is injective on latent laws supported below the family's severity bound (the finite cap on latent severity the family carries). The second clause is not implied by the first. Pinning $\ell$ closes the degree of freedom the confounding pair exploits, yet it does \emph{not} by itself invert the kernel at the fixed scale $\tau$. With only the first clause in place, in-fiber confounds for $p_\tau$ (distinct members with the same observed law, within a single edge-pinned family) are not ruled out.

The status of the results is as follows. The shape half is proved over $\Mreg$, indeed over the leaner $\Mstar$, with no anchor (Theorem~\ref{thm:I2}, Corollary~\ref{cor:Mreg}). Theorem~\ref{thm:I1} proves the \emph{necessity} of an anchor for $(L,p_\tau)$, whereas the necessity of Anchor~C \emph{in particular} is \emph{not} proved. We state the equivalence as a conjecture (Conjecture~\ref{conj:dichotomy}): for every family satisfying the first clause, $(L,p_\tau)$ is identifiable on each of its edge-indexed fibers of $\Mreg$ if and only if the family satisfies the second. Its sufficiency direction is immediate from Anchor~C, whose injectivity clause is automatic for location kernels. Discharging that clause for the \emph{non-location} kernels (those whose jitter is scaled rather than translated) is left to forthcoming companion work (Remark~\ref{rem:forthcoming}).

\subsection{Contributions}

The main contributions of this paper are the following.

\begin{itemize}
\item \textbf{A near-edge exponent lemma} with explicit two-sided constants $\varkappa_\pm$ (Lemma~\ref{lem:exponent}): $Q([a,a+h])\asymp L\,h^{\alpha+\beta+2}$ and $a=\inf\supp Q$. The near-edge constant $\varkappa$ is distinct from the dilation factor $\kappa$ of the scale-conjugation pair.

\item \textbf{Shape identifiability over the lean class $\Mstar$}, for every admissible choice of the class constants and on A1/A5${}^-$/A6 alone (Theorem~\ref{thm:I2}, Corollary~\ref{cor:Mreg}).

\item \textbf{Scale and exceedance non-identifiability over $\Mreg$}, existential (Theorem~\ref{thm:I1}). A degenerate Le~Cam two-point bound excludes any uniformly consistent estimator of $L$ or of $p_\tau$ (Lemma~\ref{lem:bridge}).

\item \textbf{One witness family, which also serves as a sharpness certificate}: the confounding pair keeps $\alpha'=\alpha$, so it \emph{corroborates} rather than contradicts the shape theorem. Although the pair moves the tail scale, the factor of the near-edge lead coefficient that in general varies from member to member comes out \emph{equal} for the pair's two members. Both then realize one near-endpoint exponent while the scale moves, which is the certificate of sharpness (Remark~\ref{rem:sharpness}).

\item \textbf{An explicit account of the axioms behind each half}: the shape half uses neither the upper edge nor A2--A4/NA, whereas the negative half verifies all of A1--A6 and NA.

\item \textbf{The anchored horn} (the anchored branch of the dichotomy)\textbf{, with its sufficiency direction proved}: under Anchor~C the triple $\psi=(\alpha,L,p_\tau)$ is determined fiber by fiber, and the injectivity clause is automatic for location kernels (Remark~\ref{rem:forthcoming}). Whether Anchor~C is also \emph{necessary} is the open direction of Conjecture~\ref{conj:dichotomy}.
\end{itemize}

\subsection{Positioning and outline}

We place the result by drawing three contrasts. \emph{Deconvolution is unavailable}: non-additivity removes the fixed error law, and with it the error characteristic function whose non-vanishing identifies the additive model \citep{Fan1991,StefanskiCarroll1990,Meister2009}, so the classical condition cannot even be stated here. \emph{The error is not asymptotically negligible}: in the contaminated extreme-value literature the measurement error is asymptotically negligible next to the sampling fluctuation of the order statistics, and tail-index estimators keep their limits \citep{Pere2024ApproxErrExtreme,Pere2025}. Here, however, the conditional spread \emph{diverges} where the estimand lives. \emph{Shape, not location}: under additive error ($\alpha$ the latent-edge exponent, $\beta$ the error-support edge exponent there) the composite $\alpha+\beta+2$ is the minimax-rate denominator for the endpoint \emph{location}, with $\alpha$ a nuisance the convolution does not recover \citep{GT2004}. That denominator is the quantity in the exponent that fixes the best achievable convergence rate. In the non-additive setting of this paper, the moving edge turns the same composite into an exponent read off $Q$ that identifies the boundary's \emph{shape}. The additive location confound widens into a location-and-scale confound.

The rest of this paper is organized as follows. Sections~\ref{sec:app}--\ref{scope:sec} contain the application, the setup, the dichotomy (with Anchor~C and Conjecture~\ref{conj:dichotomy} in Section~\ref{sec:dichotomy}), the witness family and pair, the related work and the discussion. The Supplementary Material collects the engine, the measure-theoretic preliminaries, the estimation bridge and sketch, the witness-family and pair proofs, the certified instance and the axiom accounting.
\section{A motivating measurement problem}\label{sec:app}

The axioms of the previous section may read as a list of analytic conveniences. This section shows that they are not. A single worked problem, estimating rare-crash risk from near-miss traffic conflicts, forces every one of them. Each axiom is \emph{read off} that application rather than imposed for tractability, and the problem makes concrete what ``a non-additive measurement kernel with a moving power-law edge'' means. We cite the applied literature only for the method it established and for the empirical error mechanisms it documents. The kernel $K$, its moving edge $\ell(s)$, the identification of $\alpha$, and the non-identifiability of $(L,p_\tau)$ without an anchor are our own.
\subsection{A min-of-ratio severity, measured with bounded bias and jitter}
\label{app:ssec:setup}

Crash frequency is estimated \emph{without observing a crash}. Analysts score near-miss encounters (``conflicts'') by a continuous \emph{surrogate severity}, then extrapolate the extreme-value fit of the smallest scores toward the catastrophe boundary (\citealp[Sect.~4]{SongchitruksaTarko2006}; \citealp{Tarko2018}). The standard score is \emph{time-to-collision} (TTC), introduced by \citet{Hayward1972}. Over an encounter on $[0,T]$ with inter-agent gap $d(t)\ge0$ and closing speed $v(t)>0$, \begin{equation}\label{app:eq:g} S \;=\; g(x)\;:=\;\min_{t\in[0,T]}\frac{d(t)}{v(t)}\;\in[0,\infty), \end{equation} the minimum gap-to-closing-speed ratio of the trajectory $x$. Thus $S=0$ encodes contact (the catastrophe), and \emph{small} $S$ is the dangerous regime. The latent law is the pushforward $F_S=g_\ast P_X$ of the unobserved law $P_X$ of conflict trajectories, a Borel probability law on $[0,\infty)$ with lower endpoint $0$. Since the estimand sits at that endpoint, the object of interest is the \emph{lower} tail of $S$.
Axiom~A1 models that tail: $F_S(\{0\})=0$ and $f_S(s)=L\,s^{\alpha}(1+o(1))$ as $s\downarrow0$, so $F_S(s)=\tfrac{L}{\alpha+1}s^{\alpha+1}(1+o(1))$ is regularly varying of index $\alpha+1$. Equivalently, $-S$ lies in the Weibull max-domain of attraction with index $\xi=-1/(\alpha+1)<0$ \citep{deHaanFerreira2006,Coles2001}. The member-specific shape index $\alpha$, the exponent that governs the density of $S$ near contact, is the parameter the analyst aims to recover.

The trajectory is never observed directly. It is reconstructed from drone or camera video placed in space by a fixed projective calibration (a homography, \citealp[Sect.~8.1.1]{HartleyZisserman2004}), a step that carries \emph{measurement error} of two kinds. One is a systematic \emph{bias}, a georeferencing offset that does not average out. Because projective error depends on an object's height \citep[Sect.~13.3, eq.~(13.9)]{HartleyZisserman2004}, hence on its class, this offset differs between the two agents even at a common ground location. The other is a bounded random \emph{jitter}: per-frame tracking noise with bounded support and no systematic component.
In absolute terms the localization error is small, on the order of ten centimeters after post-processing \citep{Krajewski2018}. Near the endpoint, however, its effect on the \emph{severity} scale is neither small nor benign. Joo et al.\ document this signal-dependent blow-up empirically \citep{Joo2024}. The score $\mathrm{TTC}=\Delta x/\Delta s$ is a ratio of a differenced gap to a differenced closing speed, where $\Delta x$ and $\Delta s$ denote tracked position and speed differences, not the trajectory $x$ and severity $s$ used above. Positional noise therefore propagates nonlinearly and systematically, skewing the score toward underestimation and biasing crash risk upward. In the dangerous regime the relative distortion explodes. The smallest-margin conflicts, which are the most safety-relevant, carry by far the largest relative error, in a ``cone-shaped'' scatter that denoising does not remove. That work models the input noise as \emph{additive on position} and quantifies the output scatter empirically. The non-additive severity-scale abstraction and the identifiability question it raises are ours, not claims made there \citep{Joo2024}.

\subsection{Propagation: a non-additive kernel with a moving power-law edge}
\label{app:ssec:prop}

A fixed input error, propagated through the nonlinear, ratio-valued functional \eqref{app:eq:g}, produces a non-additive contamination on the severity scale. The joint law of $(S,\tilde S)=(g(X),g(\tilde X))$, disintegrated over its first coordinate, defines a Markov kernel \[ K(A\mid s):=\PP\bigl(g(\tilde X)\in A\mid g(X)=s\bigr),\qquad A\in\mathcal B([0,\infty]), \] and the observed law is then the non-convolution mixture $Q=K F_S=\int K(\cdot\mid s)\, F_S(ds)$ of Section~\ref{sec:setting}. Non-additivity already appears in a constant-velocity encounter. There the bias differential perturbs the gap and the closing speed by deterministic offsets $(\delta_d,\delta_v)$, while the jitter adds bounded spread. The score \eqref{app:eq:g} therefore maps the latent severity by an \emph{affine action} \begin{equation}\label{app:eq:affine} S\;\longmapsto\;c\,S+a_0,\qquad c=\frac{v^\star}{v^\star+\delta_v}>0,\qquad a_0=\frac{\delta_d}{v^\star+\delta_v}\ge0, \end{equation} plus a bounded jitter term, where $v^\star$ denotes the relative speed. Four structural consequences follow, and they are the axioms.

\emph{A shifted apparent endpoint (A3).} A bias differential that does not vanish at contact (a per-agent, object-class-dependent homography error, $\delta_d(0)\neq0$) sends the latent endpoint $0$ to a strictly positive \emph{apparent endpoint} $a=\ell(0+)>0$, namely the affine intercept $a_0$ of \eqref{app:eq:affine} shifted by the lower support edge of the jitter. In kernel terms, a bounded but nonzero location offset $m(s)$ displaces the conditional law, so that $\sup_s|m(s)|\le\bar b<\infty$ and $m\not\equiv0$. This differential is what moves the boundary to which the analyst extrapolates.

\emph{A relative spread that blows up at the endpoint (A4).} The sensor fixes the input jitter while the score $S$ shrinks toward $0$. The conditional width $w(s)=u(s)-\ell(s)$ then tends to a nonzero limit $w(0+)>0$ as $s\downarrow0$, so that the \emph{relative} spread diverges, \[ \lim_{s\downarrow0}\frac{w(s)}{s}=+\infty. \]
Read on the severity scale, this is the blow-up that Joo et al.~\citeyearpar{Joo2024} document empirically, their error sensitivity fanning out into a cone as the kinematic margin shrinks. It describes a \emph{detector whose resolution floor does not shrink with the signal}, one that is least trustworthy where the estimand lives.

\emph{A moving lower edge with power-law edge mass (A5, A6).} Under the affine action \eqref{app:eq:affine} the lower edge $\ell(s)$ moves essentially linearly with the true severity. In the constant-velocity, product-form jitter model the motion is exactly linear, $\ell(s)=\ell(0)+c_\ell s$ with positive slope $c_\ell>0$, and there $c_\ell$ equals the affine scale $c$. A5 keeps only a two-point increment corridor, $c_-(s'-s)\le\ell(s')-\ell(s)\le c_+(s'-s)$ for $0\le s\le s'$ near $0$, together with global monotonicity of $\ell$, and so admits a mildly curved edge with slope inside $[c_-,c_+]$. The $s=0$ case of the corridor is the one-point corridor $c_-s'\le\ell(s')-a\le c_+ s'$. The conditional \emph{mass near that moving edge} inherits a power law from the jitter's support-edge profile. For independent components whose densities behave like $x^{\gamma_d}$ and $y^{\gamma_v}$ at those edges, \[ K\bigl([\ell(s),\ell(s)+h]\mid s\bigr)=c_K(s)\,h^{\beta+1}(1+o(1)),\qquad \beta=\gamma_d+\gamma_v+1, \] uniformly for small $s$. This is A6, with class-constant edge index $\beta$ and a prefactor $c_K(\cdot)$ controlled only through its $\liminf$ and $\limsup$ at $0$. Strict no-clipping of the numerator near contact keeps this prefactor bounded.

\emph{Non-additivity (NA).} The affine action \eqref{app:eq:affine} has scale
$c\neq1$ whenever the closing-speed offset $\delta_v$ is nonzero. A non-additive kernel is one for which no fixed law $\mu$ gives $K(\cdot\mid s)=\mu(\cdot-s)$ for $F_S$-a.e.\ $s$. Here the conditional support edges move with \emph{slopes other than $1$} (the lower edge with slope $c_\ell$), so the conditional law is not a rigid translate of a fixed error law. There is then \emph{no error characteristic function to divide out}, and classical additive deconvolution is unavailable. The bounded conditional support and the absence of escape to $+\infty$ near the endpoint (A2) follow from the physical jitter bounds and from a denominator floor that keeps the perturbed encounter closing.

\subsection{Modeling features and the axioms they force}\label{app:ssec:map}

We collect the correspondence in the following table: each axiom is read off a feature of the application.

\begin{center}
\renewcommand{\arraystretch}{1.25}
\begin{tabular}{p{0.50\textwidth} p{0.42\textwidth}}
\hline
\emph{Modeling feature of the conflict pipeline} & \emph{Axiom it forces} \\
\hline
Latent severity has a lower endpoint at $0$, and danger thins out near contact at a power-law rate (the larger the shape index $\alpha$, the faster the thinning). & \textbf{A1}: $F_S(\{0\})=0$ and $f_S(s)=L\,s^{\alpha}(1+o(1))$, with $-S$ in the Weibull domain, $\xi=-1/(\alpha+1)$. \\
Physical sensor and jitter bounds, along with a denominator floor that keeps the contaminated encounter closing (no escape to a non-closing encounter). & \textbf{A2}: $w(s)<\infty$ for all $s$, and $K(\{+\infty\}\mid s)=0$ near the endpoint. \\
The object-class-dependent homography differential does not vanish at contact, so the conditional law is displaced to a positive apparent endpoint. & \textbf{A3}: $\sup_s|m(s)|\le\bar b<\infty$ and $m\not\equiv0$, with apparent endpoint $a=\ell(0+)>0$. \\
Fixed sensor error against a shrinking score: the relative spread diverges near contact (the cone-shaped small-margin blow-up). & \textbf{A4}: $\lim_{s\downarrow0}w(s)/s=+\infty$. \\
The affine action $S\mapsto cS+a_0$ moves the contaminated lower edge linearly (up to a bounded corridor) with the true severity, and $\ell$ is globally nondecreasing. & \textbf{A5}: increment corridor $c_-(s'-s)\le\ell(s')-\ell(s)\le c_+(s'-s)$ near $0$, $\ell$ nondecreasing. \\
The conditional mass near the moving edge inherits a power-law profile from the support-edge exponents of the jitter. &
\textbf{A6}: $K([\ell(s),\ell(s)+h]\mid s)=c_K(s)\,h^{\beta+1}(1+o(1))$, common
$\beta=\gamma_d+\gamma_v+1$. \\
Scale $c\neq1$: the edges move with slopes other than $1$. The severity-scale error is a nonlinear ratio transform rather than a fixed additive shift. & \textbf{NA}: no fixed $\mu$ with $K(\cdot\mid s)=\mu(\cdot-s)$, and hence no error characteristic function. \\
\hline
\end{tabular}
\end{center}

The application is the near-edge setting treated in Section~\ref{sec:results}. Here the observed law concentrates near $a=\inf\supp Q$ with $Q([a,a+h])\asymp L\,h^{\alpha+\beta+2}$, and the composite exponent $(\alpha+1)+(\beta+1)$ splits into a band contribution and an edge-mass contribution. With $\beta$ a known class constant, the exponent read off $Q$ pins the latent shape $\alpha$, even though no error law can be deconvolved.

Section~\ref{sec:results} shows that this pipeline identifies shape but not scale. The shape index $\alpha$, which governs how fast near-crash danger thins as the surrogate score approaches contact, is recoverable from scored conflicts alone (Theorem~\ref{thm:I2}). By contrast, the \emph{rate} of that danger is not. Theorem~\ref{thm:I1} exhibits two intersections (two road sites) whose measured conflict scores are \emph{exactly identically distributed}. Yet the latent severe-event scale $L$ and the fixed-threshold exceedance $p_\tau$ each differ between the sites, by $77.8\%$ in the certified instance. The bias-and-jitter propagation offsets the change in latent scale by a compensating change in the moving edge, leaving the marginal law of the scores unchanged. No analysis of the observed conflicts can then separate the two sites, at any sample size. In practice this is a division of labour: absolute \emph{rate}, unlike \emph{shape}, needs an external calibration that pins the measurement scale, so neither an absolute near-crash frequency nor a ranking of two sites by near-crash rate should be read from uncalibrated conflict data alone.
\section{Setup and the model class}\label{sec:setting}

\subsection{The model pair and the observed law}\label{ssec:objects}

A member of the model is a pair $(F_S,K)$. Here $F_S$ is a Borel probability law on $[0,\infty)$ with endpoint $0$, and $K$ is a Markov kernel from $[0,\infty)$ to the compactified half-line $[0,\infty]$. That is, for each $s$, $K(\cdot\mid s)$ is a Borel probability measure on $[0,\infty]$, and for each Borel $B\subseteq[0,\infty]$ the map $s\mapsto K(B\mid s)$ is Borel measurable \citep{Kallenberg2002}. (The value $+\infty$ is admitted a priori, because contamination can turn a near-miss into a non-event. Near the endpoint, axiom~A2 of Section~\ref{ssec:axioms} excludes it.) The observed law is the mixture \begin{equation}\label{eq:Q} Q\;=\;K F_S\;=\;\int_{[0,\infty)}K(\cdot\mid s)\,F_S(ds), \end{equation} a Borel probability law on $[0,\infty]$. For Borel $B$, $Q(B)=\int K(B\mid s)\,F_S(ds)$, where the integrand is measurable by the kernel property and $[0,1]$-valued, so that no integrability condition is needed for \eqref{eq:Q}.

\begin{remark}[The data-generating object is the pair $(F_S,K)$]\label{rem:notfree}
The kernel $K$ is not a free functional of $F_S$. In the originating measurement model, it depends on the latent conditional structure on the level sets of an underlying functional (in the motivating instance, the time-to-collision map), and not on $F_S$ alone. Two latent laws with the same $F_S$ may induce different kernels, and identifiability is therefore posed over a class of pairs. In particular, two pairs $(F_S,K)$ and $(F_S',K')$ may have $K\neq K'$ yet still produce the same observed $Q$ (Definition~\ref{def:ident}). Any identifiability statement must allow the kernels to differ.
\end{remark}

For each $s\ge0$ the lower and upper \emph{conditional support edges} (restricted to the finite part) are \begin{equation}\label{eq:edges} \ell(s):=\inf\bigl(\supp K(\cdot\mid s)\cap[0,\infty)\bigr),\qquad u(s):=\sup\bigl(\supp K(\cdot\mid s)\cap[0,\infty)\bigr), \end{equation} where $\supp$ denotes the topological support in $[0,\infty]$. Their width and location offset are \begin{equation}\label{eq:wm} w(s):=u(s)-\ell(s),\qquad m(s):=\tfrac12\bigl(\ell(s)+u(s)\bigr)-s. \end{equation}
These are functionals of $K$ itself. Kernels with $K([0,\infty)\mid s)=0$ for some $s$ would empty the intersection in \eqref{eq:edges}, and are excluded. The axioms of Section~\ref{ssec:axioms} then quantify over well-defined edges. Between the two edges of \eqref{eq:edges}, it is the lower one, $\ell$, that the regularity axioms control, and through which the shape $\alpha$ is read. The letter $w$ is used in two senses: the width $w(s)$ always carries an argument, whereas the jitter edges $w_\pm$ of Convention~\ref{conv:W} always carry subscripts.

\subsection{The latent tail}\label{ssec:tail}

\begin{assumption}[Latent bounded tail (A1)]\label{ass:A1}
The law $F_S$ has no atom at the endpoint, $F_S(\{0\})=0$, and has a density $f_S$ on a right-neighbourhood $(0,s_0)$ of $0$, with constants $L\in(0,\infty)$ and $\alpha\in[0,\infty)$ such that \begin{equation}\label{eq:A1} f_S(s)=L\,s^{\alpha}\bigl(1+o(1)\bigr)\qquad(s\downarrow 0), \end{equation} hence $F_S(s)=\tfrac{L}{\alpha+1}s^{\alpha+1}(1+o(1))$, so $F_S$ is regularly varying at $0$ with index $\alpha+1$. (The ``hence'' relies on the no-atom clause: without it, e.g.\ $F_S=\tfrac12\delta_0+\tfrac12\mathrm{Unif}[0,1]$ satisfies \eqref{eq:A1} with $\alpha=0$ yet is not regularly varying of index $1$. Equation~\eqref{eq:A1} is understood for a fixed version of the density. The CDF asymptotic alone would be strictly weaker.)
\end{assumption}
Since $F_S$ has no atom at $0$, integrating \eqref{eq:A1} over a right-neighbourhood of $0$ gives the leading term of $F_S$, which we record here for later reference \citep[Prop.~1.5.8]{BGT1987}: \begin{equation}\label{eq:set-A1cdf} F_S(s)=\frac{L}{\alpha+1}\,s^{\alpha+1}\bigl(1+o(1)\bigr)\qquad(s\downarrow0), \end{equation} see Remark~\ref{rem:set-A1cdf} for the proof of \eqref{eq:set-A1cdf}.
\begin{remark}[Proof of the CDF consequence \eqref{eq:set-A1cdf}]\label{rem:set-A1cdf}
Fix $\eps\in(0,1)$. By the pointwise reading of \eqref{eq:A1} for the fixed version of the density, there is $\delta=\delta(\eps)\in(0,s_0)$ with $(1-\eps)Ls^{\alpha}\le f_S(s)\le(1+\eps)Ls^{\alpha}$ for all $s\in(0,\delta)$. For $s\in(0,\delta)$, the no-atom clause and the density clause of A1 give $F_S(s)=F_S(\{0\})+F_S((0,s])=\int_{(0,s]}f_S(t)\,dt$ (the sets $(0,s]$ and $(0,s)$ differ by the Lebesgue-null $\{s\}$). Integrating the two-sided bound then yields \[ (1-\eps)\,\frac{L}{\alpha+1}\,s^{\alpha+1}\;\le\;F_S(s)\;\le\;(1+\eps)\,\frac{L}{\alpha+1}\,s^{\alpha+1} \qquad\text{for all }s\in(0,\delta(\eps)), \] which is \eqref{eq:set-A1cdf} in $\eps$-$\delta$ form. Without the no-atom clause the first equality fails, since an atom at $0$ adds $F_S(\{0\})>0$ to $F_S(s)$, so that \eqref{eq:set-A1cdf} no longer holds while \eqref{eq:A1} is unaffected. We also note, for later use, that $F_S(s)>0$ for every $s>0$ small enough, as $L>0$. It follows by monotonicity that $F_S(\sigma)>0$ for every $\sigma>0$.
\end{remark}

In the positive, or shape, half of this paper, the estimand is the shape index $\alpha$. The scale $L$ and the fixed-scale exceedance $p_\tau$ are not recoverable without a further anchor (Theorem~\ref{thm:I1} and Section~\ref{scope:sec}). In extreme-value terms, \eqref{eq:A1} states that $-S$ lies in the Weibull max-domain of attraction with index $\xi=-1/(\alpha+1)<0$. This reading is not used in the proofs.
\subsection{Kernel and edge-regularity axioms}\label{ssec:axioms}

The axioms are imposed directly on the pair $(F_S,K)$ (Remark~\ref{rem:notfree}). Throughout, the following \emph{shared class constants} are held fixed: an edge index $\beta\ge0$ (A6), bounds $\bar b,\ c_-,\ c_+$, a fixed exceedance threshold $\tau$ with $0<\tau<s_R$, and neighbourhood floors $\bar s_0>0$ and $s_R,s_A\in(0,\bar s_0]$. Of these, $\tau$ excludes no member and only names the exceedance $p_\tau$. The floor $\bar s_0$ is a lower bound on every member's A1-neighbourhood, whereas $s_R$ and $s_A$ are the radii on which, respectively, the edge-regularity axioms A5--A6 and the no-escape clause of A2 are imposed. Every member satisfies its axioms with these same constants, and in particular has an A1-neighbourhood $(0,s_0)$ with $s_0\ge\bar s_0$. In contrast, the quantities $\alpha$, $L$, $F_S$, $K$, $c_K(\cdot)$, and hence the apparent endpoint $a:=\ell(0+)$ of A5, are member-specific.

\begin{assumption}[Kernel axioms]\label{ass:kernel}
The kernel $K$ satisfies:
\begin{itemize}
\item[\textbf{A2}] (\emph{bounded conditional support, no escape near the endpoint}) $w(s)<\infty$ for every $s\ge 0$, and $K(\{+\infty\}\mid s)=0$ for all $s\in[0,s_A]$. \item[\textbf{A3}] (\emph{bias present and bounded}) $\sup_{s}|m(s)|\le\bar b<\infty$ and $m\not\equiv 0$. \item[\textbf{A4}] (\emph{tail-amplified heteroscedasticity}) the relative width blows up at the endpoint: $\displaystyle\lim_{s\downarrow 0}\frac{w(s)}{s}=\infty$.
\item[\textbf{NA}] (\emph{non-additivity}) there is no fixed $\mu\in\Pclass(\R)$ with
$K(\cdot\mid s)=\mu(\cdot-s)$ for $F_S$-a.e.\ $s$. (Here $\Pclass(\R)$ is the set of Borel probability laws on $\R$. The relation ``$K(\cdot\mid s)=\mu(\cdot-s)$'' means that $K(\{+\infty\}\mid s)=0$ and $K(B\mid s)=\mu(B-s)$ for every Borel $B\subseteq[0,\infty)$, i.e.\ $K(\cdot\mid s)$ is the law of $s+Z$ with $Z\sim\mu$. NA thus excludes the classical additive-error model, in which deconvolution techniques would apply \citep{Fan1991,StefanskiCarroll1990,Meister2009}.)
\end{itemize}
\end{assumption}

\begin{assumption}[Edge-regularity axioms]\label{ass:reg}
For the class constant $s_R$ above, the kernel also satisfies:
\begin{itemize}
\item[\textbf{A5}] (\emph{edge-regular location transfer, no dipping}) the limit $a:=\ell(0+)\in[0,\infty)$ exists and, with $\ell(0):=a$, the \emph{two-point increment condition} holds, \begin{equation}\label{eq:A5} c_-\,(s'-s)\ \le\ \ell(s')-\ell(s)\ \le\ c_+\,(s'-s)\qquad\text{for all }0\le s\le s'\le s_R, \end{equation} and $\ell$ is \emph{nondecreasing on all of $[0,\infty)$}. (Taking $s=0$ gives the one-point corridor $c_-\,s'\le\ell(s')-a\le c_+\,s'$. The increment condition forbids the edge \emph{flattening} ($\ell(s)-a=o(s)$) and forbids \emph{flats} inside $(0,s_R]$. The global monotonicity prevents the conditional law at a distant $s$ from dipping to the apparent endpoint.)
\item[\textbf{A6}] (\emph{common conditional edge index}) for the class constant $\beta\ge0$ there is a map $c_K(\cdot)$ with $0<\liminf_{s\downarrow0}c_K(s)\le \limsup_{s\downarrow0}c_K(s)<\infty$ such that
\[
K\bigl([\ell(s),\,\ell(s)+h]\mid s\bigr)=c_K(s)\,h^{\beta+1}\bigl(1+o(1)\bigr)\qquad(h\downarrow0),
\]
uniformly over $s\in(0,s_R]$. Recall that $\beta$ is fixed across the class.
\end{itemize}
\end{assumption}

We now fix the reading of the A6 uniformity and the sign of the corridor constants, both of which are used throughout the proof.

\begin{convention}[Formal reading of the A6 uniformity]\label{conv:A6}\label{conv:A6read}
A6 is read in the standard uniform sense, that is, for every $\eps>0$ there exists $\eta(\eps)>0$ such that
\begin{equation}\label{eq:A6u}
(1-\eps)\,c_K(s)\,\delta^{\beta+1}\;\le\;K\bigl([\ell(s),\ell(s)+\delta]\mid s\bigr)\;\le\;(1+\eps)\,c_K(s)\,\delta^{\beta+1}
\qquad\text{for all } s\in(0,s_R],\ \delta\in(0,\eta(\eps)].
\end{equation}
Equivalently, $\sup_{s\in(0,s_R]}\bigl|K([\ell(s),\ell(s)+\delta]\mid s)/(c_K(s)\delta^{\beta+1})-1\bigr|\to0$
as $\delta\downarrow0$. No upper or positive lower bound on $c_K$ away from $0$ is assumed, and only the $\liminf/\limsup$ at $0$ are controlled. The declared codomain $(0,\infty)$ of $c_K$ imposes no restriction beyond the display itself. Indeed, if \eqref{eq:A6u} holds for some $[0,\infty]$-valued map, then that map is automatically $(0,\infty)$-valued on $(0,s_R]$. Finiteness follows since the lower band and $K\le1$ force $c_K(s)\le\bigl((1-\eps)\delta^{\beta+1}\bigr)^{-1}<\infty$. Positivity holds because $K([\ell(s),\ell(s)+\delta]\mid s)>0$ (Lemma~S1.1(iii) of the Supplementary Material) and the upper band together force $c_K(s)>0$.
\end{convention}

\begin{convention}[Sign of the corridor constants]\label{conv:cpm}\label{conv:corridor}
We read the class constants as $0<c_-\le c_+<\infty$. Positivity of $c_-$ is the content of the ``no flattening'' clause of A5 (the corridor $c_-s\le\ell(s)-a$ must forbid $\ell(s)-a=o(s)$), and it is required for the finiteness of the constant $\varkappa_+\propto c_-^{-(\alpha+1)}$ of Lemma~\ref{lem:exponent}. Likewise, $c_+<\infty$ is needed for $\varkappa_->0$. Both restrictions are therefore forced by that lemma.
\end{convention}

\subsection{Conventions}

\begin{convention}[Reading of the jitter]\label{conv:W}
In the constructions below, a \emph{jitter} is a real random variable $W$ with law $P_W$ satisfying the following two conditions.
\begin{enumerate}
\item The support edges $w_-:=\inf\supp P_W$ and $w_+:=\sup\supp P_W$ satisfy $0\le w_-<w_+<\infty$. In particular, $P_W$ has compact support $\supp P_W\subseteq[w_-,w_+]\subseteq[0,\infty)$.
\item There is $\delta_W>0$ such that the restriction of $P_W$ to $[w_-,w_-+\delta_W]$ is absolutely continuous with respect to Lebesgue measure and has a fixed, named density version $\rho_W$ satisfying
\begin{equation}\label{eq:set-rhoW}
\rho_W(w)=c_W\,(w-w_-)^{\beta}\bigl(1+o(1)\bigr)\qquad(w\downarrow w_-),\qquad c_W\in(0,\infty),
\end{equation}
where $\beta\ge0$ is the shared class constant of A6. The asymptotic \eqref{eq:set-rhoW} is read pointwise for that version: for every $\eps>0$ there is $\delta(\eps)\in(0,\delta_W]$ with $(1-\eps)c_W(w-w_-)^{\beta}\le\rho_W(w)\le(1+\eps)c_W(w-w_-)^{\beta}$ for all $w\in(w_-,w_-+\delta(\eps)]$.
\end{enumerate}
We write the midpoint offset as $\midW:=\tfrac12(w_-+w_+)$, which is $>0$ since $w_+>w_-\ge0$. The lower-edge exponent of the jitter is the class index $\beta$ itself. In the constructions of Sections~\ref{sec:witness}--\ref{sec:pair}, this exponent is the sole source of the $\beta$ appearing in A6.
\end{convention}

\subsection{Model classes and identifiability}\label{ssec:classes}

\begin{definition}[The regularity class]\label{def:M}
With the shared class constants of Section~\ref{ssec:axioms},
\[
\Mclass:=\bigl\{(F_S,K): \text{A1--A4, NA hold}\bigr\},
\qquad
\Mreg:=\bigl\{(F_S,K)\in\Mclass:\ \text{A5, A6 hold}\bigr\}.
\]
The wide class $\Mclass$ is introduced for context only (Remark~\ref{rem:whyreg}). We state the positive (shape identifiability) result over the lean class $\Mstar$ of Section~\ref{lean:sec}, and hence by inclusion also over $\Mreg$, whereas the negative (non-identifiability) result is stated over $\Mreg$. Identification is studied for a fixed pair $(F_S,K)$ with observed law $Q=KF_S$, and the contamination does not vanish.
\end{definition}

\begin{definition}[Estimand]\label{def:psi}
For a member with latent law $F_S$ satisfying A1, the estimand is the triple
\[
\psi(F_S):=\bigl(\alpha,\;L,\;p_\tau\bigr)\in[0,\infty)\times(0,\infty)\times(0,1],
\]
where:
\begin{enumerate}
\item $\alpha$ is the \emph{shape} index of \eqref{eq:A1}, the power of $s$ in the near-endpoint density. Only $L$ and $p_\tau$ are targets of the negative half. The index $\alpha$ is listed to locate them within $\psi$, and its identifiability over $\Mreg$ is the content of the positive half, Theorem~\ref{thm:I2} (with Corollary~\ref{cor:Mreg}).
\item $L$ is the \emph{tail scale}: the leading coefficient of the density's edge expansion, $f_S(s)=Ls^{\alpha}(1+o(1))$ as $s\downarrow0$. It is a \emph{scale} coefficient, not an endpoint or location parameter, since the endpoint is fixed at $0$ by A1. In CDF form it determines the unconditional near-edge mass via \eqref{eq:set-A1cdf}, $F_S(v)=\tfrac{L}{\alpha+1}v^{\alpha+1}(1+o(1))$ as $v\downarrow 0$.
\item $p_\tau$ is the \emph{fixed-scale exceedance} at the fixed, known threshold $\tau$ (a shared class constant, $0<\tau<s_R$):
\[
p_\tau:=F_S\bigl([0,\tau]\bigr)=F_S(\tau)\in(0,1],
\]
where the two expressions coincide because $F_S(\{0\})=0$, and $p_\tau>0$ by the byproduct noted in Remark~\ref{rem:set-A1cdf}. This is an \emph{exact} functional of $F_S$, \emph{not} a near-edge asymptotic. Since A1 is asymptotic only, nothing justifies replacing $p_\tau$ by the leading term $\tfrac{L}{\alpha+1}\tau^{\alpha+1}$, even when $\tau<s_0$. As a probability, $p_\tau$ is invariant under reparametrization of a \emph{fixed} $F_S$ (for example, a change of coordinates in $(\alpha,L)$). It does change, however, under \emph{rescaling of the latent law itself}, $S\mapsto\kappa S$, which is the mechanism used in Section~\ref{sec:pair}.
\end{enumerate}
\end{definition}

\begin{definition}[Observational equivalence and identifiability]\label{def:ident}
Two pairs $(F_S,K)$, $(F_S',K')$ are \emph{observationally equivalent}, written $(F_S,K)\sim(F_S',K')$, if $KF_S=K'F_S'$ as elements of $\Pclass([0,\infty])$. Equivalently, $(KF_S)(B)=(K'F_S')(B)$ for every Borel $B\subseteq[0,\infty]$. A functional $\phi$ of the pair $(F_S,K)$ is \emph{identifiable over} a class $\mathcal{N}$ of pairs if
\[
(F_S,K)\sim(F_S',K'),\ \text{both in }\mathcal{N}
\quad\Longrightarrow\quad
\phi(F_S,K)=\phi(F_S',K').
\]
Thus, to show that $\phi$ is not identifiable over $\mathcal{N}$, it suffices to exhibit one observationally equivalent pair of members of $\mathcal{N}$ on which $\phi$ differs. In this paper $\phi\in\{\alpha,L,p_\tau\}$, where $\alpha$ is the shape index of $F_S$ via \eqref{eq:A1}. Each $\phi$ depends on the pair only through $F_S$, so we write $\phi(F_S)$.
\end{definition}

\begin{lemma}[Single-valuedness of the estimand functionals]\label{lem:phiwell}
Let $F_S$ be a Borel probability law on $[0,\infty)$, and suppose that both of the following hold, where each asymptotic relation is read pointwise for the version named in it, as in \eqref{eq:A1}:
\begin{enumerate}
\item a nonnegative Borel function $f_S$ is a density of $F_S$ on $(0,s_0)$ for some $s_0>0$, and $f_S(s)=L\,s^{\alpha}\bigl(1+o(1)\bigr)$ as $s\downarrow0$, with $(L,\alpha)\in(0,\infty)\times[0,\infty)$.
\item a nonnegative Borel function $\tilde f_S$ is a density of $F_S$ on $(0,\tilde s_0)$ for some $\tilde s_0>0$, and $\tilde f_S(s)=\tilde L\,s^{\tilde\alpha}\bigl(1+o(1)\bigr)$ as $s\downarrow0$, with $(\tilde L,\tilde\alpha)\in(0,\infty)\times[0,\infty)$.
\end{enumerate}
Then $\alpha=\tilde\alpha$ and $L=\tilde L$. It follows that on any class $\mathcal{N}$ of pairs $(F_S,K)$ all of whose members satisfy A1 (in particular on $\Mclass$ and on $\Mreg$), the assignment $(F_S,K)\mapsto(\alpha,L)$ is a single-valued map. A1 supplies a choice of version, constants and neighbourhood, and all such choices return the same value. The map $(F_S,K)\mapsto p_\tau=F_S([0,\tau])$, the value of the measure $F_S$ at the fixed Borel set $[0,\tau]$ (Definition~\ref{def:psi}), is single-valued without further argument. Each $\phi\in\{L,p_\tau\}$ is a well-defined real-valued functional on $\mathcal{N}$ that depends on the pair only through its first coordinate. Accordingly, we write $\phi(F_S)$ for its value, as in Lemma~\ref{lem:bridge}.
\end{lemma}

\begin{remark}[Necessity and role of the regularity axioms]\label{rem:whyreg}
Over the wide class $\Mclass$, no feature of the tail is identifiable, $\alpha$ included. An explicit confounder is as follows. On $[0,1]$ take $F^{(\alpha)}(s)=s^{\alpha+1}$ (so A1 holds with $L=\alpha+1$) and the \emph{quantile-encoding} kernel $K^{(\alpha)}(\cdot\mid s)=\mathrm{Unif}\bigl[s^{\alpha+1},\,s^{\alpha+1}+c\bigr]$. Since $F^{(\alpha)}(S)\sim\mathrm{Unif}[0,1]$, the observed law is $K^{(\alpha)}F^{(\alpha)}=\mathrm{Unif}[0,1]*\mathrm{Unif}[0,c]$, the law of a sum of independent uniforms. This law is the same for every $\alpha>0$, whereas $\alpha$ itself differs.
Each pair with $\alpha>0$ satisfies A2--A4 and NA. Here $w\equiv c$, the function $m(s)=s^{\alpha+1}+c/2-s$ is bounded on the support and $\not\equiv0$, $w(s)/s\to\infty$, and the $s$-dependent edge displacement $s^{\alpha+1}$ precludes a fixed convolution. As for the witnesses of Section~\ref{sec:wit}, we extend $K^{(\alpha)}(\cdot\mid s)$ beyond $s=1$ by pure translation. This extension keeps $m$ constant at $c/2$ there and yields the global bound $\sup_{s\ge0}|m(s)|<\infty$ required by A3. Since it sits on the $F^{(\alpha)}$-null set $(1,\infty)$, it leaves $Q$ unchanged. The axiom that fails is A5: the edge $\ell(s)=s^{\alpha+1}$ flattens, that is, $\ell(s)-\ell(0+)=o(s)$. Under A5--A6 such flattening is ruled out, and the observed law acquires a near-endpoint exponent strictly increasing in $\alpha$. This exponent is the non-additive analogue of the composite index of Goldenshluger \& Tsybakov \citeyearpar{GT2004}, and Lemma~\ref{lem:exponent} makes it explicit. (Stochastic monotonicity alone would not suffice, since the quantile-encoding kernel is stochastically monotone.)
\end{remark}
\subsection{The lean class \texorpdfstring{$\mathcal{M}_\star$}{M-star}}\label{lean:sec}

The proof given in the Supplementary Material does not need the full strength of the model class. It uses only the latent tail axiom~A1 and the conditional edge-index axiom~A6, along with the lower-edge regularity contained in the $s=0$ corridor and global monotonicity of $\ell$. We therefore isolate a leaner class on which the same theorem holds, and $\Mreg$ is recovered as a special case. This is not merely a change of notation. Identifiability over the leaner class is a strictly stronger assertion, and the same proof gives it without extra work. The lean form of the location-transfer axiom is recorded first. It keeps the two features that the argument uses, namely a one-point corridor measured from the apparent endpoint and global monotonicity of the edge, whereas the two-point increment condition on the interior $0<s<s'$ is dropped.

\begin{convention}[Weakened location transfer, A5$^-$]\label{lean:A5minus}
We say that $(F_S,K)$ satisfies \emph{A5$^-$} if the limit $a:=\ell(0+)\in[0,\infty)$ exists, the \emph{one-point corridor}
\begin{equation}\label{lean:corridor}
c_-\,s\;\le\;\ell(s)-a\;\le\;c_+\,s\qquad\text{for all }s\in(0,s_R],
\end{equation}
holds with the class constants $0<c_-\le c_+<\infty$, and $\ell$ is nondecreasing on all of $[0,\infty)$. (As in A5, the symbol $\ell(0):=a$ is a notational convention and not a constraint on $K(\cdot\mid0)$. This is harmless because A1 makes $\{0\}$ an $F_S$-null set, cf.\ the discussion accompanying Proposition~S2.1 of the Supplementary Material.)
\end{convention}

\begin{remark}[A5 $\Rightarrow$ A5$^-$]\label{lean:A5implies}
The full axiom~A5 implies A5$^-$. Indeed, existence of $a$ and global monotonicity are common to both, and taking $(s,s')=(0,s)$ in the two-point increment condition \eqref{eq:A5} of A5, with the convention $\ell(0)=a$, yields the one-point corridor \eqref{lean:corridor}. The converse fails. Since A5$^-$ constrains only increments measured from the origin, it permits the edge to flatten or to develop flats on the interior $0<s<s'$, which makes it strictly weaker. In the proof, the interior increments of A5 are never used. Every appeal to A5 there passes through Lemma~S1.2 of the Supplementary Material, whose proof invokes \eqref{eq:A5} only at $(s,s')=(0,s)$ and otherwise uses global monotonicity. The entire argument therefore goes through unchanged under A5$^-$.
\end{remark}

\begin{convention}[The lean class $\mathcal{M}_\star$]\label{lean:Mstar}
With the shared class constants $\beta\ge0$, $0<c_-\le c_+<\infty$, $\bar s_0>0$ and $s_R\in(0,\bar s_0]$ fixed as before, define
\[
\mathcal{M}_\star:=\bigl\{(F_S,K):\ \text{A1, A5$^-$, A6, and the setup well-posedness convention hold}\bigr\}.
\]
Here the well-posedness convention is the one adopted in the setup of Section~\ref{ssec:objects}, namely that $K([0,\infty)\mid s)>0$ for every $s$ (so that $\supp K(\cdot\mid s)\cap[0,\infty)\neq\varnothing$ and $\ell(s)$ is well defined and finite). Just as in $\Mreg$, the constants $\beta,c_-,c_+,\bar s_0,s_R$ are shared across all members, whereas $\alpha,L,c_K(\cdot),F_S,K$ (hence $a$) are member-specific. The A1 neighbourhood $(0,s_0)$ satisfies $s_0\ge\bar s_0$. Identifiability over $\mathcal{M}_\star$ is understood in the usual sense, with $\phi=\alpha$, the shape index read off through A1.
\end{convention}

We now relate $\mathcal{M}_\star$ to $\Mreg$.

\begin{proposition}[The lean class contains the regular class]\label{lean:inclusion}
We have $\Mreg\subseteq\mathcal{M}_\star$, and the inclusion is, in general, strict.
\end{proposition}

\begin{proof}
Let $(F_S,K)\in\Mreg$. By definition, $(F_S,K)$ satisfies A1, A5, A6 along with A2--A4, NA and the setup well-posedness convention. Remark~\ref{lean:A5implies} shows that A5 implies A5$^-$. Then A1, A5$^-$, A6 and well-posedness all hold, so $(F_S,K)\in\mathcal{M}_\star$. The remaining conditions in the definition of $\Mreg$ are additional constraints not required for membership in $\mathcal{M}_\star$. These are the kernel axioms A2 (bounded conditional support, no escape near the endpoint), A3 (bias present and bounded), A4 (tail-amplified heteroscedasticity) and NA (non-additivity), together with the two-point interior increments of A5. They make $\Mreg$ a proper subclass of $\mathcal{M}_\star$. Indeed, consider a pair satisfying A1, A5$^-$, A6 and well-posedness whose edge flattens or develops a flat on the interior $0<s<s'$. Such a pair keeps the one-point corridor and global monotonicity, yet fails the two-point increment condition of full A5, so it lies in $\mathcal{M}_\star\setminus\Mreg$. Hence the inclusion is strict.
\end{proof}

\section{Main results}\label{sec:results}

Write $a:=\ell(0+)$ for the \emph{apparent endpoint} of a member. The engine of the positive half of the dichotomy is a near-endpoint expansion of the observed law, from which the shape identifiability theorem follows. Both are stated over the lean class $\Mstar$ of Section~\ref{lean:sec}. By the inclusion $\Mreg\subseteq\Mstar$ of Proposition~\ref{lean:inclusion}, they apply over $\Mreg$ as well.

\begin{lemma}[Near-edge exponent]\label{lem:exponent}
Let $(F_S,K)\in\Mstar$ with apparent endpoint $a$. Set $\clo:=\liminf_{s\downarrow0}c_K(s)$ and $\chigh:=\limsup_{s\downarrow0}c_K(s)$. Then there exist constants $0<\varkappa_-\le\varkappa_+<\infty$, depending only on $\beta,c_-,c_+$ and on the member through $\alpha$ and $(\clo,\chigh)$, such that for all sufficiently small $h>0$,
\begin{equation}\label{eq:twosided}
\varkappa_-\,L\,h^{\alpha+\beta+2}\;\le\;Q\bigl([a,a+h]\bigr)\;\le\;\varkappa_+\,L\,h^{\alpha+\beta+2}.
\end{equation}
In particular, $a=\inf\supp Q$ and
\begin{equation}\label{eq:exponent}
\lim_{h\downarrow0}\frac{\log Q([a,a+h])}{\log h}=\alpha+\beta+2,
\end{equation}
so that the near-endpoint exponent is determined by $Q$ alone. If in addition $\ell(s)=a+c_\ell s$ on $[0,s_R]$ and $c_K(s)\to\bar c_K\in(0,\infty)$ as $s\downarrow0$, then
\begin{equation}\label{eq:affine}
Q\bigl([a,a+h]\bigr)=\bar c_K\,\Beta(\alpha+1,\beta+2)\,L\,c_\ell^{-(\alpha+1)}\,h^{\alpha+\beta+2}\bigl(1+o(1)\bigr)\qquad(h\downarrow0),
\end{equation}
with $\Beta$ the Euler Beta function.
\end{lemma}

The explicit constants produced in the Supplementary Material are \begin{equation}\label{eq:kappas} \varkappa_-=\tfrac12\,\clo\,c_+^{-(\alpha+1)}\,\Beta(\alpha+1,\beta+2), \qquad \varkappa_+=2\,\chigh\,c_-^{-(\alpha+1)}\,\Beta(\alpha+1,\beta+2). \end{equation} The threshold below which \eqref{eq:twosided} holds is member-dependent, since it involves the $o(1)$ rates in A1, A6 and the $\liminf/\limsup$ in A6. The constants \eqref{eq:kappas} are not, given $(\alpha,\beta,c_-,c_+,\clo,\chigh)$. Here $\clo,\chigh$ (the $\liminf$ and $\limsup$ of the mass prefactor $c_K$) are distinct from the corridor slopes $c_-,c_+$. The pairing in \eqref{eq:kappas} is crossed: $\varkappa_-$ combines the lower prefactor $\clo$ with the \emph{upper} slope $c_+$, while $\varkappa_+$ combines the upper $\chigh$ with the \emph{lower} $c_-$. Note that $\varkappa_-\le\varkappa_+$, since $\clo\le\chigh$ and $c_-\le c_+$. The estimand $\alpha$ is member-specific. It enters $\varkappa_\pm$ alongside the class constants $(\beta,c_-,c_+)$, and identifiability would be vacuous if $\alpha$ were a class constant.

\begin{theorem}[Shape identifiability]\label{thm:I2}
The shape index $\alpha$ is identifiable over $\Mstar$: if $(F_S,K),(F'_S,K')\in\Mstar$ satisfy $KF_S=K'F'_S$, then $\alpha=\alpha'$.
\end{theorem}

Theorem~\ref{thm:I2} follows immediately from Lemma~\ref{lem:exponent}. If $(F_S,K)\sim(F'_S,K')$ then $Q=Q'$, so their common near-endpoint exponent \eqref{eq:exponent} equals both $\alpha+\beta+2$ and $\alpha'+\beta+2$. Since $\beta$ is the \emph{same} class constant for both members (Convention~\ref{lean:Mstar}), it follows that $\alpha=\alpha'$. The apparent endpoints coincide as well, because each equals $\inf\supp Q$ of the common law. In fact, the recovery map is explicit, \begin{equation}\label{eq:recovery} \alpha\;=\;\lim_{h\downarrow0}\frac{\log Q\bigl([\,\inf\supp Q,\ \inf\supp Q+h\,]\bigr)}{\log h}\;-\;\beta\;-\;2, \end{equation} a functional of the observed law $Q$ and the class constant $\beta$ alone. The complete argument, including the well-definedness of $\alpha$ as a functional of $F_S$, is deferred to the Supplementary Material.

\begin{corollary}[Identifiability over the regularity class]\label{cor:Mreg}
The shape index $\alpha$ is identifiable over $\Mreg$: any two observationally equivalent members of $\Mreg$ share $\alpha$. For every member of $\Mreg$, the two-sided bound \eqref{eq:twosided} holds with the same constants \eqref{eq:kappas}.
\end{corollary}

\begin{proof}
By Proposition~\ref{lean:inclusion}, $\Mreg\subseteq\Mstar$. Both claims therefore follow from Theorem~\ref{thm:I2} and Lemma~\ref{lem:exponent} applied to members of $\Mreg\subseteq\Mstar$.
\end{proof}

\begin{remark}[Affine-edge sharpening and the role of the lead constant]\label{lean:affine}
Suppose that the lower edge is exactly affine, $\ell(s)=a+c_\ell s$ on $[0,s_R]$ with $c_\ell\in[c_-,c_+]$, and that $c_K(s)\to\bar c_K\in(0,\infty)$. The two-sided bound then reduces to the asymptotic equality \eqref{eq:affine} (Proposition~S6.1 of the Supplementary Material). The lead coefficient there is proportional to the latent scale $L$, yet it also carries the kernel-dependent factors $\bar c_K$ and $c_\ell^{-(\alpha+1)}$. In this sense the expansion shows the $L$-proportionality of the lead constant without determining $L$ itself (Section~\ref{scope:sec}).
\end{remark}
To pass from a coincidence of observed laws to a statement about estimators, only the sampling setting of~\eqref{eq:Q} is needed. The statistician observes an i.i.d.\ sample $\tilde S_{1:n}=(\tilde S_1,\dots,\tilde S_n)$ from the observed law $Q=KF_S$. Here the pair $(F_S,K)$ is held fixed, so that the contamination does not vanish as $n\to\infty$. An \emph{estimator} of a functional $\phi$ is an arbitrary sequence $\hat\phi=(\hat\phi_n)_{n\ge1}$ of Borel maps $\hat\phi_n:[0,\infty]^n\to\R$. No restriction is placed on how $\hat\phi_n$ is formed: it may depend on $n$ and on complete knowledge of the class and of the two witnesses (the confounding pair exhibited by the next theorem). Such an estimator is \emph{uniformly consistent} for $\phi$ over a class if its worst-case error over the class (in mean or in probability) tends to $0$, and \emph{pointwise consistent} at a member if its error tends to $0$ there \citep[Ch.~2]{vdV1998}. The three modes are stated formally in Definition~S10.1 of the Supplementary Material. Lemma~\ref{lem:bridge} below shows that all three fail as soon as two members induce the same law of the data.
\begin{theorem}[Scale/exceedance non-identifiability over $\Mreg$]\label{thm:I1}
There exist admissible shared class constants $(\beta,\bar b,c_-,c_+,\tau,\bar s_0,s_R,s_A)$ (exhibited by the construction of Section~\ref{sec:pair}) and two members $(F_S,K)$, $(F_S',K')$ of the resulting class $\Mreg$ such that
\[
KF_S=K'F_S'\quad(\text{identical observed law, exactly}),
\qquad\text{yet}\qquad
\alpha'=\alpha,\quad L'\neq L,\quad p'_\tau\neq p_\tau,
\]
where $(\alpha,L,p_\tau)=\psi(F_S)$ and $(\alpha',L',p'_\tau)=\psi(F_S')$ are the estimands of Definition~\ref{def:psi}. Consequently, neither $L$ nor $p_\tau$ is identifiable over $\Mreg$ (Definition~\ref{def:ident}). By Lemma~\ref{lem:bridge}, neither admits a uniformly consistent estimator over $\Mreg$.
\end{theorem}

The construction behind Theorem~\ref{thm:I1} is a scale conjugation inside the affine-random witness family of Section~\ref{sec:wit}. We dilate the latent law, $S\mapsto\kappa S$, and contract the kernel's edge slope by the same factor, $c\mapsto c/\kappa$. The two deformations cancel in the observed mixture. Indeed, the pushforward identity $\Law\bigl((c/\kappa)(\kappa S)\bigr)=\Law(cS)$ matches each dilated latent value with the conditional law of the undilated one, $K'(\cdot\mid\kappa s)=K(\cdot\mid s)$, so that the observed law $Q=KF_S$ is exactly unchanged although the two kernels themselves differ. Both the tail scale $L$ and the exceedance $p_\tau=F_S([0,\tau])$, which depend on the latent law alone, are multiplied by the factor $\kappa^{-(\alpha+1)}\neq1$, whereas the shape $\alpha$ does not change. The quantifiers match those of the theorem: \emph{one} admissible tuple of shared class constants and \emph{two} members of the resulting $\Mreg$ are exhibited, not a confounding pair for every tuple. The construction needs $c\neq1$, so it never reaches the degenerate corner $c_-=c_+=1$, where $c\in[c_-,c_+]$ would force $c=1$. Example~S11.4 of the Supplementary Material realizes the construction in exact rationals ($\kappa=\tfrac34$, $\alpha=1$, $s_{\max}=3$, $c=\tfrac45$, $\tau=\tfrac12$). There $L$ and $p_\tau$ are each scaled by $\kappa^{-2}=\tfrac{16}{9}$, an increase of $77.8\%$, while the two observed laws stay exactly equal. Twelve machine-checked assertions confirm this (Section~S14.A of the Supplementary Material).
\begin{lemma}[Degenerate Le~Cam two-point bound]\label{lem:bridge}
Let $\phi\in\{L,p_\tau\}$ and suppose $\phi$ is not identifiable over a class $\mathcal{N}$ of pairs on which $\phi$ is single-valued (for $\phi\in\{L,p_\tau\}$, any class whose members satisfy A1 qualifies, by Lemma~\ref{lem:phiwell}). Fix members $(F^0,K^0)\sim(F^1,K^1)$ of $\mathcal{N}$ witnessing this failure, with $\phi(F^0)\neq\phi(F^1)$, and write $\Delta:=|\phi(F^0)-\phi(F^1)|>0$. Then for every sample size $n\ge1$ and every estimator $\hat\phi_n$ (that is, every Borel map $\hat\phi_n:[0,\infty]^n\to\R$, evaluated at the sample $\tilde S_{1:n}$ of \eqref{eq:Q}),
\[
\sup_{(F_S,K)\in\mathcal{N}}\ \E_{Q^{\otimes n}}\bigl|\hat\phi_n-\phi(F_S)\bigr|\;\ge\;\tfrac12\,\Delta,
\qquad\text{hence}\qquad
\inf_{\hat\phi_n}\ \sup_{(F_S,K)\in\mathcal{N}}\ \E_{Q^{\otimes n}}\bigl|\hat\phi_n-\phi(F_S)\bigr|\;\ge\;\tfrac12\,\Delta\;>\;0,
\]
where $Q=KF_S$ is the observed law of the member $(F_S,K)$ and $\E_{Q^{\otimes n}}$ denotes expectation under $\tilde S_{1:n}\sim Q^{\otimes n}$ (the expectations exist in $[0,\infty]$, the integrands being nonnegative). Hence no estimator sequence is uniformly consistent for $\phi$ over $\mathcal{N}$ (Definition~S10.1 of the Supplementary Material). Moreover, since $\hat\phi_n(\tilde S_{1:n})$ has the same law under the two witnesses, even \emph{pointwise} consistency cannot hold at both.
\end{lemma}

\begin{remark}[Sharpness]\label{rem:sharpness}
The scale-conjugation pair of Theorem~\ref{thm:I1} is a sharpness certificate for Theorem~\ref{thm:I2}, as can be seen from the affine sharpening \eqref{eq:affine}. Under the same conjugation ($S\mapsto\kappa S$, $c_\ell\mapsto c_\ell/\kappa$), the scale $L$ is multiplied by the factor $\kappa^{-(\alpha+1)}$ while $c_\ell^{-(\alpha+1)}$ is multiplied by the inverse factor $\kappa^{\alpha+1}$. Hence the member-dependent factor $L\,c_\ell^{-(\alpha+1)}$ of the lead coefficient of \eqref{eq:affine} is invariant (it equals $25/72$ for both members of the certified instance, Example~S11.4). The remaining factors $\bar c_K$ and $\Beta(\alpha+1,\beta+2)$ are shared by the pair, since $\bar c_K=c_W/(\beta+1)$ depends only on the shared class constants and the common jitter, and the Beta factor only on $\beta$ and the common shape $\alpha'=\alpha$. The full lead coefficient $\bar c_K\,\Beta(\alpha+1,\beta+2)\,L\,c_\ell^{-(\alpha+1)}$ is then invariant as well ($125/432$ in the instance), as the coincidence $Q'=Q$ requires. The pair realizes the near-endpoint exponent $\alpha+\beta+2$ while moving the scale. It corroborates the shape theorem rather than contradicting it, since $\alpha'=\alpha$ and the recovery map \eqref{eq:recovery}, evaluated on the common law $Q$, returns the common $\alpha$.
\end{remark}
\subsection{The dichotomy}\label{sec:dichotomy}

Theorems~\ref{thm:I2} and~\ref{thm:I1} resolve the two components of the estimand in opposite directions. The shape $\alpha$ is identifiable over $\Mreg$ with no anchor, whereas the scale $L$ and the exceedance $p_\tau$ are not. As described heuristically in the introduction (Section~\ref{intro:sec}), the leading-order mechanism behind the split is an affine bias orbit acting on (endpoint,\,scale), under which $\alpha$ is an invariant but $(L,p_\tau)$ are moved. In this subsection we describe what a full resolution would look like. First comes the anchor that a positive answer for $(L,p_\tau)$ would require, along with the fibers it cuts out. We then state the conjecture that would complete the proved shape half and the necessity direction into a single characterization.

\begin{assumption}[Anchor C]\label{ass:C}
The kernel belongs to a known family $\{K_\theta:\theta\in\Theta\}$, with $\Theta\subseteq\R^k$ finite-dimensional. This family has the standing structure (edge maps strictly increasing on $[0,\infty)$ and a severity bound $\bar s_{\mathcal K}<\infty$) and the two properties below.
\begin{itemize}
\item[\textbf{(C-a)}] (\emph{Known family, known edge.}) All non-$\theta$ structure of the family (the conditional shape of $K_\theta(\cdot\mid s)$ above its edge, the width map $w(\cdot)$, and the jitter law) is known to the analyst, and so is the lower edge map $\ell=\ell_{K_\theta}$. Within the family the edge determines $\theta$, so that $K$ is known up to the latent $F_S$.
\item[\textbf{(C-b)}] (\emph{Fixed-scale injectivity.}) For every $\theta\in\Theta$ the
mixture map $F\mapsto K_\theta F$ is injective on the probability measures on
$[0,\bar s_{\mathcal K}]$.
\end{itemize}
Pinning $\ell$ alone does \emph{not} invert the kernel at the fixed scale $\tau$ (Definition~\ref{def:psi}). Clause (C-b) is not implied by (C-a), and it carries the fixed-scale content of the anchor. Without it, in-fiber confounding pairs are not ruled out for edge-pinned families. Such a pair consists of two distinct members of one fiber (hence with a common edge) that yield the same observed law. The pair exhibited in Theorem~\ref{thm:I1} is not of this kind, since it moves the edge slope from $c$ to $c/\kappa$ and so crosses fibers. In effect, (C-b) asks for injectivity of the mixture operator $F\mapsto K_\theta F$. It is the non-additive, instrument-free counterpart of the completeness condition that characterizes identification in nonparametric instrumental-variable models, where identification of the structural function is equivalent to injectivity of a conditional-expectation operator \citep[Prop.~2.1]{NeweyPowell2003}. There that operator acts on \emph{functions} and is indexed by an \emph{instrument}, whereas here it acts on \emph{measures} and is indexed by the known kernel parameter $\theta$ at the fixed scale $\tau$. The marginal-only scheme \eqref{eq:Q} supplies no instrument.
\end{assumption}

We call $\mathcal K=\{K_\theta:\theta\in\Theta\}$ a \emph{candidate family} if it has the standing structure of Assumption~\ref{ass:C} (edge maps strictly increasing on $[0,\infty)$ and a severity bound $\bar s_{\mathcal K}<\infty$) and satisfies clause (C-a). The injectivity clause (C-b) is \emph{not} presupposed. An edge map $\ell_0$ is called \emph{admissible} if $\ell_0=\ell_{K_\theta}$ for some $\theta\in\Theta$. Each candidate family defines fibers of $\Mreg$ indexed by the observed edge:
\begin{equation}\label{def:fiber}
\Mclass_{\mathcal K}(\ell_0):=\bigl\{(F_S,K)\in\Mreg:\ K\in\mathcal K,\ \ \ell_K=\ell_0,\ \
\supp F_S\subseteq[0,\bar s_{\mathcal K}]\bigr\}.
\end{equation}

The shape half of the dichotomy is a theorem of this paper, not a conjecture. By Theorem~\ref{thm:I2} and Corollary~\ref{cor:Mreg}, with no anchor, \[ \alpha\ \text{is identifiable over }\Mreg \] or, equivalently, the observed law $Q$ determines $\alpha$ over all of $\Mreg$.

\begin{conjecture}[Sharpened dichotomy: the scale half]\label{conj:dichotomy}
Over the regularity class $\Mreg$ of Definition~\ref{def:M}, with identifiability read fiber by fiber (over all of $\Mreg$ at once it fails for suitable admissible shared class constants, by Theorem~\ref{thm:I1}), for every candidate family $\mathcal K$,
\[
\boxed{\quad(L,p_\tau)\ \text{is identifiable on every fiber }\Mclass_{\mathcal K}(\ell_0)\ \Longleftrightarrow\ \mathcal K\ \text{satisfies Anchor~C}.\quad}
\]
Since a candidate family satisfies (C-a) by definition, the right-hand side reduces to the injectivity clause (C-b). That is, for a candidate family $\mathcal K$, the observed law $Q$ determines $(L,p_\tau)$ on the fibers $\Mclass_{\mathcal K}(\ell_0)$ of \eqref{def:fiber} if and only if $\mathcal K$ satisfies (C-b).
\end{conjecture}

\begin{remark}[Status of the dichotomy]\label{rem:forthcoming}
The shape half is proved. Theorem~\ref{thm:I2}, with Corollary~\ref{cor:Mreg}, identifies $\alpha$ over $\Mreg$ (in fact over the leaner $\Mstar$), with no anchor. That \emph{some} anchor is necessary is proved as well. Together with Lemma~\ref{lem:bridge}, Theorem~\ref{thm:I1} exhibits, in its exact existential form, admissible shared class constants and two members of $\Mreg$ with identical observed law but distinct $(L,p_\tau)$. Hence over the \emph{unanchored} $\Mreg$ neither $L$ nor $p_\tau$ is identifiable. What is \emph{not} proved here is that Anchor~C \emph{in particular} is necessary, that is, the direction ``$(L,p_\tau)$ identifiable on every fiber $\Mclass_{\mathcal K}(\ell_0)$ $\Rightarrow$ $\mathcal K$ satisfies (C-b)'' of Conjecture~\ref{conj:dichotomy}. Indeed, Theorem~\ref{thm:I1} says nothing about candidate families that satisfy (C-a) but fail (C-b).

In contrast, the \emph{sufficiency} direction is immediate from Assumption~\ref{ass:C} as stated. Suppose two members of one fiber $\Mclass_{\mathcal K}(\ell_0)$ induce the same observed law. By (C-a), within the family the edge determines $\theta$, so their kernels coincide, say $K_\theta$. In view of \eqref{def:fiber}, both latent laws are carried by $[0,\bar s_{\mathcal K}]$, which is the domain on which (C-b) declares $F\mapsto K_\theta F$ injective. The two latent laws coincide, and with them $\psi=(\alpha,L,p_\tau)$ (Lemma~\ref{lem:phiwell}). All the mathematical content of the anchored horn lies in \emph{discharging} (C-b) for a given family, and there the location and non-location cases differ. For a \emph{location} family, $K_\theta(\cdot\mid s)=\Law\bigl(g_\theta(s)+W\bigr)$ with a jitter $W$ independent of $s$, clause (C-b) is redundant. Indeed, A2 bounds the conditional width, so $W$ has compact support and $\varphi_W$ is entire with $\varphi_W(0)=1$. Its real zeros are then isolated \citep{Meister2009}. Cancelling $\varphi_W$ off that discrete set in $\varphi_{(g_\theta)_\#F}\,\varphi_W=\varphi_{(g_\theta)_\#F'}\,\varphi_W$ and invoking continuity gives $(g_\theta)_\#F=(g_\theta)_\#F'$. Since $g_\theta$ is injective by the anchor's standing strict-monotonicity clause (Assumption~\ref{ass:C}), it follows that $F=F'$. The affine-random witness family of Section~\ref{sec:witness} is of this type, so (C-b) holds there automatically. What remains open is the \emph{non-location} case, that is, kernels whose jitter is \emph{scaled} rather than translated, as in the ratio and corner witnesses of Section~S12 of the Supplementary Material. The conditional width of such kernels grows with $s$, and the convolution factorization just used is not available. A discharge lemma for those families, and the anchored repair it would permit, are the subject of forthcoming companion work and are not established here. Accordingly, the biconditional is stated as a conjecture, not a theorem.

Two caveats fix the scope of the conjecture. First, over the wide class $\Mclass$ even the shape half is \emph{false} (Remark~\ref{rem:whyreg}), so the regularity axioms A5--A6 are part of the statement. Second, Anchor~C is strictly stronger than knowledge of the edge map. Clause (C-a) supplies the edge, whereas for \emph{non-location} kernels the fixed-scale injectivity clause (C-b) is an additional requirement. For location kernels it is automatic, as just noted.
\end{remark}
\section{The witness family and the confounding pair}\label{sec:wit}
Both theorems of Section~\ref{sec:results} concern members of a model class. Theorem~\ref{thm:I2} constrains every observationally equivalent pair in $\Mstar$, whereas Theorem~\ref{thm:I1} asserts the existence of a confounding pair in $\Mreg$. Each theorem is vacuous unless the class is non-empty, and in fact unless it contains genuinely non-additive members. Indeed, for an additive kernel $K(\cdot\mid s)=\mu(\cdot-s)$ the apparent endpoint is recovered by classical additive endpoint analysis \citep{GT2004}, and the non-convolution character that motivates the construction is lost. One explicit family meets both requirements. Its members lie in $\Mreg$, hence in $\Mstar$ by the inclusion $\Mreg\subseteq\Mstar$, so neither theorem is vacuous. The same family also supplies the scale-conjugation pair of Theorem~\ref{thm:I1}, assembled in Section~\ref{sec:pair}. That pair leaves the shape $\alpha$ unchanged while moving $L$ and $p_\tau$. The family realizes exactly the affine action $S\mapsto cS+a_0$ of \eqref{app:eq:affine}, that is, a latent severity composed with a random affine map of scale $c\neq1$ and a bounded jitter. The resulting kernel is non-additive since the conditional edge then moves with a slope other than $1$.

Recall from Section~\ref{ssec:objects} the objects attached to a member
$(F_S,K)$: the conditional lower and upper edges $\ell(s)$ and $u(s)$, the
width $w(s)=u(s)-\ell(s)$, the location offset $m(s)$, and the apparent
endpoint $a=\ell(0+)$. This section presents the family in full and shows that it lies in $\Mreg$. The Supplementary Material records two further members. One is a genuine \emph{ratio} witness, whose conditional width grows with $s$, and the other a unit-slope \emph{corner} witness with $c_-=c_+=1$ that realizes every edge index $\beta\ge0$. The full non-emptiness case split is recorded there as well.
\begin{convention}[Scope of this section]\label{conv:wit-scope}
The construction that follows verifies membership only. Each constructed pair satisfies A1--A6 and NA, and so lies in $\Mreg\subseteq\Mstar$. By itself it makes no identifiability claim. The confounding of distinct members that these pairs support, which is the core of the negative half, is assembled from them in Section~\ref{sec:pair}, whereas the shape half applies to them only through the inclusion, and neither conclusion is drawn here. The only quantitative by-product recorded is a consistency check, namely that the near-edge expansion of the constructed $Q$ reproduces the exponent $\alpha+\beta+2$ and, on the affine edge, the exact Beta-function constant of Lemma~\ref{lem:exponent}. This checks the Lemma against one instance.
\end{convention}
\subsection{The affine-random witness family}\label{sec:witness}
\begin{definition}[Affine-random witness]\label{def:witness}
Fix parameters
\[
\begin{gathered}
\alpha\in[0,\infty),\qquad L\in(0,\infty),\qquad c\in(0,\infty)\setminus\{1\},\qquad a_0\in[0,\infty),\\
\beta\in[0,\infty),\qquad c_W\in(0,\infty),\qquad 0\le w_-<w_+<\infty,\qquad s_{\max}\in(0,\infty).
\end{gathered}
\]
The \emph{affine-random witness} with parameters $(\alpha,L,c,a_0,\beta,c_W,w_-,w_+,s_{\max})$ is the pair $(F_S,K)$ specified by (a)--(d) below.
\begin{itemize}
\item[(a)] (\emph{Latent law.}) $F_S$ is a Borel probability law on $[0,\infty)$ with $\supp F_S\subseteq[0,s_{\max}]$ and $F_S(\{0\})=0$. It admits a window $s_0\in(0,\infty)$ on which $F_S$ has a density: there is a Borel function $f_S:(0,s_0)\to[0,\infty)$ with $F_S(B)=\int_B f_S(s)\,ds$ for every Borel $B\subseteq(0,s_0)$. This named version satisfies, pointwise (the reading of \eqref{eq:A1}),
\[
f_S(s)=L\,s^{\alpha}\bigl(1+o(1)\bigr)\qquad(s\downarrow0).
\]
\emph{Canonical latent law} (the special case used for the pair of Section~\ref{sec:pair}): given $\alpha$ and $s_{\max}$, take
\begin{equation}\label{eq:wit-canon}
F_S(s)=(s/s_{\max})^{\alpha+1}\ \ (0\le s\le s_{\max}),\qquad
f_S(s)=(\alpha+1)\,s_{\max}^{-(\alpha+1)}\,s^{\alpha}\ \ (0<s<s_{\max}),
\end{equation}
with window $s_0=s_{\max}$. Indeed, $\int_0^{s}(\alpha+1)s_{\max}^{-(\alpha+1)}t^{\alpha}\,dt=(s/s_{\max})^{\alpha+1}$ for $s\in[0,s_{\max}]$, so \eqref{eq:wit-canon} defines a Borel probability law concentrated on $[0,s_{\max}]$ with $F_S(\{0\})=0$ and with the stated density on the whole window $(0,s_{\max})$. The scale constant is \[ L=(\alpha+1)\,s_{\max}^{-(\alpha+1)}: \] the expansion $f_S(s)=L\,s^{\alpha}$ holds exactly on $(0,s_{\max})$ (the $o(1)$ vanishes identically). Moreover, $F_S$ is strictly increasing on $[0,s_{\max}]$.
\item[(b)] (\emph{Affine parameters.}) $c$ is the scale of the location map $g$ of step (d) and $a_0\ge0$ its intercept. The constraint $c\neq1$ is part of the definition. Note that $a_0$ is the \emph{affine intercept} and differs from the A5 apparent endpoint, which for this family is $a'=a_0+w_-=\ell(0+)$ (Lemma~\ref{lem:edges} and Proposition~\ref{prop:witness}). The two coincide only when $w_-=0$. (Here the prime in $a'$ is a fixed label for this edge intercept, not the member-$1$ index carried by $F_S'$, $K'$ elsewhere. Both members of the pair share this same apparent endpoint $a'$.)
\item[(c)] (\emph{Independent jitter.}) $W$ is a random variable with law $P_W$ on $\R$, independent of $S\sim F_S$ (a background space carrying such an independent pair exists by the product construction). In the reading of Convention~\ref{conv:W}, $w_-=\inf\supp P_W$ and $w_+=\sup\supp P_W$. Moreover, there is $\delta_W>0$ such that $P_W$ restricted to $[w_-,w_-+\delta_W]$ has a named density version $\rho_W$ with, pointwise,
\[
\rho_W(w)=c_W\,(w-w_-)^{\beta}\bigl(1+o(1)\bigr)\qquad(w\downarrow w_-).
\]
The lower-edge exponent of $W$ equals the class constant $\beta$ of A6, with constant $c_W$. This is the sole source of the index in A6 (condition \ref{it:mem-C1} of Proposition~\ref{prop:witness}).
\item[(d)] (\emph{Location map and kernel.}) Throughout the witness construction, $g$ denotes the location map and not the min-of-ratio severity functional that carries the same letter. On the latent support it is affine, and beyond that support it is extended by pure unit-slope translation:
\begin{equation}\label{eq:gdef}
g(s)=\begin{cases}
c\,s+a_0, & 0\le s\le s_{\max},\\[2pt]
c\,s_{\max}+a_0+(s-s_{\max}), & s>s_{\max},
\end{cases}
\end{equation}
and the kernel is
\[
K(\cdot\mid s):=\Law\bigl(g(s)+W\bigr),\qquad\text{i.e.}\qquad K(B\mid s)=\PP\bigl(g(s)+W\in B\bigr)\quad(s\ge0,\ B\text{ Borel}).
\]
Since the two branches of \eqref{eq:gdef} agree at $s_{\max}$ (both equal $c\,s_{\max}+a_0$), the map $g:[0,\infty)\to[0,\infty)$ is continuous and in particular Borel. $K$ is then a Markov kernel by Lemma~S8.2 of the Supplementary Material. By Lemma~S8.3 of the Supplementary Material, the observed law \eqref{eq:Q} is $Q=KF_S=\Law\bigl(g(S)+W\bigr)$. Because $\supp F_S\subseteq[0,s_{\max}]$ forces $S\le s_{\max}$ almost surely, $g(S)=cS+a_0$ almost surely and
\[
Q=\Law\bigl(cS+a_0+W\bigr).
\]
\end{itemize}
\end{definition}

The unit-slope branch of \eqref{eq:gdef} governs $K(\cdot\mid s)$ only for latent values $s>s_{\max}$, which form an $F_S$-null set. By Lemma~S8.3 of the Supplementary Material it therefore does not affect $Q$, and no $s\downarrow0$ quantity depends on it. Its purpose is global: beyond the latent support it holds the offset $m$ constant at the value $m(s_{\max})$, which keeps $\sup_{s\ge0}|m(s)|$ finite, in line with A3. By contrast, the raw affine offset $(c-1)s+a_0+\midW$ (with $\midW$ the jitter's support midpoint) diverges as $s\to\infty$ when $c\neq1$. In addition, the branch continues the edge $\ell=g+w_-$ with slope $1>0$, so that $\ell$ stays increasing on all of $[0,\infty)$, as the global-monotonicity clause of A5 requires. See Remark~S9.1 of the Supplementary Material and Remark~\ref{rem:bounded}.

\begin{lemma}[Closed-form edge functionals]\label{lem:edges}
Let $(F_S,K)$ be an affine-random witness with parameters $(\alpha,L,c,a_0,\beta,c_W,w_-,w_+,s_{\max})$, and let $\ell,u,w,m$ be the conditional edge functionals of Section~\ref{sec:setting}. Then:
\begin{enumerate}
\item\label{it:wit-g} $g$ is continuous and strictly increasing on $[0,\infty)$. Moreover, $g(s')-g(s)\ \ge\ \min(c,1)\,(s'-s)$ for all $0\le s\le s'$, and $a_0\le g(s)<\infty$ for every $s\ge0$.
\item\label{it:wit-supp} For every $s\ge0$ the conditional law $K(\cdot\mid s)$ is concentrated on the compact set $g(s)+[w_-,w_+]\subset[0,\infty)$, so in particular $K(\{+\infty\}\mid s)=0$. For each such $s$, the topological support of $K(\cdot\mid s)$, the same whether computed in $[0,\infty]$ or in $\R$, is the translate
\[
\supp K(\cdot\mid s)\ =\ g(s)+\supp P_W,
\]
a set whose infimum is $g(s)+w_-$ and whose supremum is $g(s)+w_+$. ($\supp P_W$ need not be all of $[w_-,w_+]$, since Convention~\ref{conv:W} prescribes only its infimum and supremum.)
\item\label{it:wit-edge} For every $s\ge0$,
\begin{equation}\label{eq:edgeforms}
\ell(s)=g(s)+w_-,\qquad u(s)=g(s)+w_+,\qquad w(s)=w_+-w_-\ \ (\text{constant in }s).
\end{equation}
\item\label{it:wit-m} For every $s\ge0$, $m(s)=g(s)-s+\midW$. Explicitly,
\begin{equation}\label{eq:mform}
m(s)=\begin{cases}
(c-1)\,s+a_0+\midW, & 0\le s\le s_{\max},\\[2pt]
(c-1)\,s_{\max}+a_0+\midW\ =\ m(s_{\max}), & s>s_{\max}.
\end{cases}
\end{equation}
\item\label{it:wit-mass} (\emph{$s$-independent edge mass.}) For every $s\ge0$ and every $h\ge0$,
\begin{equation}\label{eq:wit-mass}
K\bigl([\ell(s),\,\ell(s)+h]\,\big|\, s\bigr)\ =\ P_W\bigl([w_-,\,w_-+h]\bigr).
\end{equation}
\item\label{it:wit-asym} (\emph{Edge-mass asymptotics.}) For every $\eps\in(0,1)$ there is $\eta_W(\eps)\in(0,\delta_W]$ such that
\begin{equation}\label{eq:wit-band}
(1-\eps)\,\frac{c_W}{\beta+1}\,h^{\beta+1}\ \le\ P_W\bigl([w_-,\,w_-+h]\bigr)\ \le\ (1+\eps)\,\frac{c_W}{\beta+1}\,h^{\beta+1}
\qquad\text{for all }h\in(0,\eta_W(\eps)].
\end{equation}
Consequently, $P_W([w_-,w_-+h])=\dfrac{c_W}{\beta+1}\,h^{\beta+1}\bigl(1+o(1)\bigr)$ as $h\downarrow0$.
\end{enumerate}
\end{lemma}

All computations involving the witness reduce to Lemma~\ref{lem:edges}. Every axiom check in the membership verification (Section~S9 of the Supplementary Material) uses the kernel only through these closed forms.

\begin{proposition}[Membership of the witness in $\Mreg$]\label{prop:witness}
Let $(F_S,K)$ be an affine-random witness with parameters $(\alpha,L,c,a_0,\beta,c_W,w_-,w_+,s_{\max})$ and A1-window $(0,s_0)$ (Definition~\ref{def:witness}). Fix shared class constants $(\beta,\bar b,c_-,c_+,\tau,\bar s_0,s_R,s_A)$ satisfying the standing ordering of Section~\ref{sec:setting}, namely $0<c_-\le c_+<\infty$ as in Convention~\ref{conv:corridor}, $\bar b<\infty$, $\bar s_0>0$, $s_R,s_A\in(0,\bar s_0]$ and $0<\tau<s_R$. Suppose that the following compatibility conditions hold:
\begin{enumerate}
\renewcommand{\theenumi}{(C\arabic{enumi})}
\item\label{it:mem-C1} The class edge index of A6 equals the jitter lower-edge exponent $\beta$ of Definition~\ref{def:witness}(c), as the shared symbol already indicates.
\item\label{it:mem-C2} $c\in[c_-,c_+]$.
\item\label{it:mem-C3} $\bar b\ \ge\ |c-1|\,s_{\max}+a_0+\midW$.
\item\label{it:mem-C4} $s_R\le s_{\max}$ and $s_0\ge\bar s_0$. (For the canonical latent law, where $s_0=s_{\max}$, the first clause is automatic: the standing ordering $s_R\le\bar s_0$ and the second clause give $s_R\le\bar s_0\le s_0=s_{\max}$. The clause used by the A5 verification below is $s_R\le s_{\max}$.)
\item\label{it:mem-C5} $c\neq1$ (part of Definition~\ref{def:witness}(b), repeated because the NA verification uses it).
\end{enumerate}
Then $(F_S,K)$ satisfies A1--A4, NA (so $(F_S,K)\in\Mclass$) and A5, A6 with these shared constants. That is, $(F_S,K)\in\Mreg$. Moreover, the A5 apparent endpoint is
\[
a\ =\ \ell(0+)\ =\ a'\ :=\ a_0+w_-,
\]
and A6 holds with the constant profile $c_K(s)\equiv c_W/(\beta+1)\in(0,\infty)$ and with an $o(1)$ that does not depend on $s$.
\end{proposition}

Proposition~\ref{prop:witness} places the affine-random witness in $\Mreg$. That same witness is where the confounding pair of Theorem~\ref{thm:I1} lives, since the scale-conjugation pair of Section~\ref{sec:pair} is built from two of its members. Those members satisfy the \emph{full} edge axiom~A5 rather than merely the weakened A5$^-$, so they lie in the lean class $\Mstar$ as well, the class over which the shape theorem (Theorem~\ref{thm:I2}) holds. The pair is therefore a \emph{sharpness certificate} for that theorem, not a counterexample to it. It keeps $\alpha'=\alpha$ while moving $L$ and $p_\tau$, and this is the freedom that identifiability of $\alpha$ leaves open.
\subsection{Non-emptiness of the class}\label{ssec:wit-conclusion}

The affine-random witness verified above settles non-emptiness over the generic part of the parameter range. Fix class constants $\beta\ge0$, $0<c_-\le c_+<\infty$, $\bar b\in(0,\infty)$, and a member shape $\alpha\in[0,\infty)$. Whenever the corridor admits a scale $c\in[c_-,c_+]$ with $c\neq1$ (that is, unless $c_-=c_+=1$), Proposition~\ref{prop:witness} constructs an explicit \emph{non-additive} member of $\Mreg\subseteq\mathcal{M}_\star$ with the prescribed $\alpha$. Both $\mathcal{M}_\star$ and the regularity class $\Mreg$ inside it contain non-additive members, so Lemma~\ref{lem:exponent} and Theorem~\ref{thm:I2} are not vacuous.

Two questions remain before the non-emptiness account is complete. The first is the single corner $c_-=c_+=1$, where the scale constraint $c\neq1$ of the affine witness is unavailable and the edge slope is forced to equal one. The second is the exhibition of \emph{non-affine} members, whose conditional width grows with $s$. Such members realize $\beta\ge1$ through a ratio witness and any $\beta\ge0$ through a unit-slope construction, available whenever the corridor allows unit slope ($c_-\le1\le c_+$). Section~S12 of the Supplementary Material settles both questions by constructing and verifying a \emph{ratio} witness and a unit-slope \emph{corner} witness. It also records the exhaustive case split over $(\beta,c_-,c_+,\alpha)$ and checks the near-edge expansion of Lemma~\ref{lem:exponent} against each witness. With those members, the class is never confined to additive (translation) kernels, and whenever $\beta\ge1$ or $c_-\le1\le c_+$ it is not confined to global affine reparametrizations either.
\subsection{The scale-conjugation pair}\label{sec:pair}
\paragraph{Scale rather than shift.}
Axiom A1 leaves no latent location freedom. Every member has $F_S(\{0\})=0$ and a density that behaves near the endpoint $0$ as $f_S(s)=L\,s^{\alpha}(1+o(1))$ with $L>0$ \eqref{eq:A1}. A shifted latent law $\Law(S+\theta)$ with $\theta>0$ assigns mass $0$ to $(0,\theta)$, so any of its densities vanishes a.e.\ on a right-neighbourhood of $0$. This is incompatible with \eqref{eq:A1} for any $L\in(0,\infty)$, and the latent endpoint is therefore equal to $0$ across the entire class. Among location-scale deformations of the latent law, then, only the rescalings $S\mapsto\kappa S$, $\kappa\in(0,\infty)$ are compatible with A1, and these fix the endpoint while dilating the law around it. Apparent shifts do exist in the model, yet they enter only on the observed side, through the edge intercept $\ell(0)=a'=a_0+w_-$ of Lemma~\ref{lem:edges}, and never through the latent endpoint. The confound built below rescales the latent law and compensates through the kernel. Nothing is shifted, and A1 rules a shift out in any case.

\begin{definition}[The scale-conjugation pair]\label{def:pair}
\emph{Data.} Fix
\begin{gather*}
\alpha\in[0,\infty),\qquad s_{\max}\in(0,\infty),\qquad a_0\ge0,\\
c\in(0,\infty)\ \text{with}\ c\neq1,
\qquad
\kappa\in(0,\infty)\ \text{with}\ \kappa\neq1\ \text{and}\ \kappa\neq c,
\end{gather*}
together with a jitter law $P_W$ as in Convention~\ref{conv:W}: edges $0\le w_-<w_+<\infty$, lower-edge exponent $\beta\ge0$, and edge constant $c_W\in(0,\infty)$. Let $W\sim P_W$.

\emph{Member $0$.} Let $(F_S,K)$ be the affine-random witness of Definition~\ref{def:witness} with parameters $(\alpha,\allowbreak L,\allowbreak c,\allowbreak a_0,\allowbreak\beta,\allowbreak c_W,\allowbreak w_-,\allowbreak w_+,\allowbreak s_{\max})$ and the canonical latent law $F_S(s)=(s/s_{\max})^{\alpha+1}$ on $[0,s_{\max}]$, so that
\[
f_S(s)=(\alpha+1)\,s_{\max}^{-(\alpha+1)}\,s^{\alpha}\ \text{ on }(0,s_{\max}),
\qquad
L=(\alpha+1)\,s_{\max}^{-(\alpha+1)}.
\]
Its location map $g$ has slope $c$ on $[0,s_{\max}]$ and unit slope beyond, and its kernel is $K(\cdot\mid s)=\Law(g(s)+W)$. Write $Q:=KF_S$ for its observed law \eqref{eq:Q}.

\emph{Member $1$ (the $\kappa$-conjugate).} With $S\sim F_S$, set
\[
F_S'\;:=\;\Law(\kappa S),
\qquad
K'(\cdot\mid s)\;:=\;\Law\bigl(g'(s)+W\bigr)\quad(s\ge0),
\]
where
\begin{equation}\label{eq:pair-gprime}
g'(s)\;=\;
\begin{cases}
\dfrac{c}{\kappa}\,s+a_0, & 0\le s\le\kappa s_{\max},\\[6pt]
\dfrac{c}{\kappa}\,\kappa s_{\max}+a_0+(s-\kappa s_{\max}), & s>\kappa s_{\max},
\end{cases}
\end{equation}
i.e.\ $g'$ is the location map of step~(d) of Definition~\ref{def:witness} under the parameter substitution $(c,s_{\max})\mapsto(c/\kappa,\ \kappa s_{\max})$, with the same intercept $a_0$ and the same jitter $W$. Write $Q':=K'F_S'$.
\end{definition}

\begin{lemma}[Exact observational equivalence: $Q'=Q$]\label{lem:laws}
In the setting of Definition~\ref{def:pair},
\[
Q'\;=\;K'F_S'\;=\;KF_S\;=\;Q
\]
exactly, as Borel probability laws on $[0,\infty]$. Both laws charge only $[0,\infty)$, and indeed both assign full mass to the compact interval $[\,a_0+w_-,\ c\,s_{\max}+a_0+w_+\,]$. In particular, $(F_S,K)\sim(F_S',K')$ in the sense of Definition~\ref{def:ident}.
\end{lemma}

\begin{lemma}[The estimands under the pair]\label{lem:estimand}
In the setting of Definition~\ref{def:pair}, write $(\alpha,L)$ and $(\alpha',L')$ for the A1 shape/scale parameters of $F_S$ and $F_S'$ (single-valued by Lemma~\ref{lem:phiwell}). Then:
\begin{enumerate}
\item $\alpha'=\alpha$.
\item $L'=\kappa^{-(\alpha+1)}L$, and $L'\neq L$.
\item for every threshold $\tau$ with $0<\tau<s_{\max}\wedge\kappa s_{\max}$,
\[
p_\tau=F_S(\tau)=\Bigl(\frac{\tau}{s_{\max}}\Bigr)^{\alpha+1}\in(0,1),
\qquad
p'_\tau=F_{S'}(\tau)=\Bigl(\frac{\tau}{\kappa s_{\max}}\Bigr)^{\alpha+1}=\kappa^{-(\alpha+1)}\,p_\tau\in(0,1),
\]
and $p'_\tau\neq p_\tau$.
\end{enumerate}
Consequently $L'/L=p'_\tau/p_\tau=\kappa^{-(\alpha+1)}\neq1$, and the two estimand gaps
\begin{align}
\Delta_L&:=|L-L'|=L\,\bigl|1-\kappa^{-(\alpha+1)}\bigr|,\label{eq:pair-DL}\\
\Delta_p&:=|p_\tau-p'_\tau|=p_\tau\,\bigl|1-\kappa^{-(\alpha+1)}\bigr|\label{eq:pair-Dp}
\end{align}
are strictly positive.
\end{lemma}

Note the order of quantifiers. The two members are built first, and only afterwards are the shared class constants chosen, wide enough to cover both. Specifically, the corridor $[c_-,c_+]$ is widened to contain the two edge slopes $c$ and $c/\kappa$, and the offset bound $\bar b$ is taken large enough to dominate the two offsets. The floors $\bar s_0=s_R=s_A$ are lowered to sit inside both A1-windows. This is the form required by Theorem~\ref{thm:I1} (``there exist admissible shared class constants \emph{and} two members of the resulting class''), so no smallness of $\kappa$ is needed. Any $\kappa\in(0,\infty)$ works, except for the two measure-zero exclusions $\kappa=1$ (which would make the pair trivial) and $\kappa=c$ (which would make the kernel of member $1$ additive, violating the compatibility condition (C5)). Having shown that the class is non-empty and constructed the confounding pair, we now compare the result with prior work.

\section{Related work}\label{sec:rel}

We prove identifiability of the latent lower-tail shape index $\alpha$ through a \emph{non-additive} kernel $K(\cdot\mid s)$ whose edge $\ell(s)$ \emph{moves} with the signal and whose relative spread $w(s)/s$ blows up at the endpoint, from the \emph{marginal} $Q=KF_S$ alone, with no location anchor. Identifiability here means constancy of $\alpha$ across the fibers of $(F_S,K)\mapsto KF_S$ (the classes of observationally equivalent pairs). It is a property of the model map, not of any estimator. Adjacent work differs from ours along three axes. The pushforward there is \emph{additive} rather than non-additive, and the object is a \emph{location} rather than a \emph{shape}. Its relative spread \emph{vanishes} at the endpoint as well, whereas ours diverges. The contaminated marginal does not expose $\alpha$ as its own edge exponent. Over $\mathcal{M}_\star$ the exponent it \emph{does} expose, $\alpha+\beta+2$, pins $\alpha$ regardless. We write $a:=\ell(0+)=\inf\supp Q$ and recall the near-edge law $Q([a,a+h])\asymp L\,h^{\alpha+\beta+2}$, where $\beta\ge0$ is the conditional edge index of $K$ shared across the class.

\subsection{Additive-noise deconvolution}\label{ssec:rel-deconv}

Recovering a latent law from additively contaminated $Y=X+\eps$ rests entirely on the error characteristic function: $\varphi_X=\varphi_Y/\varphi_\eps$ is identifiable precisely when the zero set of $\varphi_\eps$ has empty interior, at a rate set by its decay \citep{Fan1991,Meister2009}. The deconvolution kernel estimator inverts it \citep{StefanskiCarroll1990}. Real zeros, as for the uniform and other common compactly supported errors, force ridge regularization \citep{HallMeister2007}. This presupposes an \emph{additive, location-invariant} error, a fixed $\mu$ with $K(\cdot\mid s)=\mu(\cdot-s)$. Deconvolution identifies a \emph{density} and needs the error law known through $\varphi_\eps$. Under axiom~NA no such $\mu$ exists: there is no $\varphi_\eps$ to divide out, and the condition ``$\varphi_\eps$ has no interval of zeros'' is not violated but cannot be stated.

\subsection{Endpoint, boundary, and frontier estimation under error}\label{ssec:rel-endpoint}

The closest classical analogue is endpoint estimation under additive error \citep{GT2004}. There the target is the \emph{location} $\theta$ of $X$ (edge density $f(x)\ge L(\theta-x)^{\alpha}$), with $X$ observed as $Y=X+\eps$.
Bounded-support error $[-r,r]$ ($r$ known, edge $(r-x)^{\beta}$) gives minimax rate $n^{-1/(\alpha+\beta+2)}$, anchored through the single known edge $r$. Convolution tail-arithmetic combines indices $\alpha,\beta$ into a marginal edge $\alpha+\beta+1$ with Beta-function lead $B(\beta+1,\alpha+1)$. This rate exponent is the classical counterpart of ours. The $\alpha+\beta+2$ of \cite{GT2004}, which \emph{there} governs the \emph{location} $\theta$ under additive error, equals the exponent we read off $Q([a,a+h])$, with the Beta-function constant appearing in both. Yet \emph{here} the same number governs the \emph{shape} $\alpha$ through a non-additive, moving-edge kernel. In the additive setting the extra $+1$ over $\alpha+\beta+1$ integrates a fixed-support error against a \emph{static} endpoint. In ours, $\alpha+\beta+2=(\alpha+1)+(\beta+1)$ splits into integrating $s^{\alpha}$ across the $O(h)$ band whose \emph{moving} edge $\ell(s)$ reaches $[a,a+h]$, plus the edge-mass profile $h^{\beta+1}$ of axiom~A6. Two variants swap the anchor. An \emph{unknown} normal scale \citep{Leng2019,KSV2015} still needs the error form known and zero-mean, while \emph{unknown but symmetric} additive error \citep{FSV2020} trades the known edge for symmetry. Symmetry, like the known edge, is then one admissible anchor among several, not indispensable.

\subsection{Extreme-quantile estimation under erroneous or approximate observations}\label{ssec:rel-evt}

A recent strand asks how imperfectly observed values affect tail-shape inference. \cite{Pere2024ApproxErrExtreme} show the Hill and extreme-quantile estimators of heavy-tailed $X$ keep the \emph{same} asymptotically normal limit once the approximation error is negligible relative to the tail quantile scale at the $\sqrt{k}$ rate, $k$ the effective tail sample size. For a bounded multiplicative relative perturbation of the order statistics, \cite{Pere2025} bound its effect on the moment estimator of the extreme-value index. In \cite{Morozov2025}, extreme quantiles of a latent heterogeneity law are recovered through a noisy proxy with vanishing error and \emph{no} parametric error form, via a tail-equivalence condition under which the noisy order statistics inherit the latent limit. We take that tail-equivalence condition as our comparison point for EVT under latent heterogeneity. It is sharp, and once it fails, some dependence structure makes the noisy proxy miss the latent limit.

\subsection{Relation to the present paper}\label{ssec:rel-gap}

No prior result identifies a bounded extreme-value tail \emph{shape} index through a \emph{non-additive} kernel with a moving power-law edge and endpoint-diverging relative spread when only the marginal is observed and no location anchor is available. Deconvolution and its heteroscedastic variants keep the contamination additive and independent of the latent value, so the error scale may vary across observations but never with the severity itself \citep[Sect.~2.1]{DelaigleMeister2008}. In our regime the relative spread $w(s)/s$ instead diverges at the endpoint. The additive endpoint/frontier/boundary literature recovers a \emph{location} object under an anchor, compounding $\alpha$ to $\alpha+\beta+1$ or obscuring it. In \cite{GT2004} only that compound is visible at the marginal edge, and their endpoint estimator (anchored by the known error-support radius) adapts to the unknown $\alpha$, which is a nuisance there and the target here. The erroneous-observation strand is estimation, not identifiability: its contamination is additive-and-vanishing \citep{Pere2024ApproxErrExtreme,Morozov2025} or multiplicative \citep{Pere2025}, so the relative error must \emph{vanish}, the opposite of our setting. Each rests on side information our marginal-only scheme forbids. The first needs refinable approximations of each latent value and the second a growing per-unit sample size $T$, whereas the third needs $\sqrt n$-consistent location/scatter from elliptical structure. Furthermore, the tail-shape parameter is recovered by Hill or moment estimators \citep{Pere2024ApproxErrExtreme,Pere2025}, while \cite{Morozov2025} imposes a known lower bound on the extreme-value index in its rate conditions. All three keep the noisy extremes on the normalizing scale of the latent ones, so the two limits agree, whereas here the kernel shifts the near-edge exponent away from the latent one. In the applied traffic-conflict study the question is error propagation, not identifiability of a contamination kernel. The theorem here closes this gap: over $\mathcal{M}_\star$, $\alpha$ is point-identified from $Q=KF_S$, which fixes both the apparent endpoint $a=\inf\supp Q$ and the near-edge exponent $\alpha+\beta+2$, the shared class constant $\beta$ apportioning it to the latent shape.
On the negative half, a standard device is pushed to a degenerate extreme. The Le~Cam two-point method \citep{LeCam1986,Tsybakov2009} lower-bounds minimax risk by the testing affinity of two hypotheses. This affinity is a number between zero and one measuring their overlap, and it equals one exactly when no test can tell them apart. In the literature on endpoint estimation under error the method is used in the usual way: two \emph{nearby} laws, separated just enough that no test resolves them at the working sample size, which yields a \emph{rate} \citep{GT2004}. Our two witnesses are not nearby. They induce the \emph{identical} sampling law $Q^{\otimes n}$ at every $n$, so their affinity is exactly $1$ and the bound degenerates: the risk floor $\tfrac12\Delta$ of Lemma~\ref{lem:bridge} carries no power of $n$ and never decays. The confounding pair therefore certifies non-identification rather than slow estimation.
\section{Discussion and scope}\label{scope:sec}\label{sec:scope}
\begin{remark}[Scope of the identification result]\label{scope:main}
For every admissible choice of the class constants, Theorem~\ref{thm:I2} identifies $\alpha$ over $\Mstar$, and Corollary~\ref{cor:Mreg} carries it to $\Mreg$. This remark fixes the scope of that statement.

\emph{(i) Shape only, and provably so.} Theorem~\ref{thm:I2} identifies the latent lower-tail shape index $\alpha$, and nothing else about the tail. The scale $L$ of A1 and the fixed-scale exceedance $p_\tau$ are addressed by the other half of the paper, with the opposite conclusion: they are \emph{not} identifiable over $\Mreg$ without a further anchor (Theorem~\ref{thm:I1}). That statement is existential in the sense proved there. Namely, for suitable admissible shared class constants there exist two members of $\Mreg$ with \emph{exactly} equal observed laws and equal $\alpha$, while both $L$ and $p_\tau$ differ by the factor $\kappa^{-(\alpha+1)}\neq1$. The near-edge expansion pins the near-endpoint \emph{exponent} $\alpha+\beta+2$ and shows that the lead coefficient is proportional to $L$ (Remark~\ref{lean:affine}), but it does \emph{not} pin $L$ itself. Indeed, both the proportionality constant $\varkappa_\pm$ in \eqref{eq:kappas} and, in the affine case, the sharp constant $\bar c_K\,\Beta(\alpha+1,\beta+2)\,c_\ell^{-(\alpha+1)}$ depend on the kernel and are not extracted from $Q$ by this argument.

\emph{(ii) Identifiability is relative to the class.} What is identified is $\alpha$ \emph{given the class}, in particular given the shared edge index $\beta$. The observed law delivers a single number, the near-endpoint exponent $\alpha+\beta+2$ of \eqref{eq:exponent}. Splitting that number unambiguously between the latent shape $\alpha$ and the kernel edge index $\beta$ rests on the class assumption that $\beta$ is common across members. If $\beta$ were allowed to vary across members, the exponent alone would not separate the two. No claim is made (and none is needed) that $\alpha$ and $\beta$ are individually recoverable in that wider setting.

\emph{(iii) Identification, not estimation.} The shape theorem states that $\alpha$ is constant on the fibers of the map $(F_S,K)\mapsto Q$ restricted to the class: observationally equivalent members share $\alpha$. It is not an estimator and carries no sampling guarantee or rate. A natural plug-in estimator and a heuristic consistency argument are sketched in the Supplementary Material, yet a full consistency or rate theorem is left to future work.
\end{remark}
\begin{remark}[The $\kappa=1$ collapse on unbounded support]\label{rem:bounded}
Suppose, counterfactually, that the affine action $s\mapsto cs+a_0$ held on an \emph{unbounded} latent support, with no $s_{\max}$ and no off-support extension, the kernel being $\Law(cs+a_0+W)$ for all $s\ge0$. The offset computation of Lemma~\ref{lem:edges} would then read $m(s)=(c-1)s+a_0+\midW$ for \emph{all} $s\ge0$, an affine function of slope $c-1$. For $c\neq1$ it satisfies $|m(s)|\to\infty$ as $s\to\infty$, so that $\sup_{s\ge0}|m(s)|=\infty$ exceeds every finite $\bar b$ and axiom A3 fails. A3 would then force $c=1$ on member $0$ and $c/\kappa=1$ on member $1$. Together the two constraints give $\kappa=1$, i.e.\ only the \emph{trivial} pair $F_S'=F_S$, $K'=K$ survives. (Even that degenerate ``pair'' would leave the class, since $c=1$ turns the kernel into the translate family $K(\cdot\mid s)=\mu(\cdot-s)$ with $\mu=\Law(a_0+W)$, which violates NA. This is the content of condition (C5) of Proposition~\ref{prop:witness}.) Hence scale conjugation collapses on unbounded support, and no genuine confound exists there.

Our proofs use the bounded support together with the unit-slope extension in two places only, both inside the membership verification (Proposition~\ref{prop:witness}, invoked for both members through Proposition~S11.3 of the Supplementary Material):
\begin{enumerate}
\item \emph{the A3 verification, through (C3):} by \eqref{eq:mform}, the offset $m$ is affine with slope $c-1$ only up to the break at $s_{\max}$ and is \emph{constant} beyond it. Hence $\sup_{s\ge0}|m(s)|\le|c-1|\,s_{\max}+a_0+\midW$, which is the finite quantity that (C3) asks $\bar b$ to dominate. For member $1$ the same bound holds with $(c/\kappa,\kappa s_{\max})$ in place of $(c,s_{\max})$. Boundedness of the offset for a slope $c\neq1$ is available \emph{only} because the affine regime stops at the break.
\item \emph{the A5 global monotonicity:} beyond the break the location map, and with it the edge $\ell=g+w_-$ of \eqref{eq:edgeforms}, continues with slope $1>0$. It follows that $\ell$ is increasing on all of $[0,\infty)$, and no distant conditional law returns to the apparent endpoint. This is the global half of A5, which the two-point increment condition on $[0,s_R]$ does not cover.
\end{enumerate}
Both members of Definition~\ref{def:pair} carry their own break ($s_{\max}$ and $\kappa s_{\max}$ respectively) and their own unit-slope tail. That definition and Lemma~S11.1 of the Supplementary Material are built accordingly. It follows that the scale/exceedance confound of Theorem~\ref{thm:I1} is a \emph{bounded-support} phenomenon. A latent severity bound together with an off-support translation regime is what makes a non-unit edge slope, and with it the conjugation $c\mapsto c/\kappa$, admissible.
\end{remark}

\begin{remark}[What carries the confound, and what would close it]\label{rem:anchor}
The entire confound is carried by one degree of freedom: the kernel's edge slope, moved from $c$ to $c/\kappa$ to compensate the latent rescaling $S\mapsto\kappa S$. More precisely, what differs between the two kernels is the \emph{observable edge map} $\ell$. On the respective corridors, member $0$ has $\ell(s)=cs+a'$ and member $1$ has $\ell'(s)=(c/\kappa)s+a'$ (Lemma~\ref{lem:edges}). The slope moves, while the intercept $a'=a_0+w_-$ stays and the observed law $Q$ remains fixed (Lemma~\ref{lem:laws}). An \emph{anchor} axiom pinning the edge map $\ell$ (its slope and intercept) across the class would therefore close the degree of freedom that our pair uses. Pinning $\ell$ is the edge clause of Anchor~C (Assumption~\ref{ass:C}), and full repair also needs C's fixed-scale injectivity clause (C-b). Granted both clauses, the sufficiency direction of Conjecture~\ref{conj:dichotomy} is immediate (Remark~\ref{rem:forthcoming}). What remains open here is discharging (C-b) for non-location families, along with the necessity direction of the conjecture. The mechanism is the location-and-scale version, through a non-additive kernel, of the additive-location endpoint confound of Goldenshluger and Tsybakov \citeyearpar{GT2004}. In that setting a fixed error law shifts an endpoint, whereas here a dilated latent law and a contracted edge slope trade off against each other, leaving the endpoint, and everything else observable, unchanged.
\end{remark}
\begin{remark}[Scope of the non-identifiability result]\label{rem:scope}
The result has four scope boundaries. (i) \emph{Not the shape}: the exhibited pair keeps $\alpha'=\alpha$ (Lemma~\ref{lem:estimand}). Identifiability of $\alpha$ over $\Mreg$ is the positive half, Theorem~\ref{thm:I2} (with Corollary~\ref{cor:Mreg}), which the pair supports rather than overturns. (ii) \emph{Not the anchored horn} (the anchored branch of the identifiability dichotomy): Remark~\ref{rem:anchor} only locates the degree of freedom that an anchor axiom would close. Nothing is proved about anchored sub-fibers beyond the sufficiency argument of Remark~\ref{rem:forthcoming}, nor is (C-b) discharged for non-location families. (iii) \emph{Not estimator construction}: the statistical conclusions are the minimax floors and consistency failures of Lemma~\ref{lem:bridge} and Corollary~S10.2 of the Supplementary Material. No estimator or rate is given, and no positive consistency claim is made. (iv) \emph{Non-emptiness is supplied}: Proposition~S11.3 of the Supplementary Material exhibits two members of $\Mreg$ for the constructed shared constants, so the class quantified over is non-empty by construction. The result itself needs no non-emptiness beyond these constants. For the general record see Section~\ref{ssec:wit-conclusion}, whose affine witness covers every shared-constant regime except the corner $c_-=c_+=1$, where non-emptiness is settled by the unit-slope corner witness of Section~S12 of the Supplementary Material. The four constructions behind the result are Proposition~\ref{prop:witness} with Proposition~S11.3 of the Supplementary Material, Lemma~\ref{lem:laws}, Lemma~\ref{lem:estimand}, and the proof of Theorem~\ref{thm:I1} via Lemma~\ref{lem:bridge}.
\end{remark}
\section*{Acknowledgments}
The authors used OpenAI ChatGPT to help edit the manuscript to improve the grammar and readability, assist verifying mathematical correctness and proofs, and check and debug analysis and .tex code. The authors reviewed and verified all output and take full responsibility for the content of this article.
\section*{Data and code availability}
No datasets were generated or analysed during this work. Two numerical self-check
scripts accompany this submission as ancillary files. The first,
\texttt{i1\_selfcheck.py}, certifies the exact-rational scale-conjugation instance of
Theorem~\ref{thm:I1}, which is Section~S14.A, ``Numerical self-check''. The second,
\texttt{i2\_selfcheck.py}, corroborates the affine-case constant of
Lemma~\ref{lem:exponent}, which is Section~S15. Both are deterministic, write no
files, and print a PASS/FAIL summary of $12/12$ assertions and $4/4$ checks
respectively. \texttt{i1\_selfcheck.py} fixes its random seeds, whereas
\texttt{i2\_selfcheck.py} uses no randomness. Both run on Python~3.13 with
\texttt{numpy}, \texttt{scipy}, \texttt{mpmath} and \texttt{sympy}.
\phantomsection

\clearpage
\setcounter{section}{0}
\setcounter{equation}{0}
\renewcommand{\thesection}{S\arabic{section}}
\renewcommand{\theequation}{S\arabic{equation}}
\renewcommand{\theHsection}{S\arabic{section}}
\renewcommand{\theHequation}{S\arabic{equation}}
\renewcommand{\thesubsection}{\thesection.\Alph{subsection}}
\renewcommand{\theHsubsection}{S\arabic{section}.\Alph{subsection}}
\phantomsection\addcontentsline{toc}{section}{Supplementary Material}
\begin{center}
{\LARGE\bf Supplementary Material}\\[6pt]
{\large to ``Shape without scale: an identifiability dichotomy for a bounded tail\\ observed through a non-additive measurement kernel''}
\end{center}
\medskip
\noindent The material below is numbered with the prefix `S': Sections S1--S15, results as Lemma~S1.1, Proposition~S2.1, \dots, equations as (S1), (S2), \dots. Unprefixed pointers (Theorem~4.2, equation~(3), Section~3) refer to the main text above.
\medskip

\section{Support preliminaries}\label{sec:supp}

We record elementary facts about supports on the compactified half-line. Recall that $[0,\infty]$ with the order topology is a compact, second-countable metrizable space (homeomorphic to $[0,1]$ via $x\mapsto x/(1+x)$). For a Borel probability $\mu$ on such a space, the support $\supp\mu$ (the set of points all of whose open neighbourhoods have positive $\mu$-mass) is closed and carries full mass. Indeed, the complement of $\supp\mu$ is the union of all open $\mu$-null sets, which by second countability (Lindel\"of) is a countable union of null sets, hence null \citep{Kallenberg2002}.

\begin{lemma}[Edge facts]\label{lem:supp}
Fix $s\ge0$ and write $\mu:=K(\cdot\mid s)$, $E:=\supp\mu\cap[0,\infty)$ (nonempty by the setup). Then:
\begin{enumerate}
\item\label{it:attain} $\ell(s)=\inf E\in E$. In particular, $\ell(s)\in\supp\mu$.
\item\label{it:nomass} $\mu\bigl([0,t]\bigr)=0$ for every $t<\ell(s)$, and $\mu\bigl([0,\ell(s))\bigr)=0$.
\item\label{it:posmass} $\mu\bigl([\ell(s),\ell(s)+r)\bigr)>0$ for every $r>0$.
\item\label{it:noatom} If $s\in(0,s_R]$ and A6 holds, then $\mu(\{\ell(s)\})=0$.
\end{enumerate}
\end{lemma}

\begin{proof}
\ref{it:attain} $E$ is nonempty and bounded below, so $\ell(s)=\inf E<\infty$, and there are $x_n\in E$ with $x_n\to\ell(s)$. Since $\supp\mu$ is closed in $[0,\infty]$ and $\ell(s)\in[0,\infty)$, the limit lies in $\supp\mu\cap[0,\infty)=E$.

\ref{it:nomass} Let $t<\ell(s)$. Every $x\in[0,t]$ is a finite point with $x<\inf E$, hence $x\notin E$. Such an $x$ is not $+\infty$ either, so it does not lie in $\supp\mu$. Thus $[0,t]\cap\supp\mu=\varnothing$, and since $\supp\mu$ carries full mass, $\mu([0,t])\le\mu\bigl([0,\infty]\setminus\supp\mu\bigr)=0$. Then $\mu([0,\ell(s)))=\lim_{n}\mu([0,\ell(s)-\tfrac1n])=0$ by continuity from below.

\ref{it:posmass} The set $(\ell(s)-r,\ell(s)+r)\cap[0,\infty]$ is an open neighbourhood of the support point $\ell(s)$, so it has positive mass. By \ref{it:nomass}, the part below $\ell(s)$ is null, whence $\mu([\ell(s),\ell(s)+r))>0$.

\ref{it:noatom} By \eqref{eq:A6u} with any fixed $\eps\in(0,1)$, $\mu(\{\ell(s)\})\le\mu\bigl([\ell(s),\ell(s)+\delta]\bigr)\le(1+\eps)c_K(s)\delta^{\beta+1}$ for all $\delta\in(0,\eta(\eps)]$. Now let $\delta\downarrow0$ and use $c_K(s)<\infty$.
\end{proof}

\begin{lemma}[Geometry of the edge under A5$^-$]\label{lem:A5geom}
Under A5$^-$:
\begin{enumerate}
\item\label{it:corridor} (\emph{corridor}) $a+c_-s\;\le\;\ell(s)\;\le\;a+c_+s$ for all $s\in[0,s_R]$.
\item\label{it:far} (\emph{far field}) $\ell(s)\;\ge\;a+c_-s_R$ for all $s\ge s_R$.
\item\label{it:above} $\ell(s)\;\ge\;a+c_-\min(s,s_R)\;>\;a$ for all $s>0$.
\end{enumerate}
\end{lemma}

\begin{proof}
\ref{it:corridor} is the one-point corridor of A5$^-$ (Convention~\ref{lean:A5minus}, \eqref{lean:corridor}) on $(0,s_R]$, together with the trivial case $s=0$, where $\ell(0)-a=0$ by the convention $\ell(0)=a$. For \ref{it:far}, let $s\ge s_R$. Monotonicity of $\ell$ on $[0,\infty)$ gives $\ell(s)\ge\ell(s_R)$, and $\ell(s_R)\ge a+c_-s_R$ by \ref{it:corridor}. \ref{it:above} combines the two, using $c_->0$ (Convention~\ref{conv:cpm}).
\end{proof}

Lemma~\ref{lem:A5geom} deliberately leaves one subtlety untouched. In A5$^-$, the value $\ell(0)$ is a \emph{notational convention} ($\ell(0):=a=\ell(0+)$), not a constraint on the kernel at $s=0$. The actual conditional law $K(\cdot\mid 0)$, and hence the true edge $\inf(\supp K(\cdot\mid0)\cap[0,\infty))$, is unconstrained by A5$^-$. This is harmless because A1 makes $\{0\}$ an $F_S$-null set, so $K(\cdot\mid0)$ never enters the mixture \eqref{eq:Q}. See the proof of Proposition~\ref{prop:endpoint}.

\section{The apparent endpoint: \texorpdfstring{$a=\inf\supp Q$}{a = inf supp Q}}\label{sec:endpoint}

\begin{proposition}[Apparent endpoint: uses A1 and A5$^-$ only]\label{prop:endpoint}
Let $(F_S,K)$ satisfy A1 and A5$^-$ (together with the setup conventions of Section~\ref{ssec:objects}). Then
\begin{enumerate}
\item\label{it:null} $Q([0,a])=0$.
\item\label{it:pos} $Q([a,a+h])>0$ for every $h>0$.
\item\label{it:inf} consequently $a=\min\supp Q=\inf\supp Q$.
\end{enumerate}
\end{proposition}

\begin{proof}
\ref{it:null} For every $s>0$ we have $\ell(s)>a$ by Lemma~\ref{lem:A5geom}\ref{it:above}, hence $K([0,a]\mid s)=0$ by Lemma~\ref{lem:supp}\ref{it:nomass} (applied with $t=a<\ell(s)$). Thus the measurable integrand $s\mapsto K([0,a]\mid s)$ vanishes on $(0,\infty)$, i.e.\ $F_S$-almost everywhere, because $F_S(\{0\})=0$ by A1. It follows that $Q([0,a])=\int K([0,a]\mid s)F_S(ds)=0$. (No information about $K(\cdot\mid0)$ is used, and none is available.)

\ref{it:pos} Fix $h>0$. Since $Q([a,a+h])$ is nondecreasing in $h$, we may and do assume $h<c_+\min(s_R,s_0)$. By A1 there is $s_1\in(0,s_0)$ with $f_S(s)\ge\tfrac12Ls^{\alpha}>0$ for all $s\in(0,s_1)$. Set
\[
J:=\bigl(0,\ \min(h/c_+,\,s_1)\bigr),\qquad F_S(J)=\int_J f_S(s)\,ds\;\ge\;\int_J\tfrac12Ls^\alpha ds\;>\;0 .
\]
For every $s\in J$ we have $s\le s_R$ and, by Lemma~\ref{lem:A5geom}\ref{it:corridor}, $a<\ell(s)\le a+c_+s<a+h$, so $r(s):=a+h-\ell(s)>0$ and, by Lemma~\ref{lem:supp}\ref{it:posmass},
\[
K\bigl([a,a+h]\mid s\bigr)\;\ge\;K\bigl([\ell(s),\ell(s)+r(s))\mid s\bigr)\;>\;0,
\]
the inclusion $[\ell(s),\ell(s)+r(s))=[\ell(s),a+h)\subseteq[a,a+h]$ being valid because $\ell(s)>a$. The integrand $s\mapsto K([a,a+h]\mid s)$ is therefore strictly positive at every point of $J$, and $F_S(J)>0$. Suppose $Q([a,a+h])=\int_{[0,\infty)}K([a,a+h]\mid s)F_S(ds)$ vanished. Then the integrand would vanish $F_S$-a.e., in particular $F_S$-a.e.\ on $J$, which is impossible. Hence $Q([a,a+h])>0$.

\ref{it:inf} By \ref{it:null}, $[0,a)$ is an open subset of $[0,\infty]$ with $Q([0,a))=0$, so no point of $[0,a)$ lies in $\supp Q$, i.e.\ $\inf\supp Q\ge a$. In the other direction, every open neighbourhood of $a$ in $[0,\infty]$ contains a set $[a,a+h)$ with $h>0$. Such a set has positive $Q$-mass, since $Q([a,a+h))\ge Q([a,a+h'])>0$ for any $h'\in(0,h)$ by the inclusion $[a,a+h']\subseteq[a,a+h)$ and \ref{it:pos}. Hence $a\in\supp Q$ and \ref{it:inf} follows.
\end{proof}

Proposition~\ref{prop:endpoint} establishes the apparent-endpoint identity $a=\inf\supp Q$ from A1 and A5$^-$ alone. A6 was not used. The identification of $a$ from $Q$, which underpins the ``determined by $Q$'' clause of Lemma~\ref{lem:exponent} and the deduction of Theorem~\ref{thm:I2}, is therefore available on all of $\Mstar$ independently of the near-edge expansion.

\section{Localization}\label{sec:local}

For the rest of the proof we set, for $h>0$ and $s\in(0,s_R]$,
\begin{equation}\label{eq:deltah}
\delta_h(s):=\bigl(h-(\ell(s)-a)\bigr)_+\;\in\;[0,h].
\end{equation}
We call $\delta_h(s)$ the \emph{depth}: how far the observation window $[a,a+h]$ reaches past the lower edge $\ell(s)$ of the conditional support at latent value $s$. It is positive while $\ell(s)$ lies below $a+h$, and $0$ once $\ell(s)$ has passed $a+h$.

\begin{lemma}[Exact localization and window identity]\label{lem:local}
Let $(F_S,K)$ satisfy A1, A5$^-$, A6 and let $0<h<c_-s_R$. Then:
\begin{enumerate}
\item\label{it:loc} $\displaystyle Q\bigl([a,a+h]\bigr)=\int_{(0,\,h/c_-]}K\bigl([a,a+h]\mid s\bigr)\,F_S(ds)$.
\item\label{it:window} for every $s\in(0,s_R]$,
\[
K\bigl([a,a+h]\mid s\bigr)=K\bigl([\ell(s),\,\ell(s)+\delta_h(s)]\mid s\bigr),
\]
both sides being $0$ when $\delta_h(s)=0$ (intervals are read with the convention $[x,x]=\{x\}$).
\item\label{it:sandwich} for every $s\in(0,s_R]$,
\[
\bigl(h-c_+s\bigr)_+\;\le\;\delta_h(s)\;\le\;\bigl(h-c_-s\bigr)_+ . \] In particular $\delta_h(s)=0$ for $s>h/c_-$, and $\delta_h(s)>0$ for $s<h/c_+$.
\end{enumerate}
If moreover $h<c_-\min(s_R,s_0)$, then \ref{it:loc} can be written as a Lebesgue integral,
\begin{equation}\label{eq:lebesgue}
Q\bigl([a,a+h]\bigr)=\int_0^{h/c_-}K\bigl([a,a+h]\mid s\bigr)\,f_S(s)\,ds .
\end{equation}
\end{lemma}

\begin{proof}
\ref{it:loc} Split $\int_{[0,\infty)}=\int_{\{0\}}+\int_{(0,h/c_-]}+\int_{(h/c_-,\,s_R]}+\int_{(s_R,\infty)}$. The $\{0\}$ term vanishes since $F_S(\{0\})=0$ (A1). For $s\in(h/c_-,s_R]$, Lemma~\ref{lem:A5geom}\ref{it:corridor} gives $\ell(s)\ge a+c_-s>a+h$, so $[a,a+h]\subseteq[0,\ell(s))$ and $K([a,a+h]\mid s)=0$ by Lemma~\ref{lem:supp}\ref{it:nomass}. When $s>s_R$, Lemma~\ref{lem:A5geom}\ref{it:far} gives $\ell(s)\ge a+c_-s_R>a+h$ (here $h<c_-s_R$ enters), and the same conclusion follows. Note that $h/c_-<s_R$, so the four ranges are exhaustive and only $(0,h/c_-]$ survives.

\ref{it:window} Fix $s\in(0,s_R]$ and recall that $\ell(s)>a$ (Lemma~\ref{lem:A5geom}\ref{it:above}). Suppose first that $\ell(s)>a+h$, so that $\delta_h(s)=0$. Then $[a,a+h]\subseteq[0,\ell(s))$, so the left side is $0$, while the right side is $K(\{\ell(s)\}\mid s)=0$ by Lemma~\ref{lem:supp}\ref{it:noatom}. If instead $\ell(s)\le a+h$, so that $\ell(s)+\delta_h(s)=a+h$, decompose $[a,a+h]=[a,\ell(s))\,\dot\cup\,[\ell(s),a+h]$. The first piece lies in $[0,\ell(s))$ and is null by Lemma~\ref{lem:supp}\ref{it:nomass}, so
\[
K([a,a+h]\mid s)=K\bigl([\ell(s),a+h]\mid s\bigr)=K\bigl([\ell(s),\ell(s)+\delta_h(s)]\mid s\bigr),
\]
and if $\delta_h(s)=0$ (the case $\ell(s)=a+h$) this equals $K(\{\ell(s)\}\mid s)=0$.

\ref{it:sandwich} This is immediate from the corridor $c_-s\le\ell(s)-a\le c_+s$ on $(0,s_R]$ and the monotonicity of $x\mapsto(h-x)_+$.

Finally, we prove \eqref{eq:lebesgue}. If $h<c_-\min(s_R,s_0)$, then $(0,h/c_-]\subseteq(0,s_0)$, where $F_S(ds)=f_S(s)\,ds$ by A1.
\end{proof}

\section{The two-sided bound}\label{sec:bound}

We first record the elementary integral identity that produces both the exponent and the Beta constant.

\begin{proposition}[Beta identity]\label{prop:beta}
For every $c>0$, $h>0$, $\alpha\ge0$, $\beta\ge0$,
\begin{equation}\label{eq:betaid}
\int_0^{h/c}\bigl(h-cs\bigr)^{\beta+1}s^{\alpha}\,ds
\;=\;c^{-(\alpha+1)}\,\Beta(\alpha+1,\beta+2)\,h^{\alpha+\beta+2}.
\end{equation}
\end{proposition}

\begin{proof}
Substitute $s=(h/c)t$, $t\in[0,1]$, $ds=(h/c)\,dt$, $h-cs=h(1-t)$:
\[
\int_0^{h/c}(h-cs)^{\beta+1}s^\alpha\,ds
=h^{\beta+1}\Bigl(\frac hc\Bigr)^{\alpha+1}\int_0^1(1-t)^{\beta+1}t^{\alpha}\,dt
=c^{-(\alpha+1)}h^{\alpha+\beta+2}\,\Beta(\alpha+1,\beta+2),
\]
since $\int_0^1t^{x-1}(1-t)^{y-1}dt=\Beta(x,y)$ with $x=\alpha+1\ge1>0$ and $y=\beta+2\ge2>0$. The integral is proper here, because the integrand is bounded and continuous on $[0,1]$.
\end{proof}

\begin{theorem}[Two-sided near-edge bound]\label{thm:twosided}
Let $(F_S,K)\in\Mstar$. Define $\varkappa_\pm$ by \eqref{eq:kappas} (explicit positive constants built from $\clo,\chigh,c_\pm,\alpha,\beta$). Then there is $h_0>0$ (depending on the member) such that
\[
\varkappa_-\,L\,h^{\alpha+\beta+2}\;\le\;Q\bigl([a,a+h]\bigr)\;\le\;\varkappa_+\,L\,h^{\alpha+\beta+2}
\qquad\text{for all }h\in(0,h_0).
\]
\end{theorem}

\begin{proof}
Throughout the proof, fix
\[
\eps:=\tfrac15,\qquad\text{so that }(1+\eps)^3=1.728\le2\text{ and }(1-\eps)^3=0.512\ge\tfrac12 .
\]
We collect the four smallness thresholds that the axioms provide for this $\eps$:
\begin{itemize}
\item[(T1)] (\emph{A1 window}) there is $s_1\in(0,s_0]$ such that $(1-\eps)Ls^{\alpha}\le f_S(s)\le(1+\eps)Ls^{\alpha}$ for all $s\in(0,s_1)$. This is \eqref{eq:A1}.
\item[(T2)] (\emph{$c_K$ window}) there is $s_2\in(0,s_R]$ such that
\[
(1-\eps)\,\clo\;\le\;c_K(s)\;\le\;(1+\eps)\,\chigh\qquad\text{for all }s\in(0,s_2).
\]
Indeed, by definition of $\liminf/\limsup$ there is $s_2$ with $c_K(s)>\clo-\eps\,\clo$ and $c_K(s)<\chigh+\eps\,\chigh$ on $(0,s_2)$.
\item[(T3)] (\emph{A6 window}) $\eta(\eps)>0$ as in Convention~\ref{conv:A6}.
\item[(T4)] (\emph{geometry windows}) the corridor holds on $(0,s_R]$ and the density representation on $(0,s_0)$.
\end{itemize}
Set
\begin{equation}\label{eq:h0}
h_0:=\min\bigl\{c_-s_R,\;c_-s_0,\;c_-s_1,\;c_-s_2,\;\eta(\eps)\bigr\}>0,
\end{equation}
and let $0<h<h_0$. Then $h<c_-s_R$, so Lemma~\ref{lem:local} applies. Moreover, the contributing range satisfies
\begin{equation}\label{eq:range}
(0,h/c_-]\;\subseteq\;(0,\min(s_0,s_1,s_2,s_R)),
\end{equation}
and every depth satisfies $\delta_h(s)\le h<\eta(\eps)$. By Lemma~\ref{lem:local}\ref{it:loc}--\ref{it:window} and \eqref{eq:lebesgue},
\begin{equation}\label{eq:master}
Q\bigl([a,a+h]\bigr)=\int_0^{h/c_-}K\bigl([\ell(s),\ell(s)+\delta_h(s)]\mid s\bigr)\,f_S(s)\,ds .
\end{equation}

\emph{Upper bound.} Fix $s\in(0,h/c_-]$ and assume that $\delta_h(s)>0$. Then \eqref{eq:A6u} applies, since $s\in(0,s_R]$ and $0<\delta_h(s)\le h<\eta(\eps)$, and so does (T2), since $s<s_2$ by \eqref{eq:range}. Together with Lemma~\ref{lem:local}\ref{it:sandwich}, these bounds yield
\[
K\bigl([\ell(s),\ell(s)+\delta_h(s)]\mid s\bigr)
\;\le\;(1+\eps)\,c_K(s)\,\delta_h(s)^{\beta+1}
\;\le\;(1+\eps)^2\,\chigh\,\bigl(h-c_-s\bigr)^{\beta+1}.
\]
In the remaining case $\delta_h(s)=0$ the left side is $0$ and the inequality is trivial. With (T1), which is valid since $s<s_1$,
\[
Q\bigl([a,a+h]\bigr)\;\le\;(1+\eps)^3\,\chigh\,L\int_0^{h/c_-}(h-c_-s)^{\beta+1}s^{\alpha}\,ds
\;=\;(1+\eps)^3\,\chigh\,L\,c_-^{-(\alpha+1)}\Beta(\alpha+1,\beta+2)\,h^{\alpha+\beta+2},
\]
by Proposition~\ref{prop:beta} with $c=c_-$. Since $(1+\eps)^3\le2$, the upper bound holds with $\varkappa_+$ of \eqref{eq:kappas}.

\emph{Lower bound.} The integrand in \eqref{eq:master} is nonnegative, so we may restrict the integral to $(0,h/c_+)\subseteq(0,h/c_-]$. Fix $s\in(0,h/c_+)$. Then $\delta_h(s)\ge h-c_+s>0$ by Lemma~\ref{lem:local}\ref{it:sandwich}, so \eqref{eq:A6u} and (T2), along with monotonicity of $\delta\mapsto\delta^{\beta+1}$, give
\[
K\bigl([\ell(s),\ell(s)+\delta_h(s)]\mid s\bigr)
\;\ge\;(1-\eps)\,c_K(s)\,\delta_h(s)^{\beta+1}
\;\ge\;(1-\eps)^2\,\clo\,\bigl(h-c_+s\bigr)^{\beta+1}.
\]
By (T1) and Proposition~\ref{prop:beta} applied with $c=c_+$ (the integrals over $(0,h/c_+)$ and $[0,h/c_+]$ coincide),
\[
Q\bigl([a,a+h]\bigr)\;\ge\;(1-\eps)^3\,\clo\,L\int_0^{h/c_+}(h-c_+s)^{\beta+1}s^{\alpha}\,ds
\;=\;(1-\eps)^3\,\clo\,L\,c_+^{-(\alpha+1)}\Beta(\alpha+1,\beta+2)\,h^{\alpha+\beta+2},
\]
and since $(1-\eps)^3\ge\tfrac12$, the lower bound holds with $\varkappa_-$ of \eqref{eq:kappas}.
\end{proof}

\begin{remark}[Technical notes on the two-sided bound]\label{rem:uniformity}
\emph{Uniformity of A6, and why $\liminf/\limsup$ suffices.}\label{rem:localize} The two displays above invoke \eqref{eq:A6u} at the depth $\delta=\delta_h(s)$, which \emph{depends on $s$} and ranges over all of $(0,h)$ as $s$ runs through the contributing window $(0,h/c_-]$. That window itself moves with $h$. The single threshold $\eta(\eps)$, valid simultaneously for all $s\in(0,s_R]$, lets one $h$-condition ($h<\eta(\eps)$) validate every invocation at once. A merely pointwise A6, granting each $s$ a private $\eta(\eps,s)>0$, fails. The proof would then need $h\le\inf\{\eta(\eps,s):0<s\le h/c_-\}$, a \emph{self-referential} requirement, and a pointwise hypothesis cannot prevent $\eta(\eps,s)\downarrow0$ as $s\downarrow0$ fast enough that the infimum vanishes for every window. Restricting uniformity to \emph{compact subsets} of $(0,s_R]$ fails for the same reason, the contributing window accumulating at $s=0$. Uniformity up to the endpoint, $\eta$ free of $s$, is therefore the load-bearing form of A6, not a convenience. Only the $\liminf/\limsup$ of $c_K$ at $0$ is needed. Indeed, the corridor (A5) localizes the contributing range to $(0,h/c_-]$ \emph{before} any bound on $c_K$ is invoked, so the sharp two-sided window (T2) near $0$ suffices. No positive lower bound on $c_K$ away from $0$ is required, and $c_K$ may decay to $0$ along $s\uparrow s_R$.

\emph{Edge curvature.}\label{rem:curvature} Within the corridor, $\ell$ need not be affine. Its shape enters the estimate \emph{only} through the depth $\delta_h(s)=(h-(\ell(s)-a))_+$, which Lemma~\ref{lem:local}\ref{it:sandwich} sandwiches between the exactly affine envelopes $(h-c_+s)_+$ and $(h-c_-s)_+$. By Proposition~\ref{prop:beta} both integrate to the \emph{same} power $h^{\alpha+\beta+2}$, with constants $c_+^{-(\alpha+1)}$ and $c_-^{-(\alpha+1)}$, so curvature perturbs only the constant, never the exponent. The spread $\varkappa_+/\varkappa_-=4\,(\chigh/\clo)(c_+/c_-)^{\alpha+1}$ widens with the corridor, and the bookkeeping factor $4$ is pure slack, driven to $1$ as $\eps\downarrow0$. When $c_-=c_+$ and $c_K$ has a limit the two envelopes coincide, and the bound collapses to the asymptotic equality of Proposition~\ref{prop:affine}.

\emph{Measurability.} Every integral above pairs the measurable function $s\mapsto K([a,a+h]\mid s)$ with $F_S$, split along fixed intervals and bounded above and below by the \emph{explicit} continuous envelopes $s\mapsto\mathrm{const}\cdot(h-c_\mp s)_+^{\beta+1}s^{\alpha}$. Here $\ell$ and $c_K$ enter only through pointwise inequalities between measurable functions, never as integrands. No measurability of $s\mapsto\ell(s)$ or $s\mapsto c_K(s)$ is therefore required.
\end{remark}

\section{The near-endpoint exponent}\label{sec:exponent}

\begin{corollary}[Exponent, determined by $Q$]\label{cor:exponent}
Let $(F_S,K)\in\Mstar$. Then
\[
\lim_{h\downarrow0}\frac{\log Q([a,a+h])}{\log h}=\alpha+\beta+2 .
\]
Moreover the left-hand side is a functional of $Q$ alone: $a=\inf\supp Q$ by Proposition~\ref{prop:endpoint}, and $h\mapsto Q([a,a+h])$ is then determined by $Q$.
\end{corollary}

\begin{proof}
Let $h\in(0,\min(h_0,1))$ with $h_0$ from Theorem~\ref{thm:twosided}. Then $Q([a,a+h])\ge\varkappa_-Lh^{\alpha+\beta+2}>0$, so the logarithm is defined, and $\log h<0$. Set
\[
R(h):=\log\frac{Q([a,a+h])}{h^{\alpha+\beta+2}},\qquad\text{so that}\qquad
\frac{\log Q([a,a+h])}{\log h}=(\alpha+\beta+2)+\frac{R(h)}{\log h}.
\]
Taking logarithms in \eqref{eq:twosided} confines $R(h)$ to the \emph{fixed} compact interval $[\log(\varkappa_-L),\,\log(\varkappa_+L)]$ for all $h\in(0,\min(h_0,1))$. Hence
\[
\Bigl|\frac{R(h)}{\log h}\Bigr|\;\le\;\frac{\max\bigl(|\log(\varkappa_-L)|,|\log(\varkappa_+L)|\bigr)}{|\log h|}\;\xrightarrow[h\downarrow0]{}\;0,
\]
since $|\log h|\to\infty$, and the limit \eqref{eq:exponent} follows. (Equivalently one may divide the logarithmic form of \eqref{eq:twosided} by $\log h<0$, which reverses the inequalities, and then squeeze. The form used above makes that sign bookkeeping unnecessary.)
\end{proof}

Note what the corollary does \emph{and does not} extract from $Q$: the exponent $\alpha+\beta+2$ is a $Q$-functional, whereas the constants $\varkappa_\pm L$ are not pinned by $Q$ (they are kernel-dependent). This is consistent with the scope stated in the main text (Section~\ref{scope:sec}). No identifiability claim about $L$ is made here. The recovery of $L$ is the opposite-direction question, settled by Theorem~\ref{thm:I1} of the main text, which shows that $L$ is \emph{not} identifiable over $\Mreg$ without an anchor.

\section{The affine sharpening}\label{sec:affine}

\begin{proposition}[Exactly affine edge]\label{prop:affine}
Let $(F_S,K)\in\Mstar$ and assume in addition that $\ell(s)=a+c_\ell s$ for all $s\in[0,s_R]$ and that $c_K(s)\to\bar c_K\in(0,\infty)$ as $s\downarrow0$. Then $c_\ell\in[c_-,c_+]$, and
\[
Q\bigl([a,a+h]\bigr)=\bar c_K\,\Beta(\alpha+1,\beta+2)\,L\,c_\ell^{-(\alpha+1)}\,h^{\alpha+\beta+2}\bigl(1+o(1)\bigr)\qquad(h\downarrow0).
\]
\end{proposition}

\begin{proof}
Since the member lies in $\Mstar$, A5$^-$ holds. The corridor (Lemma~\ref{lem:A5geom}\ref{it:corridor}), applied to $\ell(s)=a+c_\ell s$, gives $c_-s\le c_\ell s\le c_+s$ for $s\in(0,s_R]$, i.e.\ $c_\ell\in[c_-,c_+]$, and in particular $c_\ell>0$. Also $c_K(s)\to\bar c_K$ implies $\clo=\chigh=\bar c_K$.

Let $\eps\in(0,\tfrac15]$ be arbitrary. Take the thresholds (T1)--(T3) of the proof of Theorem~\ref{thm:twosided} for this $\eps$, with (T2) now reading $(1-\eps)\bar c_K\le c_K(s)\le(1+\eps)\bar c_K$ on $(0,s_2)$ (definition of the limit). Set $h_0(\eps):=\min\{c_-s_R,c_-s_0,c_-s_1,c_-s_2,\eta(\eps)\}$ as in \eqref{eq:h0}, and let $0<h<h_0(\eps)$. The identity \eqref{eq:master} holds verbatim, and now the depth is \emph{exact}:
\[
\delta_h(s)=\bigl(h-c_\ell s\bigr)_+\qquad(s\in(0,s_R]),
\]
vanishing precisely for $s\ge h/c_\ell$. By Lemma~\ref{lem:local}\ref{it:window} the integrand of \eqref{eq:master} therefore vanishes for $s\ge h/c_\ell$, so the integral effectively runs over $(0,h/c_\ell]$. Since $h/c_\ell\le h/c_-$, this range sits inside the range of \eqref{eq:master} and inside all windows \eqref{eq:range}. For $s\in(0,h/c_\ell)$ the two A6 bounds \eqref{eq:A6u} and (T2) give
\[
(1-\eps)^2\,\bar c_K\,(h-c_\ell s)^{\beta+1}\;\le\;K\bigl([\ell(s),\ell(s)+\delta_h(s)]\mid s\bigr)\;\le\;(1+\eps)^2\,\bar c_K\,(h-c_\ell s)^{\beta+1},
\]
and combining with (T1) and Proposition~\ref{prop:beta} with $c=c_\ell$,
\[
(1-\eps)^3\;\le\;\frac{Q([a,a+h])}{\bar c_K\,\Beta(\alpha+1,\beta+2)\,L\,c_\ell^{-(\alpha+1)}\,h^{\alpha+\beta+2}}\;\le\;(1+\eps)^3 .
\]
Since $\eps\in(0,\tfrac15]$ was arbitrary and $h_0(\eps)>0$ for each such $\eps$, the ratio tends to $1$ as $h\downarrow0$, which is the assertion.
\end{proof}

\begin{proof}[Proof of Lemma~\ref{lem:exponent}]
Combine Theorem~\ref{thm:twosided} (two-sided bound, constants \eqref{eq:kappas}), Proposition~\ref{prop:endpoint} ($a=\inf\supp Q$), Corollary~\ref{cor:exponent} (exponent, determined by $Q$), and Proposition~\ref{prop:affine} (affine case).
\end{proof}

\begin{remark}[The choice of $\mathcal{M}_\star$, and the role of A2--A4 and NA]\label{lean:freeupgrade}
Because $\Mreg\subseteq\mathcal{M}_\star$, an identifiability statement quantified over $\mathcal{M}_\star$ is \emph{strictly stronger} than the same statement over $\Mreg$. Over $\mathcal{M}_\star$ the conclusion is asserted for every regular member \emph{and} for the many additional members that drop A2--A4, NA, and the interior increments of A5. The axiom budget of Section~\ref{sec:budget} certifies this upgrade at no cost. Indeed, the proof of Sections~\ref{sec:supp}--\ref{sec:I2} consumes \emph{only} A1, A5$^-$, A6 (and the setup well-posedness convention), so it transfers verbatim from $\Mreg$ to $\mathcal{M}_\star$.

This does not demote A2--A4 and NA to irrelevance. They are the \emph{modeling} axioms that make the identifiability question non-trivial, and they earn their place in the motivation rather than in the proof. Non-additivity (NA) is what removes the classical tool. Since $K(\cdot\mid s)$ is not a translate $\mu(\cdot-s)$ of a fixed error law, there is no error characteristic function to divide out and Fourier deconvolution does not apply (cf.\ the deconvolution literature \citealp{Fan1991,StefanskiCarroll1990,Meister2009}, whose identifying device is that error characteristic function). Tail-amplified heteroscedasticity (A4), $w(s)/s\to\infty$, places the worst contamination exactly where the estimand lives, so the question cannot be reduced to a vanishing-noise limit. The bounded-support/no-escape clause (A2) and the bias clause (A3) are what locate the problem in the regime where the near-edge expansion is the right instrument. That regime is characterized by a moving, power-law conditional edge near a shifted apparent endpoint $a$. A member of $\mathcal{M}_\star$ that violated these would be a mathematically admissible but modeling-vacuous object. That the conclusion survives for it is a bonus, not the point. The content of the result is that, on the modeling-relevant subclass $\Mreg$, where deconvolution is unavailable, the shape $\alpha$ is nonetheless pinned.
\end{remark}
\section{Proof of Theorem~\ref{thm:I2} (shape identifiability)}\label{sec:I2}

\begin{proof}[Proof of Theorem~\ref{thm:I2}]
First, $\phi=\alpha$ is a well-defined functional on $\Mstar$. Suppose densities $f_S$ and $\tilde f_S$ of $F_S$ on $(0,s_0)$ satisfy \eqref{eq:A1} with $(L,\alpha)$ and with $(\tilde L,\tilde\alpha)$ respectively, $L,\tilde L\in(0,\infty)$. Two versions of the density agree Lebesgue-a.e., so we may pick a sequence $s_n\downarrow0$ along which $f_S(s_n)=\tilde f_S(s_n)$. Then
\[
L\,s_n^{\alpha}\bigl(1+o(1)\bigr)\;=\;f_S(s_n)\;=\;\tilde f_S(s_n)\;=\;\tilde L\,s_n^{\tilde\alpha}\bigl(1+o(1)\bigr),
\qquad\text{hence}\qquad
s_n^{\,\alpha-\tilde\alpha}\;\longrightarrow\;\frac{\tilde L}{L}\in(0,\infty).
\]
If $\alpha>\tilde\alpha$ the left side tends to $0$, and if $\alpha<\tilde\alpha$ it tends to $+\infty$. Both cases contradict $\tilde L/L\in(0,\infty)$. Hence $\alpha=\tilde\alpha$ (and then also $L=\tilde L$): the exponent read off through \eqref{eq:A1} is an unambiguous functional of $F_S$.

Now let $(F_S,K),(F'_S,K')\in\Mstar$ be observationally equivalent: $Q:=KF_S=K'F'_S=:Q'$. Both members satisfy A1 and A5$^-$, so Proposition~\ref{prop:endpoint} applies to each:
\[
a=\inf\supp Q=\inf\supp Q',
\]
i.e.\ both apparent endpoints are the same functional of the same law. Write $a^*$ for this common value and $G(h):=Q([a^*,a^*+h])=Q'([a^*,a^*+h])$, the same function of $h$ for both members. By Corollary~\ref{cor:exponent} applied to each member,
\[
\alpha+\beta+2=\lim_{h\downarrow0}\frac{\log G(h)}{\log h}=\alpha'+\beta'+2 .
\]
Every member of $\Mstar$ satisfies A6 with the same $\beta$, by the definition of the class. Both members carry the \emph{same} edge index $\beta'=\beta$. Hence $\alpha=\alpha'$.
\end{proof}

\begin{remark}[The role of the shared $\beta$]
Identifiability of $\alpha$ is relative to the class $\Mstar$ (hence to $\Mreg$, by Proposition~\ref{lean:inclusion}), and the definition of $\Mstar$ fixes $\beta$ across members. The observed exponent $\alpha+\beta+2$ is a single number extracted from $Q$. Apportioning it unambiguously between the latent shape ($\alpha$) and the kernel edge index ($\beta$) rests on the class assumption. No claim is made (and none is needed here) that $\alpha$ and $\beta$ are separately recoverable when $\beta$ is allowed to vary across members.
\end{remark}

\begin{remark}[Explicit recovery map]
The proof exhibits the recovery map in closed form: on $\Mstar$,
\[
\alpha\;=\;\lim_{h\downarrow0}\frac{\log Q\bigl([\,\inf\supp Q,\ \inf\supp Q+h\,]\bigr)}{\log h}\;-\;\beta\;-\;2 ,
\]
a functional of the observed law $Q$ and the class constant $\beta$ alone, with the same value on each observational-equivalence class. (This is an identification formula, not an estimator, in keeping with the scope of Section~\ref{scope:sec}.)
\end{remark}

\section{Measure-theoretic preliminaries}\label{sec:prelim}

This section proves the measure-theoretic facts consumed by the constructions of Sections~\ref{sec:witness}--\ref{sec:pair}. The behaviour of laws on $\R$ under translation, dilation and convolution is recorded in Lemma~\ref{lem:push}. Lemma~\ref{lem:kernelmeas} establishes the Markov-kernel property of a location family driven by an independent jitter, and Lemma~\ref{lem:mixture} identifies the resulting mixture \eqref{eq:Q} with the law of $g(S)+W$. We use two standard notions throughout. First, for a Borel map $T:E\to E'$ between Borel subsets of metrizable spaces and a Borel probability measure $P$ on $E$, the \emph{pushforward} $T_{\#}P$ is the Borel probability on $E'$ defined by $(T_{\#}P)(B):=P(T^{-1}B)$. If $X\sim P$ then $T(X)\sim T_{\#}P$, directly from the definition of the law. Second, the \emph{support} $\supp\mu$ of a Borel measure $\mu$ on a second-countable metrizable space is the set of points all of whose open neighbourhoods have positive $\mu$-mass. It is closed, since its complement, the union of all open $\mu$-null sets, is open. By second countability that complement is a \emph{countable} union of open null sets (Lindel\"of), hence null, so $\supp\mu$ carries full mass \citep{Kallenberg2002}.

\begin{lemma}[Translation, dilation, convolution]\label{lem:push}
Let $P$ be a Borel probability measure on $\R$ and $X\sim P$. For $t\in\R$ and $\lambda>0$ write $T_t(x):=t+x$ and $D_\lambda(x):=\lambda x$, homeomorphisms of $\R$.
\begin{enumerate}
\item\label{it:pre-supp} (\emph{supports}) $\supp\Law(t+X)=t+\supp P$ and $\supp\Law(\lambda X)=\lambda\,\supp P$.
\item\label{it:pre-cov} (\emph{CDF and density change of variables}) Let $S\sim F_S$ be a $[0,\infty)$-valued random variable and $\kappa>0$. Then
\[
F_{\kappa S}(x):=\PP(\kappa S\le x)=F_S(x/\kappa)\quad\text{for every }x\in\R,
\qquad\text{and}\qquad
\PP(\kappa S=0)=\PP(S=0).
\]
If moreover $F_S$ restricted to $(0,s_0)$ has a density $f_S$ (i.e.\ $F_S(B)=\int_Bf_S(u)\,du$ for every Borel $B\subseteq(0,s_0)$), then $\Law(\kappa S)$ restricted to $(0,\kappa s_0)$ has the density
\[
x\;\longmapsto\;\kappa^{-1}f_S(x/\kappa),\qquad x\in(0,\kappa s_0).
\]
\item\label{it:pre-comp} (\emph{composition of dilations}) For $c,\kappa>0$, $\ \Law\bigl(\tfrac{c}{\kappa}\,(\kappa S)\bigr)=\Law(cS)$.
\item\label{it:pre-conv} (\emph{convolution}) If $X$ and $Y$ are independent real random variables, then $\Law(X+Y)=\Law(X)*\Law(Y)$, where $(\mu*\nu)(B):=\int_{\R}\mu(B-y)\,\nu(dy)$ for Borel $B\subseteq\R$. In particular $\Law(X+Y)$ depends on $(X,Y)$ only through the pair of marginal laws $(\Law(X),\Law(Y))$, and for deterministic $t\in\R$, $\Law(t+X)=\delta_t*\Law(X)$.
\end{enumerate}
\end{lemma}

\begin{proof}
\ref{it:pre-supp} We first show that for any homeomorphism $T$ of $\R$, $\supp(T_{\#}P)=T(\supp P)$. Indeed, fix $x\in\R$. As $U$ ranges over the open neighbourhoods of $x$, the set $T^{-1}U$ ranges exactly over the open neighbourhoods of $T^{-1}x$. This holds because $T^{-1}$ is continuous and open, and every open $V\ni T^{-1}x$ arises as $V=T^{-1}(TV)$ with $TV$ open $\ni x$. Hence every open $U\ni x$ has $(T_{\#}P)(U)=P(T^{-1}U)>0$ if and only if every open $V\ni T^{-1}x$ has $P(V)>0$. That is, $x\in\supp(T_{\#}P)\iff T^{-1}x\in\supp P\iff x\in T(\supp P)$. Now $\Law(t+X)=(T_t)_{\#}P$ and $\Law(\lambda X)=(D_\lambda)_{\#}P$ (pushforward describes the law of a transformed variable), and $T_t,D_\lambda$ are homeomorphisms with inverses $T_{-t},D_{1/\lambda}$. Therefore $\supp\Law(t+X)=T_t(\supp P)=t+\supp P$ and $\supp\Law(\lambda X)=D_\lambda(\supp P)=\lambda\,\supp P$.

\ref{it:pre-cov} Since $\kappa>0$, for every $x\in\R$ the events $\{\kappa S\le x\}$ and $\{S\le x/\kappa\}$ coincide, whence $F_{\kappa S}(x)=F_S(x/\kappa)$. Likewise $\{\kappa S=0\}=\{S=0\}$, so $\PP(\kappa S=0)=\PP(S=0)$. For the density claim we use the substitution identity
\begin{equation}\label{eq:pre-sub}
\int_{\R}g(u)\,du\;=\;\kappa^{-1}\int_{\R}g(x/\kappa)\,dx
\qquad\text{for every Borel }g:\R\to[0,\infty],
\end{equation}
which holds for $g=\mathbf 1_A$ ($A$ Borel) by the dilation covariance of Lebesgue measure, $\mathrm{Leb}(\kappa A)=\kappa\,\mathrm{Leb}(A)$ (note $\mathbf 1_A(x/\kappa)=\mathbf 1_{\kappa A}(x)$). Linearity then extends \eqref{eq:pre-sub} to nonnegative simple $g$, and monotone convergence to all nonnegative Borel $g$. Let $B\subseteq(0,\kappa s_0)$ be Borel. Then $\kappa^{-1}B:=D_{1/\kappa}(B)\subseteq(0,s_0)$ is Borel, and
\[
\Law(\kappa S)(B)=\PP(\kappa S\in B)=\PP\bigl(S\in\kappa^{-1}B\bigr)=F_S\bigl(\kappa^{-1}B\bigr)
=\int_{\R}\mathbf 1_{\kappa^{-1}B}(u)\,f_S(u)\,du,
\]
using the density hypothesis on the Borel set $\kappa^{-1}B\subseteq(0,s_0)$ (with $f_S$ extended by $0$ off $(0,s_0)$). Applying \eqref{eq:pre-sub} to $g:=\mathbf 1_{\kappa^{-1}B}\,f_S$ and noting $\mathbf 1_{\kappa^{-1}B}(x/\kappa)=\mathbf 1_B(x)$,
\[
\Law(\kappa S)(B)=\kappa^{-1}\int_{\R}\mathbf 1_B(x)\,f_S(x/\kappa)\,dx=\int_B\kappa^{-1}f_S(x/\kappa)\,dx,
\]
which is the asserted density statement (the map $x\mapsto\kappa^{-1}f_S(x/\kappa)$ is Borel, as a constant multiple of the composition of $f_S$ with a continuous map).

\ref{it:pre-comp} For Borel maps $T,U:\R\to\R$ and Borel $B$, $\bigl((T\circ U)_{\#}P\bigr)(B)=P\bigl(U^{-1}(T^{-1}B)\bigr)=\bigl(T_{\#}(U_{\#}P)\bigr)(B)$, i.e.\ pushforwards compose. Since $D_{c/\kappa}\circ D_{\kappa}=D_c$ as maps ($\tfrac{c}{\kappa}(\kappa x)=cx$),
\[
\Law\bigl(\tfrac{c}{\kappa}(\kappa S)\bigr)=(D_{c/\kappa})_{\#}\bigl((D_\kappa)_{\#}F_S\bigr)=(D_{c/\kappa}\circ D_\kappa)_{\#}F_S=(D_c)_{\#}F_S=\Law(cS).
\]

\ref{it:pre-conv} Write $P_X:=\Law(X)$ and $P_Y:=\Law(Y)$, and fix Borel $B\subseteq\R$. Independence means $\Law(X,Y)=P_X\otimes P_Y$ on $\R^2$ \citep{Kallenberg2002}. The set $A:=\{(x,y):x+y\in B\}$ is Borel in $\R^2$, being the preimage of $B$ under the continuous addition map. Second countability gives $\mathcal{B}(\R^2)=\mathcal{B}(\R)\otimes\mathcal{B}(\R)$, so $A$ is product-measurable. By Tonelli's theorem for the finite product measure $P_X\otimes P_Y$,
\begin{multline*}
\Law(X+Y)(B)=\PP\bigl((X,Y)\in A\bigr)=(P_X\otimes P_Y)(A)\\
=\int_{\R}\Bigl[\int_{\R}\mathbf 1_B(x+y)\,P_X(dx)\Bigr]P_Y(dy)
=\int_{\R}P_X(B-y)\,P_Y(dy),
\end{multline*}
since $\{x:x+y\in B\}=B-y$. The right-hand side is $(P_X*P_Y)(B)$ and is a functional of $(P_X,P_Y)$ alone, proving both claims of \ref{it:pre-conv} for general $(X,Y)$. For the special case, a deterministic $Y\equiv t$ is independent of every $X$ and has law $\delta_t$, so $\Law(t+X)=P_X*\delta_t$. Moreover $(P_X*\delta_t)(B)=\int P_X(B-y)\,\delta_t(dy)=P_X(B-t)=(\delta_t*P_X)(B)$. The last equality holds because $\int\delta_t(B-x)P_X(dx)=\int\mathbf 1_{B-t}(x)P_X(dx)=P_X(B-t)$ as well (convolution is symmetric in its arguments).
\end{proof}

The next lemma builds the kernels of the affine-random witness family: location families $s\mapsto\Law(g(s)+W)$ driven by a single jitter. Throughout, for Borel $B\subseteq[0,\infty]$ we write $B_0:=B\cap[0,\infty)$, a Borel subset of $\R$. For real $x$ we read $\mathbf 1_B(x):=\mathbf 1_{B_0}(x)$ (a finite point $x$ belongs to $B$ iff it belongs to $B_0$).

\begin{lemma}[Location families driven by a jitter are Markov kernels]\label{lem:kernelmeas}
Let $g:[0,\infty)\to[0,\infty)$ be Borel, and let $W$ be a real random variable whose law $P_W$ satisfies $\supp P_W\subseteq[w_-,w_+]$ with $0\le w_-\le w_+<\infty$ (in particular the jitter of Convention~\ref{conv:W} qualifies). For $s\in[0,\infty)$ and Borel $B\subseteq[0,\infty]$ set
\begin{equation}\label{eq:pre-Kdef}
K(B\mid s):=\int_{\R}\mathbf 1_B\bigl(g(s)+w\bigr)\,P_W(dw)\;=\;P_W\bigl(B_0-g(s)\bigr).
\end{equation}
Then $K$ is a Markov kernel from $[0,\infty)$ to $[0,\infty]$:
\begin{enumerate}
\item\label{it:pre-prob} for each $s\ge0$, $K(\cdot\mid s)$ is a Borel probability measure on $[0,\infty]$, and it coincides with $\Law(g(s)+W)$: it is concentrated on the compact set $g(s)+\supp P_W\subseteq[g(s)+w_-,\,g(s)+w_+]\subseteq[0,\infty)$. In particular $K(\{+\infty\}\mid s)=0$ for \emph{every} $s\ge0$.
\item\label{it:pre-meas} for each Borel $B\subseteq[0,\infty]$, the map $s\mapsto K(B\mid s)$ is Borel measurable on $[0,\infty)$.
\end{enumerate}
\end{lemma}

\begin{proof}
First, \eqref{eq:pre-Kdef} is well defined and its two expressions agree. For real $x$ we have $\mathbf 1_B(x)=\mathbf 1_{B_0}(x)$ by the reading above, so $\mathbf 1_B(g(s)+w)=\mathbf 1_{B_0-g(s)}(w)$. As a translate of the Borel set $B_0$, the set $B_0-g(s)$ is a Borel subset of $\R$. The integral is thus the $P_W$-mass of a Borel set.

\ref{it:pre-prob} Fix $s\ge0$. \emph{Measure property.} $K(\varnothing\mid s)=0$, and for pairwise disjoint Borel $B_1,B_2,\dots\subseteq[0,\infty]$ the indicators satisfy $\sum_{j\le N}\mathbf 1_{B_j}(g(s)+w)\uparrow\mathbf 1_{\bigcup_jB_j}(g(s)+w)$ pointwise in $w$ as $N\to\infty$ (the sets $(B_j)_0-g(s)$ are pairwise disjoint). Monotone convergence then gives $\sum_jK(B_j\mid s)=K(\bigcup_jB_j\mid s)$, so $K(\cdot\mid s)$ is a Borel measure on $[0,\infty]$. \emph{Total mass.} Since $g(s)\ge0$ and $w_-\ge0$, the set $[0,\infty)-g(s)=[-g(s),\infty)$ contains $[w_-,w_+]\supseteq\supp P_W$. Recall from the head of this section that the support carries full mass: $P_W(\supp P_W)=1$. Therefore $K([0,\infty]\mid s)\ge K([0,\infty)\mid s)=P_W([-g(s),\infty))=1$, so $K(\cdot\mid s)$ is a probability measure. \emph{Concentration.} The support $\supp P_W$ is closed and contained in the compact set $[w_-,w_+]$, hence compact. Its image $C:=g(s)+\supp P_W$ under the homeomorphism $T_{g(s)}$ is compact as well, and $C\subseteq[g(s)+w_-,g(s)+w_+]\subseteq[0,\infty)$. Moreover $C_0-g(s)=C-g(s)=\supp P_W$, so $K(C\mid s)=P_W(\supp P_W)=1$. In particular $K(\{+\infty\}\mid s)=1-K([0,\infty)\mid s)=0$ for every $s\ge0$. \emph{Identification with $\Law(g(s)+W)$.} The real random variable $g(s)+W$ lies in $C\subseteq[0,\infty)$ almost surely (as $\PP(W\in\supp P_W)=1$), and for every Borel $B\subseteq[0,\infty]$,
\[
\PP\bigl(g(s)+W\in B\bigr)=\PP\bigl(g(s)+W\in B_0\bigr)=\PP\bigl(W\in B_0-g(s)\bigr)=P_W\bigl(B_0-g(s)\bigr)=K(B\mid s),
\]
the first equality because $g(s)+W$ is real-valued (never $+\infty$). Thus $K(\cdot\mid s)$ is the law of $g(s)+W$ on $[0,\infty]$, with all of its mass on the finite part.

\ref{it:pre-meas} The map $h:[0,\infty)\times\R\to\R$, $h(s,w):=g(s)+w$, is Borel. Indeed, $h$ is the sum of the maps $(s,w)\mapsto g(s)$ and $(s,w)\mapsto w$. The first is the composition $g\circ\pi_1$ of the Borel $g$ with the continuous first projection, and the second is continuous. Since sums of Borel functions are Borel, the claim follows. Fix Borel $B\subseteq[0,\infty]$. Then $(s,w)\mapsto\mathbf 1_{B_0}(h(s,w))=\mathbf 1_{h^{-1}(B_0)}(s,w)$ is jointly Borel, being the indicator of the Borel set $h^{-1}(B_0)$, and it is nonnegative. Since $P_W$ is finite (hence $\sigma$-finite), the measurability part of Tonelli's theorem \citep{Kallenberg2002} yields that the partial integral
\[
s\;\longmapsto\;\int_{\R}\mathbf 1_{B_0}\bigl(h(s,w)\bigr)\,P_W(dw)\;=\;K(B\mid s)
\]
is Borel measurable on $[0,\infty)$. This is the kernel measurability property, which, together with \ref{it:pre-prob}, makes $K$ a Markov kernel.
\end{proof}

\begin{lemma}[The mixture is the law of $g(S)+W$]\label{lem:mixture}
In the setting of Lemma~\ref{lem:kernelmeas}, let $S\sim F_S$ (a $[0,\infty)$-valued random variable) and $W\sim P_W$ be \emph{independent} on a common probability space, and let $K$ be the kernel \eqref{eq:pre-Kdef}. Then
\[
KF_S=\Law\bigl(g(S)+W\bigr)\qquad\text{in }\Pclass\bigl([0,\infty]\bigr),
\]
with $KF_S$ the mixture \eqref{eq:Q}. Moreover, if $\tilde g:[0,\infty)\to[0,\infty)$ is Borel with $\tilde g=g$ $F_S$-almost everywhere, and $\tilde K$ denotes the kernel \eqref{eq:pre-Kdef} built from $\tilde g$ (and the same $P_W$), then
\[
\tilde KF_S=KF_S=\Law\bigl(g(S)+W\bigr)=\Law\bigl(\tilde g(S)+W\bigr).
\]
In particular the mixture is unchanged by modifying the location map on an $F_S$-null set.
\end{lemma}

\begin{proof}
As in the proof of Lemma~\ref{lem:kernelmeas}\ref{it:pre-prob}, the real random variable $g(S)+W$ lies in $[0,\infty)$ almost surely, since $g(S)\ge0$ everywhere and $W\ge w_-\ge0$ a.s. Its law on $[0,\infty]$ therefore reads $B\mapsto\PP(g(S)+W\in B_0)$ for Borel $B\subseteq[0,\infty]$, and it puts no mass on $\{+\infty\}$. Fix such a $B$. Independence of $S$ and $W$ means $\Law(S,W)=F_S\otimes P_W$ on $[0,\infty)\times\R$ \citep{Kallenberg2002}. With $h(s,w)=g(s)+w$ as in the proof of Lemma~\ref{lem:kernelmeas}, the set $h^{-1}(B_0)$ is Borel, so Tonelli's theorem for the finite product measure $F_S\otimes P_W$ gives
\begin{multline*}
\PP\bigl(g(S)+W\in B_0\bigr)
=(F_S\otimes P_W)\bigl(h^{-1}(B_0)\bigr)\\
=\int_{[0,\infty)}\Bigl[\int_{\R}\mathbf 1_{B_0}\bigl(g(s)+w\bigr)\,P_W(dw)\Bigr]F_S(ds)
=\int_{[0,\infty)}K(B\mid s)\,F_S(ds),
\end{multline*}
and the right-hand side is $(KF_S)(B)$ by \eqref{eq:Q}. Since $B$ was an arbitrary Borel subset of $[0,\infty]$, $\Law(g(S)+W)=KF_S$.

For the almost-everywhere clause, let $N:=\{s\ge0:g(s)\neq\tilde g(s)\}$. It is the set where the Borel function $g-\tilde g$ is nonzero, hence Borel, and $F_S(N)=0$ by hypothesis. For every $s\notin N$ the defining integral \eqref{eq:pre-Kdef} depends on the location map only through its value at $s$, so $\tilde K(\cdot\mid s)=K(\cdot\mid s)$ on $N^{c}$. Hence, for every Borel $B\subseteq[0,\infty]$,
\[
(\tilde KF_S)(B)=\int_{[0,\infty)}\tilde K(B\mid s)\,F_S(ds)
=\int_{N^{c}}\tilde K(B\mid s)\,F_S(ds)
=\int_{N^{c}}K(B\mid s)\,F_S(ds)
=(KF_S)(B),
\]
the second and fourth equalities because $N$ is $F_S$-null (removing a null set does not change an integral) and the third because the integrands agree on $N^{c}$. Finally $\Law(\tilde g(S)+W)=\tilde KF_S$ by the first part applied to $\tilde g$. Alternatively, $\PP(S\in N)=F_S(N)=0$, so $\tilde g(S)+W=g(S)+W$ almost surely, and almost-surely equal random variables have the same law. Chaining the four displays proves the lemma.
\end{proof}

These three lemmas are consumed as follows. Lemma~\ref{lem:kernelmeas} legitimizes the kernels of the affine-random witness family (Section~\ref{sec:witness}). Next, Lemma~\ref{lem:push} supplies the edge and density computations behind the family's closed forms and behind the estimand transformation of the conjugate member (Sections~\ref{sec:witness} and~\ref{sec:pair}). Finally, Lemma~\ref{lem:mixture} converts kernel mixtures into laws of sums, which is how the exact coincidence $Q'=Q$ of Theorem~\ref{thm:I1} is verified in Section~\ref{sec:pair-proofs}.

\section{Membership: the witness lies in \texorpdfstring{$\Mreg$}{M-reg}}\label{sec:membership}
\begin{proof}[Proof of Lemma~\ref{lem:edges}]
\ref{it:wit-g} Each branch of \eqref{eq:gdef} is affine, hence continuous, and the branches agree at $s_{\max}$. Thus $g$ is continuous on $[0,\infty)$. For the increments, take $0\le s\le s'$ and distinguish three cases. If $s'\le s_{\max}$, then $g(s')-g(s)=c(s'-s)\ge\min(c,1)(s'-s)$. If $s\ge s_{\max}$, then $g(s')-g(s)=s'-s\ge\min(c,1)(s'-s)$. If $s\le s_{\max}\le s'$, then
\[
g(s')-g(s)=c\,(s_{\max}-s)+(s'-s_{\max})\ \ge\ \min(c,1)\bigl[(s_{\max}-s)+(s'-s_{\max})\bigr]=\min(c,1)\,(s'-s).
\]
Since $c>0$, $\min(c,1)>0$, so $g$ is strictly increasing. Finally $g(0)=a_0$ and monotonicity give $g(s)\ge a_0\ge0$ for all $s\ge0$, and $g(s)<\infty$ since each branch is a finite affine expression.

\ref{it:wit-supp} Fix $s\ge0$. First, $P_W(\supp P_W)=1$ (full mass of the support, cf.\ Lemma~\ref{lem:push}). In particular $\supp P_W\neq\varnothing$. Every point of $\supp P_W$ lies between the infimum $w_-$ and the supremum $w_+$ of that set (Convention~\ref{conv:W}), so $\supp P_W\subseteq[w_-,w_+]$ and $W\in[w_-,w_+]$ almost surely. Consequently $g(s)+W\in g(s)+[w_-,w_+]$ almost surely, and this set is compact and contained in $[0,\infty)$ because $g(s)+w_-\ge a_0+0\ge0$ by \ref{it:wit-g} and $w_-\ge0$. Thus $K(\cdot\mid s)=\Law(g(s)+W)$ is concentrated on $g(s)+[w_-,w_+]\subset[0,\infty)$, and $K(\{+\infty\}\mid s)=0$.

The measure $K(\cdot\mid s)$ is the pushforward of $P_W$ under the translation $x\mapsto g(s)+x$, a homeomorphism of $\R$. By the support identity of Lemma~\ref{lem:push}, its support computed in $\R$ is therefore $g(s)+\supp P_W$. Computed in $[0,\infty]$, the support is the same set. Indeed, $+\infty$ is excluded because its neighbourhood $(g(s)+w_+,\,\infty]$ is $K(\cdot\mid s)$-null, whereas each finite $x\in[0,\infty)$ has, in $\R$ and in $[0,\infty]$, neighbourhood bases whose elements differ only by subsets of $\R\setminus[0,\infty)$, which are null under the law of $g(s)+W$ on $\R$. The positive-mass criterion defining the support then selects the same points $x\in[0,\infty)$ in either ambient space. Finally, translation by $g(s)$ is an order isomorphism of $\R$, so for the nonempty set $A:=\supp P_W$ we get $\inf\,(g(s)+A)=g(s)+\inf A=g(s)+w_-$ and $\sup\,(g(s)+A)=g(s)+\sup A=g(s)+w_+$.

\ref{it:wit-edge} By \ref{it:wit-supp}, $\supp K(\cdot\mid s)\subseteq[0,\infty)$, so $\supp K(\cdot\mid s)\cap[0,\infty)=\supp K(\cdot\mid s)$ and the edge functionals of Section~\ref{sec:setting} evaluate to
\[
\ell(s)=\inf\supp K(\cdot\mid s)=g(s)+w_-,\qquad u(s)=\sup\supp K(\cdot\mid s)=g(s)+w_+,
\]
both finite. Hence $w(s)=u(s)-\ell(s)=w_+-w_-$ for every $s\ge0$, which is \eqref{eq:edgeforms}.

\ref{it:wit-m} By \eqref{eq:edgeforms}, $m(s)=\tfrac12\bigl(\ell(s)+u(s)\bigr)-s=g(s)+\tfrac12(w_-+w_+)-s=g(s)-s+\midW$. Now substitute the two branches of \eqref{eq:gdef}. For $0\le s\le s_{\max}$ this gives $g(s)-s=(c-1)s+a_0$, while for $s>s_{\max}$ it gives $g(s)-s=c\,s_{\max}+a_0+(s-s_{\max})-s=(c-1)s_{\max}+a_0$, a constant equal to the value of the first branch at $s_{\max}$. This is \eqref{eq:mform}.

\ref{it:wit-mass} Fix $s\ge0$ and $h\ge0$. For every $x\in\R$,
\[
g(s)+x\in\bigl[g(s)+w_-,\ g(s)+w_-+h\bigr]\iff x\in[w_-,\,w_-+h]
\]
(subtract $g(s)$ from all three members). Hence, using $\ell(s)=g(s)+w_-$ from \eqref{eq:edgeforms}, the two events coincide:
\begin{multline*}
K\bigl([\ell(s),\ell(s)+h]\mid s\bigr)
=\PP\bigl(g(s)+W\in[g(s)+w_-,\,g(s)+w_-+h]\bigr)\\
=\PP\bigl(W\in[w_-,\,w_-+h]\bigr)
=P_W\bigl([w_-,\,w_-+h]\bigr),
\end{multline*}
which is \eqref{eq:wit-mass}. The edge mass is thus one and the same function of $h$ for every $s\ge0$.

\ref{it:wit-asym} Fix $\eps\in(0,1)$. By Definition~\ref{def:witness}(c) the named density version $\rho_W$ satisfies $\rho_W(w_-+t)/(c_W t^{\beta})\to1$ as $t\downarrow0$ pointwise (Convention~\ref{conv:W}). Unwinding this limit yields $\eta_W(\eps)\in(0,\delta_W]$ such that
\[
(1-\eps)\,c_W\,t^{\beta}\ \le\ \rho_W(w_-+t)\ \le\ (1+\eps)\,c_W\,t^{\beta}
\qquad\text{for all }t\in(0,\eta_W(\eps)].
\]
Let $h\in(0,\eta_W(\eps)]$. Since $[w_-,w_-+h]\subseteq[w_-,w_-+\delta_W]$, Convention~\ref{conv:W} gives
\[
P_W\bigl([w_-,w_-+h]\bigr)=\int_{[w_-,\,w_-+h]}\rho_W(w)\,dw=\int_0^h \rho_W(w_-+t)\,dt,
\]
by translation invariance of Lebesgue measure. The singleton $\{w_-\}$, resp.\ $\{t=0\}$, is Lebesgue-null and does not affect the integrals (in particular $P_W(\{w_-\})=0$). Integrating the pointwise band over $t\in(0,h]$ and using $\int_0^h t^{\beta}\,dt=h^{\beta+1}/(\beta+1)$ (valid for $\beta\ge0$, the integrand being continuous and bounded on $(0,h]$) yields
\[
(1-\eps)\,\frac{c_W}{\beta+1}\,h^{\beta+1}\ \le\ P_W\bigl([w_-,w_-+h]\bigr)\ \le\ (1+\eps)\,\frac{c_W}{\beta+1}\,h^{\beta+1},
\]
which is \eqref{eq:wit-band}. Since $\eps\in(0,1)$ was arbitrary, $P_W([w_-,w_-+h])\big/\bigl(\tfrac{c_W}{\beta+1}h^{\beta+1}\bigr)\to1$ as $h\downarrow0$, i.e.\ the stated $(1+o(1))$ form.
\end{proof}

\begin{proof}[Proof of Proposition~\ref{prop:witness}]
Throughout, $\ell,u,w,m$ are the edge functionals of Section~\ref{sec:setting}, evaluated in closed form by Lemma~\ref{lem:edges}, and $\midW=\tfrac12(w_-+w_+)$. We record once that $0\le w_-<w_+<\infty$ forces
\[
w_+-w_->0
\qquad\text{and}\qquad
\midW=\tfrac12(w_-+w_+)\ \ge\ \tfrac12 w_+\ >\ 0 .
\]
We verify the axioms one at a time.

\smallskip\noindent\emph{A1 (latent bounded tail).}
Definition~\ref{def:witness}(a) states exactly the content of A1. By that definition, $F_S$ is a Borel probability law on $[0,\infty)$ with $F_S(\{0\})=0$, and on the window $(0,s_0)$ it has the named density version $f_S$ satisfying \eqref{eq:A1} pointwise with the member constants $L\in(0,\infty)$ and $\alpha\in[0,\infty)$. The class requirement that the A1-window reach the shared floor, $s_0\ge\bar s_0$, is the second clause of \ref{it:mem-C4}. Thus A1 holds.

\smallskip\noindent\emph{A2 (bounded conditional support, no escape near the endpoint).}
By Lemma~\ref{lem:edges}\ref{it:wit-edge}, $w(s)=w_+-w_-<\infty$ for \emph{every} $s\ge0$. Part~\ref{it:wit-supp} of the same lemma (via Lemma~\ref{lem:kernelmeas}, the conditional law being concentrated on the compact set $g(s)+[w_-,w_+]\subset[0,\infty)$) gives $K(\{+\infty\}\mid s)=0$ for \emph{every} $s\ge0$, in particular for all $s\in[0,s_A]$. Thus A2 holds.

\smallskip\noindent\emph{A3 (bias present and bounded).}
\emph{Bounded:} by \eqref{eq:mform}, for $s\in[0,s_{\max}]$,
\[
|m(s)|=\bigl|(c-1)s+a_0+\midW\bigr|\ \le\ |c-1|\,s+a_0+\midW\ \le\ |c-1|\,s_{\max}+a_0+\midW,
\]
using the triangle inequality, $a_0+\midW\ge0$, and $0\le s\le s_{\max}$. For $s>s_{\max}$ we have $m(s)=m(s_{\max})$, so the same bound applies. Hence
\[
\sup_{s\ge0}|m(s)|\ \le\ |c-1|\,s_{\max}+a_0+\midW\ \le\ \bar b\ <\ \infty
\]
by \ref{it:mem-C3}. \emph{Present:} $m(0)=a_0+\midW\ge\midW>0$ by the display opening this proof, so $m(0)\neq0$ and $m\not\equiv0$. (In fact $m$ is even non-constant on $[0,s_{\max}]$, its slope there being $c-1\neq0$ by \ref{it:mem-C5}, yet the positivity of $m(0)$ alone suffices for A3.) Thus A3 holds.

\smallskip\noindent\emph{A4 (tail-amplified heteroscedasticity).}
By \eqref{eq:edgeforms}, $w(s)/s=(w_+-w_-)/s$ for $s>0$, with the fixed numerator $w_+-w_->0$. Given $M>0$, every $s\in\bigl(0,(w_+-w_-)/M\bigr)$ satisfies $w(s)/s>M$. Hence $\lim_{s\downarrow0}w(s)/s=+\infty$, and A4 holds.

\smallskip\noindent\emph{A5 (edge-regular location transfer, no dipping).}
Again by \eqref{eq:edgeforms}, $\ell(s)=g(s)+w_-$ for every $s\ge0$. Lemma~\ref{lem:edges}\ref{it:wit-g} therefore shows $\ell$ to be continuous and strictly increasing on $[0,\infty)$, in particular \emph{nondecreasing on all of $[0,\infty)$}, which is the global-monotonicity clause of A5. No distant conditional dip back toward the apparent endpoint is possible, since $\ell$ has slope $c\ge c_->0$ on $[0,s_{\max}]$ by \ref{it:mem-C2} and Convention~\ref{conv:corridor}, slope $1>0$ on $[s_{\max},\infty)$, and the branches agree at $s_{\max}$. \emph{Existence of the limit:} for $s\in(0,s_{\max}]$, $\ell(s)=c\,s+a_0+w_-$, so
\[
a:=\ell(0+)=\lim_{s\downarrow0}\,(c\,s+a_0+w_-)=a_0+w_-=a'\ \in[0,\infty)
\]
exists. Moreover, the actual edge at $s=0$ is $\ell(0)=g(0)+w_-=a_0+w_-$, so the A5 convention $\ell(0):=a$ agrees with the true value here (for this family nothing needs to be overridden). \emph{Two-point increment condition:} let $0\le s\le s'\le s_R$. By the first clause of \ref{it:mem-C4}, $s'\le s_R\le s_{\max}$, so both points lie on the affine branch of \eqref{eq:gdef} and
\[
\ell(s')-\ell(s)=\bigl(c\,s'+a_0+w_-\bigr)-\bigl(c\,s+a_0+w_-\bigr)=c\,(s'-s);
\]
by \ref{it:mem-C2} and Convention~\ref{conv:corridor},
\[
c_-\,(s'-s)\ \le\ \ell(s')-\ell(s)\ \le\ c_+\,(s'-s),
\]
which is \eqref{eq:A5}. Thus A5 holds, with apparent endpoint $a=a'=a_0+w_-$. (Taking $s=0$ gives the corridor form $\ell(s')-a'=c\,s'$ on $[0,s_R]$. The corridor of A5 is anchored at the \emph{apparent endpoint} $a'=\ell(0+)$, \emph{not} at the affine intercept $a_0$, and the two differ by $w_-$. Substituting $a_0$ for $a'$ would produce the spurious relation $\ell(s')-a_0=c\,s'+w_-$.)

\smallskip\noindent\emph{A6 (common conditional edge index).}
By Lemma~\ref{lem:edges}\ref{it:wit-mass}, for every $s\ge0$ and $h\ge0$,
\[
K\bigl([\ell(s),\ell(s)+h]\mid s\bigr)=P_W\bigl([w_-,w_-+h]\bigr):
\]
the edge mass depends on $h$ alone, not on $s$. Define $c_K:(0,s_R]\to(0,\infty)$ as the constant map $c_K(s):=c_W/(\beta+1)$. We now verify the uniform reading \eqref{eq:A6u} of Convention~\ref{conv:A6read}. Let $\eps>0$ and set $\eps_0:=\min(\eps,\tfrac12)\in(0,1)$ and $\eta(\eps):=\eta_W(\eps_0)$, with $\eta_W$ from Lemma~\ref{lem:edges}\ref{it:wit-asym}. For every $s\in(0,s_R]$ and every $h\in(0,\eta(\eps)]$, combining the display above with \eqref{eq:wit-band} at $\eps_0$ gives
\begin{multline}\label{eq:mem-band}
(1-\eps)\,c_K(s)\,h^{\beta+1}\ \le\ (1-\eps_0)\,\frac{c_W}{\beta+1}\,h^{\beta+1}\ \le\ K\bigl([\ell(s),\ell(s)+h]\mid s\bigr)\\
\le\ (1+\eps_0)\,\frac{c_W}{\beta+1}\,h^{\beta+1}\ \le\ (1+\eps)\,c_K(s)\,h^{\beta+1},
\end{multline}
the outer inequalities holding because $1-\eps\le1-\eps_0$ and $1+\eps_0\le1+\eps$. Since the threshold $\eta(\eps)$ does not depend on $s$, \eqref{eq:mem-band} is exactly \eqref{eq:A6u}. The expansion $K([\ell(s),\ell(s)+h]\mid s)=c_K(s)\,h^{\beta+1}(1+o(1))$ as $h\downarrow0$ thus holds \emph{uniformly} over $s\in(0,s_R]$, with an $o(1)$ that does not depend on $s$. The exponent is $\beta+1$ with $\beta$ the shared class edge index by \ref{it:mem-C1}, and the profile $c_K\equiv c_W/(\beta+1)$ satisfies
\[
\liminf_{s\downarrow0}c_K(s)=\limsup_{s\downarrow0}c_K(s)=\frac{c_W}{\beta+1}\in(0,\infty).
\]
Thus A6 holds.

\smallskip\noindent\emph{NA (non-additivity).}
Suppose, for contradiction, that there were $\mu\in\Pclass(\R)$ with $K(\cdot\mid s)=\mu(\cdot-s)$ for $F_S$-a.e.\ $s$. Here $\mu(\cdot-s)$ denotes the translate $B\mapsto\mu(B-s)$, i.e.\ the law of $s+Z$ for $Z\sim\mu$. The equality is read as in the NA statement of Section~\ref{sec:setting}: $K(\{+\infty\}\mid s)=0$ and $K(B\mid s)=\mu(B-s)$ for every Borel $B\subseteq[0,\infty)$. This extends to equality of Borel measures on all of $\R$. Taking $B=[0,\infty)$ gives $\mu([-s,\infty))=K([0,\infty)\mid s)=1$, so the translate $\mu(\cdot-s)$ charges no Borel subset of $(-\infty,0)$, while $K(\cdot\mid s)$ is concentrated on a compact subset of $[0,\infty)$ (Lemma~\ref{lem:edges}\ref{it:wit-supp}). The two measures therefore agree on every Borel subset of $\R$, and in particular their supports in $\R$ coincide. Let $N\subseteq[0,\infty)$ be a Borel $F_S$-null set such that the equality holds for every $s\in[0,\infty)\setminus N$.

\emph{Step 1 (a positive-density zone).} By Definition~\ref{def:witness}(a) the named version $f_S$ satisfies, pointwise, $f_S(s)/(L s^{\alpha})\to1$ as $s\downarrow0$. Reading that limit with tolerance $\tfrac12$ yields $\delta_1\in(0,s_0)$ such that $f_S(s)\ge\tfrac{L}{2}\,s^{\alpha}$ for all $s\in(0,\delta_1]$. Set $\delta_*:=\min(\delta_1,s_{\max})\in(0,s_{\max}]$. Then $(0,\delta_*]\subseteq(0,s_0)$ and
\[
f_S(s)\ \ge\ \tfrac{L}{2}\,s^{\alpha}\ >\ 0\qquad\text{for all }s\in(0,\delta_*].
\]

\emph{Step 2 (two good points below $s_{\max}$).} We claim $\mathrm{Leb}\bigl(N\cap(0,\delta_*]\bigr)=0$. Suppose not. Since $F_S$ has density $f_S$ on $(0,s_0)\supseteq(0,\delta_*]$,
\[
F_S(N)\ \ge\ F_S\bigl(N\cap(0,\delta_*]\bigr)=\int_{N\cap(0,\delta_*]}f_S(s)\,ds\ \ge\ \tfrac{L}{2}\int_{N\cap(0,\delta_*]}s^{\alpha}\,ds\ >\ 0,
\]
where the last integral is strictly positive because its integrand is strictly positive at every point of the domain and the domain has positive Lebesgue measure. Indeed, with $A_n:=\bigl(N\cap(0,\delta_*]\bigr)\cap[\tfrac1n,\infty)$, continuity from below gives $\mathrm{Leb}(A_n)\uparrow\mathrm{Leb}(N\cap(0,\delta_*])>0$, so $\mathrm{Leb}(A_n)>0$ for some $n$, whence $\int_{N\cap(0,\delta_*]}s^{\alpha}\,ds\ge n^{-\alpha}\,\mathrm{Leb}(A_n)>0$ (here $\alpha\ge0$). This contradicts $F_S(N)=0$, proving the claim. Consequently $E:=(0,\delta_*]\setminus N$ has $\mathrm{Leb}(E)=\delta_*>0$. A set of positive Lebesgue measure is infinite, so we may pick two points $s_1,s_2\in E$ with $0<s_1<s_2\le\delta_*\le s_{\max}$, and at both of them $K(\cdot\mid s_i)=\mu(\cdot-s_i)$.

\emph{Step 3 (edge matching forces $c=1$).} Fix $i\in\{1,2\}$. The translate $\mu(\cdot-s_i)$ is the pushforward of $\mu$ under the translation $x\mapsto x+s_i$, a homeomorphism of $\R$, so by the support identity of Lemma~\ref{lem:push},
\[
\supp\mu(\cdot-s_i)\ =\ s_i+\supp\mu .
\]
Equal Borel measures have equal topological supports, the support being a functional of the measure alone. Lemma~\ref{lem:edges}\ref{it:wit-supp}, in which the support is computed indifferently in $\R$ or $[0,\infty]$, therefore gives
\[
s_i+\supp\mu\ =\ \supp K(\cdot\mid s_i)\ =\ g(s_i)+\supp P_W\ \subseteq\ g(s_i)+[w_-,w_+].
\]
The left side is a translate of $\supp\mu$, whereas the right side is nonempty (it is the support of a probability measure, which carries full mass) and bounded. Hence $\supp\mu$ is nonempty and bounded, and $e:=\inf\supp\mu\in\R$ is finite. Translation preserves infima, $\inf(s_i+A)=s_i+\inf A$. Taking them and using $\ell(s_i)=g(s_i)+w_-$ together with the affine branch of \eqref{eq:gdef}, valid since $s_i\le s_{\max}$, we obtain
\[
s_i+e\ =\ \inf\supp K(\cdot\mid s_i)\ =\ \ell(s_i)\ =\ c\,s_i+a_0+w_-\ =\ c\,s_i+a'\qquad(i=1,2).
\]
Subtracting the equation for $i=1$ from the one for $i=2$ eliminates $e$ and $a'$:
\[
s_2-s_1\ =\ c\,(s_2-s_1),\qquad\text{i.e.}\qquad (1-c)\,(s_2-s_1)=0 .
\]
Since $s_2>s_1$, this forces $c=1$, contradicting \ref{it:mem-C5}. Hence no such $\mu$ exists and NA holds.

\smallskip
All of A1--A4 and NA hold, so $(F_S,K)\in\Mclass$. Since A5 and A6 hold with the shared constants as well, $(F_S,K)\in\Mreg$, with $a=a'=a_0+w_-$ and $c_K\equiv c_W/(\beta+1)$ as stated.
\end{proof}

\begin{remark}[What consumes the off-support branch]\label{rem:mem-null}
The unit-slope branch of \eqref{eq:gdef} governs $K(\cdot\mid s)$ only at latent values $s>s_{\max}$, an $F_S$-null set, so by Lemma~\ref{lem:mixture} it leaves the observed law $Q$ unchanged. By \ref{it:mem-C4}, every $s\downarrow0$ or near-endpoint clause of the axioms reads the kernel only at $s\le s_{\max}$. Those clauses are A1, A4, the A5 increments on $[0,s_R]$ and A6 on $(0,s_R]$, along with the NA argument, which chose its two points below $s_{\max}$. In the verification above the extension is consumed at exactly two places. One is the A3 bound $\sup_{s\ge0}|m(s)|\le\bar b$, which holds because \eqref{eq:mform} freezes $m$ at $m(s_{\max})$ beyond $s_{\max}$ (the raw affine offset $(c-1)s+a_0+\midW$ diverges as $s\to\infty$ when $c\neq1$, so without the extension A3 would fail). The other is the A5 global-monotonicity clause on all of $[0,\infty)$, which holds because the extension continues $\ell$ with slope $1>0$. (A2 also quantifies over every $s\ge0$ but is insensitive to the branch, the width being constant.) This division of labour is revisited in Remark~\ref{rem:bounded}.
\end{remark}

\section{The estimation bridge}\label{sec:bridge}
\paragraph{The estimation setting.}
A member $(F_S,K)$ of the class under study is fixed, and the statistician observes an i.i.d.\ sample
\[
\tilde S_1,\dots,\tilde S_n\ \overset{\text{iid}}{\sim}\ Q=KF_S
\]
from its observed law \eqref{eq:Q}, a Borel probability measure on $[0,\infty]$. Nothing else is observed, neither a covariate nor a paired $(S,\tilde S)$. Here the asymptotics are in the sample size $n$ alone: the pair $(F_S,K)$ is held fixed, so the contamination does not vanish as $n\to\infty$. The data vector $\tilde S_{1:n}:=(\tilde S_1,\dots,\tilde S_n)$ takes values in $[0,\infty]^n$ and has law $Q^{\otimes n}$, the $n$-fold product of $Q$ with itself. Two conventions make this precise.

First, the sample space. The space $[0,\infty]$ is compact metrizable (homeomorphic to $[0,1]$, e.g.\ via $x\mapsto x/(1+x)$), hence second countable. Consequently the Borel $\sigma$-field of the product topology on $[0,\infty]^n$ coincides with the $n$-fold product $\sigma$-field, $\mathcal{B}([0,\infty]^n)=\mathcal{B}([0,\infty])^{\otimes n}$. Indeed, the product $\sigma$-field is generated by the sets $\pi_j^{-1}(B)$, where $\pi_j$ is the $j$-th coordinate projection and $B\in\mathcal{B}([0,\infty])$. Each $\pi_j$ is continuous, hence Borel measurable for the product topology, so every generator, and with it the whole product $\sigma$-field, is contained in $\mathcal{B}([0,\infty]^n)$. Conversely, by second countability every open subset of $[0,\infty]^n$ is a countable union of open boxes $U_1\times\cdots\times U_n$, and each open box is a measurable rectangle, so it lies in the product $\sigma$-field. Since $\mathcal{B}([0,\infty]^n)$ is generated by the open sets, it is contained in the product $\sigma$-field as well. ``Borel'' on $[0,\infty]^n$ is therefore unambiguous. Write $Q^{\otimes n}$ for the product probability measure on this $\sigma$-field, namely the unique probability measure satisfying
\[
Q^{\otimes n}\bigl(B_1\times\cdots\times B_n\bigr)\;=\;\prod_{j=1}^{n}Q(B_j)
\qquad\text{for all }B_1,\dots,B_n\in\mathcal{B}([0,\infty])
\]
(for existence see \citealp{Kallenberg2002}, and for uniqueness the $\pi$-system argument recorded in Step~1 of the proof of Lemma~\ref{lem:bridge} below).

Second, integrals. We write $\E_{Q^{\otimes n}}$ and $\PP_{Q^{\otimes n}}$ for expectation and probability under $Q^{\otimes n}$, abbreviating $\hat\phi_n:=\hat\phi_n(\tilde S_{1:n})$ inside both. Every expectation occurring below is the integral of a nonnegative Borel function of the sample and is read in $[0,\infty]$, where monotonicity and additivity of the integral remain valid. Suprema, limits and inequalities between such (possibly infinite) quantities are likewise read in $[0,\infty]$. Inside a supremum over a class of pairs, $Q=KF_S$ always denotes the observed law of the running member.

\paragraph{Estimators.}
An \emph{estimator} of a real-valued functional is an arbitrary sequence $\hat\phi=(\hat\phi_n)_{n\ge1}$ of Borel maps
\[
\hat\phi_n:\;[0,\infty]^n\;\longrightarrow\;\R ,
\]
the estimate at sample size $n$ being $\hat\phi_n(\tilde S_{1:n})$. No further structure is assumed: no continuity or boundedness, and no restriction on how $\hat\phi_n$ is built. It may depend on $n$, on all the shared class constants, indeed on complete knowledge of the two witnesses to be constructed for Theorem~\ref{thm:I1}. The impossibility proved below is information-theoretic, not a defect of particular procedures.

Before consistency can even be formulated, the target $\phi(F_S)$ must be a single real number. For $\phi=p_\tau$ this is immediate from Definition~\ref{def:psi}. The case $\phi=L$ (and that of the shape $\alpha$) needs an argument, because A1 introduces $(L,\alpha)$ through an existence statement: \emph{some} version of the density near $0$ satisfies \eqref{eq:A1} with \emph{some} admissible constants. Different versions, or different constants for the same version, might conceivably return different values.

\begin{proof}[Proof of Lemma~\ref{lem:phiwell}]
Throughout, ``$f$ is a density of $F_S$ on $(0,r)$'' means $F_S(B)=\int_Bf\,d\lambda$ for every Borel $B\subseteq(0,r)$, with $\lambda$ Lebesgue measure.

\emph{The two versions agree almost everywhere near $0$.} Put $s_\ast:=\min(s_0,\tilde s_0)>0$. For every Borel $B\subseteq(0,s_\ast)$,
\[
\int_Bf_S\,d\lambda\;=\;F_S(B)\;=\;\int_B\tilde f_S\,d\lambda ,
\]
and both functions are $\lambda$-integrable on $(0,s_\ast)$, their integrals there being $F_S\bigl((0,s_\ast)\bigr)\le1$. Hence $g:=f_S-\tilde f_S$ is $\lambda$-integrable on $(0,s_\ast)$ with $\int_Bg\,d\lambda=0$ for every Borel $B\subseteq(0,s_\ast)$. Taking $B=\{g>0\}\cap(0,s_\ast)$ and $B=\{g<0\}\cap(0,s_\ast)$ gives $\int_{(0,s_\ast)}g_+\,d\lambda=\int_{(0,s_\ast)}g_-\,d\lambda=0$, so $g=0$ $\lambda$-a.e.\ on $(0,s_\ast)$. There is thus a $\lambda$-null set $N\subseteq(0,s_\ast)$ with $f_S=\tilde f_S$ on $(0,s_\ast)\setminus N$.

\emph{A test sequence inside the agreement set.} For each integer $k\ge1$ the set $\bigl(0,\min(1/k,s_\ast)\bigr)\setminus N$ has Lebesgue measure $\min(1/k,s_\ast)>0$ and is therefore nonempty. Choose a point $s_k$ in it. Then $s_k\in(0,s_\ast)$, $0<s_k<1/k$, and $f_S(s_k)=\tilde f_S(s_k)$ for every $k$. In particular $s_k\to0$ as $k\to\infty$.

\emph{Comparison of the two expansions along the test sequence.} Define $\eps_1(s):=f_S(s)/(Ls^{\alpha})-1$ for $s\in(0,s_0)$ and $\eps_2(s):=\tilde f_S(s)/(\tilde Ls^{\tilde\alpha})-1$ for $s\in(0,\tilde s_0)$. Both denominators are strictly positive ($L,\tilde L>0$ and $s^{\alpha},s^{\tilde\alpha}>0$ for $s>0$). The hypotheses amount to $\eps_1(s)\to0$ and $\eps_2(s)\to0$ as $s\downarrow0$. Since $s_k\to0$ with $s_k>0$, also $\eps_1(s_k)\to0$ and $\eps_2(s_k)\to0$. Equating the two expressions for the common value $f_S(s_k)=\tilde f_S(s_k)$,
\[
L\,s_k^{\alpha}\bigl(1+\eps_1(s_k)\bigr)\;=\;\tilde L\,s_k^{\tilde\alpha}\bigl(1+\eps_2(s_k)\bigr),
\qquad k\ge1 .
\]
For all $k$ large enough that $1+\eps_1(s_k)>0$ and $1+\eps_2(s_k)>0$ (both factors tend to $1$), division is legitimate and yields
\[
s_k^{\,\alpha-\tilde\alpha}
\;=\;\frac{\tilde L}{L}\cdot\frac{1+\eps_2(s_k)}{1+\eps_1(s_k)}
\;\longrightarrow\;\frac{\tilde L}{L}\in(0,\infty)\qquad(k\to\infty).
\]
If $\alpha>\tilde\alpha$, then $s_k^{\,\alpha-\tilde\alpha}\to0$ (continuity of $x\mapsto x^{p}$ at $x=0$ for the exponent $p=\alpha-\tilde\alpha>0$), and uniqueness of limits in $\R$ forces $\tilde L/L=0$, contradicting $\tilde L/L>0$. When instead $\alpha<\tilde\alpha$, $s_k^{\,\alpha-\tilde\alpha}=1/s_k^{\,\tilde\alpha-\alpha}\to+\infty$, and a sequence tending to $+\infty$ cannot converge to the finite limit $\tilde L/L$. Hence $\alpha=\tilde\alpha$. The display then reads $1\to\tilde L/L$, whose left side is the constant $1$, so $\tilde L=L$.

\emph{Conclusion.} If every member of $\mathcal{N}$ satisfies A1, then for each member A1 supplies at least one choice of version, constants and neighbourhood as in the hypotheses. By the uniqueness just proved, all such choices return one and the same pair $(\alpha,L)$, so the assignment is total and single-valued on $\mathcal{N}$. For $p_\tau$ no auxiliary choice enters at all. The value $p_\tau=F_S([0,\tau])$ evaluates the measure $F_S$ at the fixed Borel set $[0,\tau]$, $\tau$ being a shared class constant. By the no-atom clause $F_S(\{0\})=0$ of A1 it coincides with $F_S((0,\tau])$, i.e.\ with the distribution-function value $F_S(\tau)$, as recorded in Definition~\ref{def:psi}. Both functionals ignore the kernel coordinate $K$, whence the notation $\phi(F_S)$.
\end{proof}

\begin{definition}[Modes of consistency]\label{def:consistency}
Let $\mathcal{N}$ be a class of pairs on which $\phi$ is a single-valued real-valued functional (for $\phi\in\{L,p_\tau\}$ and classes whose members satisfy A1, see Lemma~\ref{lem:phiwell}), and let $\hat\phi=(\hat\phi_n)_{n\ge1}$ be an estimator. Then $\hat\phi$ is called:
\begin{enumerate}
\item \emph{uniformly consistent in risk} for $\phi$ over $\mathcal{N}$ if
\[
\sup_{(F_S,K)\in\mathcal{N}}\E_{Q^{\otimes n}}\bigl|\hat\phi_n-\phi(F_S)\bigr|
\;\longrightarrow\;0\qquad(n\to\infty);
\]
\item \emph{uniformly consistent in probability} for $\phi$ over $\mathcal{N}$ if for every $\eps>0$,
\[
\sup_{(F_S,K)\in\mathcal{N}}\PP_{Q^{\otimes n}}\bigl(\bigl|\hat\phi_n-\phi(F_S)\bigr|\ge\eps\bigr)
\;\longrightarrow\;0\qquad(n\to\infty);
\]
\item \emph{pointwise consistent} at a member $(F_S,K)\in\mathcal{N}$ if $\hat\phi_n\to\phi(F_S)$ in $Q^{\otimes n}$-probability, i.e.\ if $\PP_{Q^{\otimes n}}\bigl(|\hat\phi_n-\phi(F_S)|\ge\eps\bigr)\to0$ as $n\to\infty$ for every $\eps>0$.
\end{enumerate}
Here $Q=KF_S$ is the observed law \eqref{eq:Q} of the running member. The composition $|\hat\phi_n-\phi(F_S)|$ of the Borel map $\hat\phi_n$ with the continuous map $x\mapsto|x-\phi(F_S)|$ is a nonnegative Borel function of $\tilde S_{1:n}$. All the expectations and probabilities above are therefore defined, the former possibly equal to $+\infty$.
\end{definition}

By Markov's inequality (integrate the pointwise bound $\eps\,\mathbf{1}\bigl\{|\hat\phi_n-\phi(F_S)|\ge\eps\bigr\}\le|\hat\phi_n-\phi(F_S)|$ and divide by $\eps>0$),
\[
\PP_{Q^{\otimes n}}\bigl(|\hat\phi_n-\phi(F_S)|\ge\eps\bigr)\;\le\;\eps^{-1}\,\E_{Q^{\otimes n}}\bigl|\hat\phi_n-\phi(F_S)\bigr|
\qquad(\eps>0),
\]
the bound holding vacuously when its right side is $+\infty$. Taking suprema over $\mathcal{N}$ and letting $n\to\infty$ shows that (i) implies (ii). Moreover, (ii) implies pointwise consistency at every member, the supremum dominating each member's term. In Lemma~\ref{lem:bridge} the clause ``no estimator sequence is uniformly consistent'' may accordingly be read in either uniform sense. The proof below refutes (i) directly and (ii) through Corollary~\ref{cor:prob}. It goes further, ruling out pointwise consistency at one of two explicitly named members. That is the strongest form of failure.

\begin{proof}[Proof of Lemma~\ref{lem:bridge}]
Let $(F^0,K^0),(F^1,K^1)\in\mathcal{N}$ be the two witnesses supplied by the hypothesis: $(F^0,K^0)\sim(F^1,K^1)$ and
\[
\Delta\;:=\;\bigl|\phi(F^0)-\phi(F^1)\bigr|\;>\;0 .
\]
By hypothesis the values $\phi(F^0),\phi(F^1)\in\R$ exist and differ. That each is a single unambiguous real number is the content of Lemma~\ref{lem:phiwell}, valid for $\phi\in\{L,p_\tau\}$ on any class whose members satisfy A1, in particular on $\Mreg$, the class of the application. For $i\in\{0,1\}$ call \emph{configuration $i$} the hypothesis that the data-generating member is $(F^i,K^i)$. Fix a sample size $n\ge1$.

\emph{Step 1: both configurations sample from the same law.} By Definition~\ref{def:ident}, observational equivalence means equality of the observed laws as Borel probability measures on $[0,\infty]$:
\[
Q^0\;:=\;K^0F^0\;=\;K^1F^1\;=:\;Q^1 ;
\]
write $Q$ for the common law. The two sampling laws then coincide as well. Indeed, $(Q^0)^{\otimes n}$ and $(Q^1)^{\otimes n}$ assign the same mass $\prod_{j=1}^{n}Q(B_j)$ to every measurable rectangle $B_1\times\cdots\times B_n$ with $B_j\in\mathcal{B}([0,\infty])$. The rectangles form a $\pi$-system (they are closed under intersections, coordinatewise) which contains the full space and generates $\mathcal{B}([0,\infty]^n)$, and two probability measures agreeing on such a $\pi$-system are equal (Dynkin's uniqueness theorem, e.g.\ \citealp{Kallenberg2002}). Thus
\[
(Q^0)^{\otimes n}\;=\;(Q^1)^{\otimes n}\;=\;Q^{\otimes n}\qquad\text{for every }n\ge1:
\]
under either configuration the data $\tilde S_{1:n}$ have the law $Q^{\otimes n}$. Consequently, for every nonnegative Borel function $h:[0,\infty]^n\to[0,\infty]$ the integral $\int h\,d(Q^i)^{\otimes n}$ is the same for $i=0$ and $i=1$. Any fixed Borel map $\hat\phi_n:[0,\infty]^n\to\R$, one map common to both configurations, likewise has one and the same distribution under both, namely the pushforward $Q^{\otimes n}\circ\hat\phi_n^{-1}$. The truth index $i$ enters the estimation problem only through the target value $\phi(F^i)$, never through the data.

\emph{Step 2: the risk floor.} Let $\hat\phi_n:[0,\infty]^n\to\R$ be an arbitrary Borel map. Both witnesses belong to $\mathcal{N}$, and both sample from $Q^{\otimes n}$ by Step 1, so
\begin{align*}
\sup_{(F_S,K)\in\mathcal{N}}\E_{Q^{\otimes n}}\bigl|\hat\phi_n-\phi(F_S)\bigr|
\;&\ge\;\max_{i\in\{0,1\}}\;\E_{Q^{\otimes n}}\bigl|\hat\phi_n-\phi(F^i)\bigr|\\
&\ge\;\tfrac12\,\E_{Q^{\otimes n}}\bigl|\hat\phi_n-\phi(F^0)\bigr|
\;+\;\tfrac12\,\E_{Q^{\otimes n}}\bigl|\hat\phi_n-\phi(F^1)\bigr|\\
&=\;\tfrac12\,\E_{Q^{\otimes n}}\Bigl[\,\bigl|\hat\phi_n-\phi(F^0)\bigr|+\bigl|\hat\phi_n-\phi(F^1)\bigr|\,\Bigr]
\;\ge\;\tfrac12\,\Delta .
\end{align*}
In order of appearance, the four relations are justified as follows. The first inequality restricts the supremum to the two witnesses, both members of $\mathcal{N}$, and rewrites the risk of witness $i$. A priori that risk is the integral of the nonnegative Borel function $\omega\mapsto|\hat\phi_n(\omega)-\phi(F^i)|$ with respect to its own sampling law $(Q^i)^{\otimes n}$, and Step 1 turns it into the integral of that same function with respect to the common law $Q^{\otimes n}$. The second uses $\max(x,y)\ge\tfrac12(x+y)$ for $x,y\in[0,\infty]$, valid also when a summand is infinite. Additivity of the $[0,\infty]$-valued integral on nonnegative Borel integrands gives the equality. The final inequality integrates the two-centre triangle inequality
\[
\bigl|x-\phi(F^0)\bigr|+\bigl|x-\phi(F^1)\bigr|\;\ge\;\bigl|\phi(F^0)-\phi(F^1)\bigr|\;=\;\Delta
\qquad(x\in\R),
\]
valid for every real $x$ since $|\phi(F^0)-\phi(F^1)|=|(x-\phi(F^1))-(x-\phi(F^0))|\le|x-\phi(F^0)|+|x-\phi(F^1)|$. This inequality is applied at $x=\hat\phi_n(\omega)$ for each sample point $\omega\in[0,\infty]^n$ and combined with monotonicity of the integral and $Q^{\otimes n}\bigl([0,\infty]^n\bigr)=1$. The chain is the first displayed inequality of Lemma~\ref{lem:bridge}. Since $\hat\phi_n$ was an arbitrary Borel map, the bound survives the infimum,
\[
\inf_{\hat\phi_n}\;\sup_{(F_S,K)\in\mathcal{N}}\;\E_{Q^{\otimes n}}\bigl|\hat\phi_n-\phi(F_S)\bigr|\;\ge\;\tfrac12\,\Delta\;>\;0
\qquad\text{for every }n\ge1,
\]
which is the second. The floor $\Delta/2$ is a strictly positive constant free of $n$. A sequence bounded below by it cannot tend to $0$, so no estimator is uniformly consistent in risk for $\phi$ over $\mathcal{N}$ (Definition~\ref{def:consistency}(i)). The in-probability mode, Definition~\ref{def:consistency}(ii), fails as well, by Corollary~\ref{cor:prob} below, whose proof reuses Step~1 only, so no circularity arises. Thus ``uniformly consistent'' is refuted in both senses of Definition~\ref{def:consistency}.

\emph{Step 3: pointwise consistency fails at one of the witnesses.} Suppose, for contradiction, that a single estimator $(\hat\phi_n)_{n\ge1}$ were pointwise consistent at \emph{both} witnesses (Definition~\ref{def:consistency}(iii)). By Step 1 the two configurations share, for each $n$, the sampling law $Q^{\otimes n}$. The two assumed convergences therefore concern a single sequence of random variables under one and the same sequence of laws:
\[
\hat\phi_n\longrightarrow\phi(F^0)
\quad\text{and}\quad
\hat\phi_n\longrightarrow\phi(F^1),
\qquad\text{both in }Q^{\otimes n}\text{-probability}.
\]
Limits in probability are unique up to null events, and the two limits here are constants, so they would have to coincide, contradicting $\Delta>0$. In detail, no uniqueness theorem need be quoted. At every point of $[0,\infty]^n$ the two-centre triangle inequality $\Delta\le|\hat\phi_n-\phi(F^0)|+|\hat\phi_n-\phi(F^1)|$ shows that at least one of the two events $\bigl\{|\hat\phi_n-\phi(F^i)|\ge\Delta/2\bigr\}$, $i\in\{0,1\}$, occurs. If neither event occurred, the right side would be $<\Delta/2+\Delta/2=\Delta$. Hence, for every $n$,
\[
1\;\le\;\PP_{Q^{\otimes n}}\bigl(|\hat\phi_n-\phi(F^0)|\ge\Delta/2\bigr)
\;+\;\PP_{Q^{\otimes n}}\bigl(|\hat\phi_n-\phi(F^1)|\ge\Delta/2\bigr).
\]
Letting $n\to\infty$ and applying the two assumed consistencies with the fixed $\eps:=\Delta/2>0$ sends the right side to $0$, giving the absurdity $1\le0$. Hence pointwise consistency cannot hold at both witnesses, as claimed.

Lemma~\ref{lem:bridge} is named after Le~Cam's two-point method. There one lower-bounds a minimax risk by the separation of two parameter values multiplied by a testing affinity between the two induced observation laws. The bound degrades as the laws pull apart in total variation (cf.\ \citealp[Ch.~2]{Tsybakov2009}, \citealp{LeCam1986}). Here the method appears in its degenerate extreme: the two observation laws are not merely close but equal, $(Q^0)^{\otimes n}=(Q^1)^{\otimes n}$ for every $n$, so their total-variation distance is $0$, and so is their Kullback--Leibler divergence. No sample size then gives the statistician any power to separate the configurations, and the two-point bound returns the full separation $\Delta/2$, uniformly in $n$.
\end{proof}

\begin{corollary}[In-probability two-point floor]\label{cor:prob}
Assume the hypotheses of Lemma~\ref{lem:bridge}: witnesses $(F^0,K^0)\sim(F^1,K^1)$ in $\mathcal{N}$ with $\phi(F^0)\neq\phi(F^1)$. Write $Q$ for the common observed law $K^0F^0=K^1F^1$ (Definition~\ref{def:ident}) and $\Delta:=|\phi(F^0)-\phi(F^1)|>0$ for the gap. Then for every $n\ge1$, every Borel map $\hat\phi_n:[0,\infty]^n\to\R$, and every $\eps\in(0,\Delta/2)$,
\[
\max_{i\in\{0,1\}}\;\PP_{Q^{\otimes n}}\bigl(\bigl|\hat\phi_n-\phi(F^i)\bigr|\ge\eps\bigr)\;\ge\;\tfrac12 .
\]
Consequently $\sup_{(F_S,K)\in\mathcal{N}}\PP_{Q^{\otimes n}}\bigl(|\hat\phi_n-\phi(F_S)|\ge\eps\bigr)\ge\tfrac12$ for every $n$, every estimator, and every $\eps\in(0,\Delta/2)$. No estimator is therefore uniformly consistent in probability for $\phi$ over $\mathcal{N}$ (Definition~\ref{def:consistency}(ii)).
\end{corollary}

\begin{proof}
By Step 1 of the proof of Lemma~\ref{lem:bridge}, both configurations induce the single sampling law $Q^{\otimes n}$ on $[0,\infty]^n$, under which all probabilities in this proof are computed. Consider the events
\[
A_i\;:=\;\bigl\{\,\omega\in[0,\infty]^n:\ \bigl|\hat\phi_n(\omega)-\phi(F^i)\bigr|<\eps\,\bigr\},
\qquad i\in\{0,1\},
\]
Borel as preimages of open intervals under the Borel map $\hat\phi_n$. They are disjoint: at any point $\omega\in A_0\cap A_1$ the two-centre triangle inequality would give
\[
\Delta\;=\;\bigl|\phi(F^0)-\phi(F^1)\bigr|\;\le\;\bigl|\hat\phi_n(\omega)-\phi(F^0)\bigr|+\bigl|\hat\phi_n(\omega)-\phi(F^1)\bigr|\;<\;2\eps\;<\;\Delta ,
\]
which is absurd, so $A_0\cap A_1=\varnothing$. Therefore $Q^{\otimes n}(A_0)+Q^{\otimes n}(A_1)=Q^{\otimes n}(A_0\cup A_1)\le1$, whence
\[
\PP_{Q^{\otimes n}}\bigl(|\hat\phi_n-\phi(F^0)|\ge\eps\bigr)
+\PP_{Q^{\otimes n}}\bigl(|\hat\phi_n-\phi(F^1)|\ge\eps\bigr)
\;=\;\bigl(1-Q^{\otimes n}(A_0)\bigr)+\bigl(1-Q^{\otimes n}(A_1)\bigr)\;\ge\;1 ,
\]
and the larger of two nonnegative numbers summing to at least $1$ is at least $\tfrac12$ (if $x+y\ge1$ then $\max(x,y)\ge\tfrac12(x+y)\ge\tfrac12$). This is the asserted floor. It remains to prove the final clause. Both witnesses lie in $\mathcal{N}$ and sample from $Q^{\otimes n}$ (Step 1), so the supremum over the class dominates the maximum over the pair, giving $\sup_{(F_S,K)\in\mathcal{N}}\PP_{Q^{\otimes n}}(|\hat\phi_n-\phi(F_S)|\ge\eps)\ge\tfrac12$ for every $n$ and every estimator. Since $\Delta>0$, the range $(0,\Delta/2)$ is nonempty. For any fixed $\eps$ in that range, this supremum is bounded below by $\tfrac12$ along all $n$ and therefore does not tend to $0$, whereas Definition~\ref{def:consistency}(ii) demands that it do so for \emph{every} $\eps>0$. Hence uniform consistency in probability fails.
\end{proof}

\begin{remark}[Scope of the bridge]\label{rem:br-generality}
The proofs of Lemma~\ref{lem:bridge} and Corollary~\ref{cor:prob} use nothing about $\Mreg$ or about the specific functionals $L$ and $p_\tau$ beyond two facts. First, the two witnesses are members of the class $\mathcal{N}$ over which the supremum is taken. Second, $\phi$ is single-valued on $\mathcal{N}$. Classwide single-valuedness is what makes the suprema in Lemma~\ref{lem:bridge} and Definition~\ref{def:consistency} well-formed expressions, whereas the inequalities themselves evaluate $\phi$ only at the two witnesses. The axioms A1--A6 and NA never enter, nor do the shared class constants or the structure of the kernels. Both statements therefore hold verbatim for an arbitrary class $\mathcal{N}$ of pairs and an arbitrary real-valued functional $\phi$ that is single-valued on $\mathcal{N}$. In the application developed here, $\mathcal{N}=\Mreg$ and $\phi\in\{L,p_\tau\}$. Single-valuedness is Lemma~\ref{lem:phiwell}, the sole point at which this section touches A1. The witnesses are supplied by the scale-conjugation pair constructed in the proof of Theorem~\ref{thm:I1}.
\end{remark}

\section{Pair proofs and the certified instance}\label{sec:pair-proofs}
Member~$1$, the $\kappa$-conjugate, carries its \emph{own} unit-slope off-support extension, breaking at $\kappa s_{\max}$, not at $s_{\max}$. This extension is essential. Had $K'$ been built from the bare affine map $s\mapsto(c/\kappa)s+a_0$ on all of $[0,\infty)$, its location offset $m'(s)=(c/\kappa-1)s+a_0+\midW$ would be unbounded (the slope $c/\kappa-1$ is nonzero because $\kappa\neq c$), and axiom A3 would fail. The extension instead caps the offset at the break and keeps the edge globally increasing. Remark~\ref{rem:bounded} isolates this point. Note also that the extension branches of $g$ and $g'$ live on the respective $F_S$- and $F_S'$-null sets $(s_{\max},\infty)$ and $(\kappa s_{\max},\infty)$. They will therefore be invisible to the observed laws (Lemma~\ref{lem:laws}), while doing their work in the membership checks (Proposition~\ref{prop:constants}).

\begin{lemma}[Member $1$ is an affine-random witness]\label{lem:member1}
In the setting of Definition~\ref{def:pair}, the pair $(F_S',K')$ is the affine-random witness of Definition~\ref{def:witness} with parameters
\[
\bigl(\alpha,\ L',\ c/\kappa,\ a_0,\ \beta,\ c_W,\ w_-,\ w_+,\ \kappa s_{\max}\bigr),
\qquad
L'\;=\;\kappa^{-(\alpha+1)}L\;=\;(\alpha+1)\,(\kappa s_{\max})^{-(\alpha+1)},
\]
and with \emph{canonical} latent law: $F_{S'}(x)=\bigl(x/(\kappa s_{\max})\bigr)^{\alpha+1}$ on $[0,\kappa s_{\max}]$. Its A1 data are exact (the $o(1)$ in \eqref{eq:A1} vanishes identically) with A1-window $s_0'=\kappa s_{\max}$, and its edge slope $c/\kappa$ satisfies $c/\kappa\in(0,\infty)$ and $c/\kappa\neq1$.
\end{lemma}

\begin{proof}
\emph{The latent law.} Write $S'\sim F_S'=\Law(\kappa S)$. By Lemma~\ref{lem:push}(ii) (distribution function under a dilation), $F_{S'}(x)=F_S(x/\kappa)$ for every $x\ge0$. Since $F_S$ is the canonical law of Definition~\ref{def:pair},
\begin{gather*}
F_{S'}(x)=F_S(x/\kappa)=\Bigl(\frac{x/\kappa}{s_{\max}}\Bigr)^{\alpha+1}=\Bigl(\frac{x}{\kappa s_{\max}}\Bigr)^{\alpha+1}
\qquad\text{for }0\le x\le\kappa s_{\max},\\
F_{S'}(x)=1\qquad\text{for }x>\kappa s_{\max},
\end{gather*}
which is the canonical latent distribution function with scale parameter $\kappa s_{\max}$ in place of $s_{\max}$. In particular $\supp F_S'\subseteq[0,\kappa s_{\max}]$ (also directly from Lemma~\ref{lem:push}(i): the support of $\Law(\kappa S)$ is $\kappa\cdot\supp F_S$). By Lemma~\ref{lem:push}(ii) again (density change of variables), $F_S'$ has on $(0,\kappa s_{\max})$ the density
\[
f_{S'}(x)\;=\;\kappa^{-1}f_S(x/\kappa)
\;=\;\kappa^{-1}(\alpha+1)\,s_{\max}^{-(\alpha+1)}\Bigl(\frac{x}{\kappa}\Bigr)^{\alpha}
\;=\;(\alpha+1)\,(\kappa s_{\max})^{-(\alpha+1)}\,x^{\alpha},
\]
i.e.\ the canonical density with scale $\kappa s_{\max}$. Its leading coefficient is
\[
L'=(\alpha+1)(\kappa s_{\max})^{-(\alpha+1)}=\kappa^{-(\alpha+1)}(\alpha+1)s_{\max}^{-(\alpha+1)}=\kappa^{-(\alpha+1)}L .
\]
There is no atom at the endpoint: $F_S'(\{0\})=\PP(\kappa S=0)=\PP(S=0)=F_S(\{0\})=0$, because $\kappa>0$ makes $\{\kappa S=0\}=\{S=0\}$. Thus $F_S'$ satisfies clause~(a) of Definition~\ref{def:witness} in its canonical form with scale parameter $\kappa s_{\max}$. The density is \emph{exactly} $L'x^{\alpha}$ on the full window $(0,s_0')$ with $s_0'=\kappa s_{\max}$, so the $o(1)$ of \eqref{eq:A1} is identically zero, and the shape and scale parameters are $\alpha$ and $L'$.

\emph{The kernel.} By \eqref{eq:pair-gprime}, $g'$ is the step-(d) location map of Definition~\ref{def:witness} for the scale $c/\kappa$ and the break point $\kappa s_{\max}$. It is affine with slope $c/\kappa$ and intercept $a_0$ on $[0,\kappa s_{\max}]$, and continues with unit slope beyond, the two branches agreeing at $s=\kappa s_{\max}$ (both equal $c\,s_{\max}+a_0$ there, since $(c/\kappa)\kappa s_{\max}=c\,s_{\max}$). Admissibility of the scale parameter is immediate: $c/\kappa\in(0,\infty)$ as a ratio of two elements of $(0,\infty)$, and $c/\kappa=1$ would mean $\kappa=c$, excluded by Definition~\ref{def:pair}. The randomization is the same jitter variable $W$, with the same Convention~\ref{conv:W} data $(\beta,c_W,w_-,w_+)$. Hence $K'(\cdot\mid s)=\Law(g'(s)+W)$ is the step-(d) kernel for these parameters. Collecting parameters, $(F_S',K')$ is the affine-random witness with parameter vector $(\alpha,L',c/\kappa,a_0,\beta,c_W,w_-,w_+,\kappa s_{\max})$ and canonical latent law.
\end{proof}

\begin{proof}[Proof of Lemma~\ref{lem:laws}]
All computations are at the level of laws.

\emph{Step 1: member $0$ as a convolution.} Let $S\sim F_S$ and $W\sim P_W$ be independent. By Lemma~\ref{lem:mixture}, $Q=KF_S=\Law(g(S)+W)$. Since $\supp F_S\subseteq[0,s_{\max}]$ and the support carries full mass, $\PP(S\in[0,s_{\max}])=1$, and on $[0,s_{\max}]$ the location map is its affine branch, $g(s)=cs+a_0$. Hence $g(S)=cS+a_0$ almost surely, and almost-surely-equal random variables share their law, so
\[
Q=\Law(cS+a_0+W).
\]
(Only the on-support branch of $g$ is ever seen by the mixture, since the unit-slope branch lives on the $F_S$-null set $(s_{\max},\infty)$.) The variables $cS$ and $W$ are independent (Borel images of the independent pair $(S,W)$). By Lemma~\ref{lem:push}(iv), the law of a sum of independent variables is the convolution of their laws and is determined by those laws alone, so
\[
Q\;=\;\bigl(\delta_{a_0}\ast\Law(cS)\bigr)\ast P_W .
\]

\emph{Step 2: member $1$ as a convolution.} Let $S'\sim F_S'$ and $W'\sim P_W$ be independent, on any probability space. Lemma~\ref{lem:mixture}, applied to the pair $(F_S',K')$ with location map $g'$, gives $Q'=K'F_S'=\Law(g'(S')+W')$, and the value of the right-hand side depends only on the two marginal laws $(F_S',P_W)$, not on the realization. By Lemma~\ref{lem:member1}, $\supp F_S'\subseteq[0,\kappa s_{\max}]$, so $\PP(S'\in[0,\kappa s_{\max}])=1$. On $[0,\kappa s_{\max}]$, \eqref{eq:pair-gprime} reads $g'(s)=(c/\kappa)s+a_0$. Just as in Step~1,
\[
Q'\;=\;\Law\bigl(\tfrac{c}{\kappa}S'+a_0+W'\bigr)\;=\;\bigl(\delta_{a_0}\ast\Law\bigl(\tfrac{c}{\kappa}S'\bigr)\bigr)\ast P_W .
\]

\emph{Step 3: the middle factors coincide.} By definition $F_S'=\Law(\kappa S)$, i.e.\ $F_S'$ is the pushforward of $F_S$ under the dilation $x\mapsto\kappa x$. In turn, $\Law\bigl(\tfrac{c}{\kappa}S'\bigr)$ is the pushforward of $F_S'$ under the dilation $x\mapsto\tfrac{c}{\kappa}x$. The composite of the dilations by $\kappa$ and by $c/\kappa$ is the dilation by $\tfrac{c}{\kappa}\cdot\kappa=c$, so by Lemma~\ref{lem:push}(iii) on composition of pushforwards,
\[
\Law\bigl(\tfrac{c}{\kappa}S'\bigr)\;=\;\Law\bigl(\tfrac{c}{\kappa}(\kappa S)\bigr)\;=\;\Law(cS).
\]

\emph{Conclusion.} The convolutions in Steps~1 and~2 have identical factors, $\delta_{a_0}$, $\Law(cS)=\Law(\tfrac{c}{\kappa}S')$, and $P_W$, composed with the same bracketing. Since convolution is an operation on the factor laws (Lemma~\ref{lem:push}(iv)), the two laws of real random variables coincide. This equality of $\R$-laws transfers to the laws on $[0,\infty]$. By Lemma~\ref{lem:mixture}, $Q(B)=\PP(cS+a_0+W\in B_0)$ and $Q'(B)=\PP(\tfrac{c}{\kappa}S'+a_0+W'\in B_0)$ for every Borel $B\subseteq[0,\infty]$, the two random variables lying in $[0,\infty)$ almost surely, so $Q(B)=Q'(B)$ for every such $B$. That is, $Q'=Q$ in $\Pclass([0,\infty])$. Finally, $S\in[0,s_{\max}]$ and $W\in[w_-,w_+]$ almost surely (each support carries full mass and is contained in the stated interval), so $cS+a_0+W\in[a_0+w_-,\ c\,s_{\max}+a_0+w_+]$ almost surely. Hence $Q=Q'$ assigns full mass to this compact subset of $[0,\infty)$, and in particular no mass to $\{+\infty\}$. Observational equivalence $(F_S,K)\sim(F_S',K')$ is now the definition (Definition~\ref{def:ident}) applied to $KF_S=K'F_S'$.
\end{proof}

\begin{remark}[The pathwise mechanism behind $Q'=Q$]
The equality of laws has a one-line coupling explanation. Realize both members on a single probability space: take $S\sim F_S$ and $W\sim P_W$ independent, and \emph{define} $S':=\kappa S$, a legitimate realization of $F_S'=\Law(\kappa S)$. This makes $W$ independent of $(S,S')$ (indeed $\sigma(S')\subseteq\sigma(S)$). Then, pointwise on this space,
\[
\frac{c}{\kappa}\,S'+a_0+W\;=\;\frac{c}{\kappa}\,(\kappa S)+a_0+W\;=\;c\,S+a_0+W ,
\]
so the two observations are not merely equal in law but are \emph{the same random variable} almost surely: the latent dilation by $\kappa$ and the kernel contraction by $1/\kappa$ cancel path by path. Both sides use only the on-support (affine) branches of $g$ and $g'$, since $S\in[0,s_{\max}]$ and $S'\in[0,\kappa s_{\max}]$ almost surely, as in the proof of Lemma~\ref{lem:laws}.
\end{remark}

\begin{proof}[Proof of Lemma~\ref{lem:estimand}]
\emph{A preliminary fact: $\kappa^{-(\alpha+1)}\neq1$.} When $\kappa\in(0,\infty)\setminus\{1\}$, we have $\log\kappa\neq0$, so the map $t\mapsto\kappa^{t}=e^{t\log\kappa}$ is strictly monotone on $\R$. Since $\kappa^{0}=1$ and $\alpha+1\ge1>0$, it follows that $\kappa^{\alpha+1}\neq1$, hence $\kappa^{-(\alpha+1)}=1/\kappa^{\alpha+1}\neq1$. (Equivalently, $\kappa^{\alpha+1}=1$ with $\kappa>0$ forces $(\alpha+1)\log\kappa=0$, so $\log\kappa=0$ and therefore $\kappa=1$, which is excluded.)

(i)--(ii) By Definition~\ref{def:pair} and Lemma~\ref{lem:member1}, $F_S$ and $F_S'$ have exact canonical densities on their full A1-windows:
\[
f_S(s)=L\,s^{\alpha}\ \text{ on }(0,s_{\max}),
\qquad
f_{S'}(s)=L'\,s^{\alpha}\ \text{ on }(0,\kappa s_{\max}),
\qquad
L'=\kappa^{-(\alpha+1)}L .
\]
Each display is an instance of \eqref{eq:A1} with $o(1)\equiv0$, exhibiting the A1 shape/scale parameters $(\alpha,L)$ for member $0$ and $(\alpha,L')$ for member $1$. By Lemma~\ref{lem:phiwell} the A1 parameters are single-valued functionals of the latent law, so these parameters \emph{are} the values, namely $\alpha'=\alpha$ and $L'=\kappa^{-(\alpha+1)}L$. Since $L>0$ and $\kappa^{-(\alpha+1)}\neq1$, we get $L'\neq L$.

(iii) Fix $\tau\in(0,\,s_{\max}\wedge\kappa s_{\max})$. By Definition~\ref{def:psi} (and the absence of atoms at $0$), $p_\tau=F_S([0,\tau])=F_S(\tau)$ and $p'_\tau=F_{S'}([0,\tau])=F_{S'}(\tau)$. These are direct evaluations of the two canonical distribution functions. Since $0<\tau<s_{\max}$,
\[
p_\tau=F_S(\tau)=\Bigl(\frac{\tau}{s_{\max}}\Bigr)^{\alpha+1}\in(0,1)
\]
(the base lies in $(0,1)$ and $x\mapsto x^{\alpha+1}$ maps $(0,1)$ into $(0,1)$). Likewise $0<\tau<\kappa s_{\max}$ places $\tau$ strictly inside member $1$'s support window, so by Lemma~\ref{lem:member1} (equivalently, by $F_{S'}(\tau)=F_S(\tau/\kappa)$ from Lemma~\ref{lem:push}(ii) together with $0<\tau/\kappa<s_{\max}$),
\[
p'_\tau=F_{S'}(\tau)=\Bigl(\frac{\tau}{\kappa s_{\max}}\Bigr)^{\alpha+1}
=\kappa^{-(\alpha+1)}\Bigl(\frac{\tau}{s_{\max}}\Bigr)^{\alpha+1}
=\kappa^{-(\alpha+1)}\,p_\tau\in(0,1).
\]
As $p_\tau>0$ and $\kappa^{-(\alpha+1)}\neq1$, we conclude $p'_\tau\neq p_\tau$.

The ratio identity $L'/L=p'_\tau/p_\tau=\kappa^{-(\alpha+1)}$ collects (ii) and (iii), division being permitted because $L,p_\tau>0$. The gaps \eqref{eq:pair-DL}--\eqref{eq:pair-Dp} are then $\Delta_L=L|1-\kappa^{-(\alpha+1)}|>0$ and $\Delta_p=p_\tau|1-\kappa^{-(\alpha+1)}|>0$, each a product of strictly positive factors.
\end{proof}

\begin{proposition}[One set of shared class constants dominating both members]\label{prop:constants}
Given the data of Definition~\ref{def:pair}, define
\begin{equation}\label{eq:pair-const}
\begin{aligned}
\beta\ &:=\ \text{the lower-edge exponent of }P_W\ \text{(Convention~\ref{conv:W})},\\
c_-\ &:=\ \min\{c,\ c/\kappa\},\qquad c_+\ :=\ \max\{c,\ c/\kappa\},\\
\bar b\ &:=\ \max\bigl(|c-1|,\ |c/\kappa-1|\bigr)\,\bigl(s_{\max}\vee\kappa s_{\max}\bigr)\;+\;a_0\;+\;\midW,\\
\bar s_0\ &:=\ s_R\ :=\ s_A\ :=\ s_{\max}\wedge\kappa s_{\max},
\end{aligned}
\end{equation}
and fix any $\tau\in(0,s_R)$. Then:
\begin{enumerate}
\item (\emph{admissibility}) $(\beta,\bar b,c_-,c_+,\tau,\bar s_0,s_R,s_A)$ is an admissible tuple of shared class constants: $\beta\ge0$, $0<c_-\le c_+<\infty$, $\bar b\in(0,\infty)$, $\bar s_0>0$, $s_R,s_A\in(0,\bar s_0]$, and $0<\tau<s_R$.
\item (\emph{membership}) both members of Definition~\ref{def:pair} satisfy the compatibility conditions (C1)--(C5) of Proposition~\ref{prop:witness} relative to these constants. Consequently
\[
(F_S,K)\in\Mreg
\qquad\text{and}\qquad
(F_S',K')\in\Mreg ,
\]
the \emph{same} class $\Mreg$, formed with the \emph{same} shared constants \eqref{eq:pair-const}.
\end{enumerate}
\end{proposition}

\begin{proof}
(i) \emph{Admissibility, item by item.} $\beta\ge0$ is part of Convention~\ref{conv:W}. Both $c$ and $c/\kappa$ lie in $(0,\infty)$ ($c,\kappa\in(0,\infty)$), so their minimum and maximum satisfy $0<c_-\le c_+<\infty$ (Convention~\ref{conv:corridor}). Consider the constant $\bar b$. Each summand in \eqref{eq:pair-const} is finite ($s_{\max},\kappa s_{\max}<\infty$ and $a_0,\midW<\infty$), so $\bar b<\infty$. For the lower bound, $\bar b\ge a_0+\midW\ge\midW=\tfrac12(w_-+w_+)\ge\tfrac12 w_+>0$, using $w_+>w_-\ge0$. Next, $\bar s_0=s_{\max}\wedge\kappa s_{\max}>0$ as the minimum of two elements of $(0,\infty)$, and $s_R=s_A=\bar s_0$ gives $s_R,s_A\in(0,\bar s_0]$. Finally $\tau\in(0,s_R)$ by its choice, the interval being nonempty since $s_R>0$.

(ii) \emph{Member $0$.} By Definition~\ref{def:pair} it is the affine-random witness with canonical latent law, A1-window $s_0=s_{\max}$ (the canonical density is exact on all of $(0,s_{\max})$), and parameters
\[
\bigl(\alpha,\ L,\ c,\ a_0,\ \beta,\ c_W,\ w_-,\ w_+,\ s_{\max}\bigr).
\]
We check (C1)--(C5) of Proposition~\ref{prop:witness}:
\begin{itemize}
\item[(C1)] Its jitter is $W$, whose lower-edge exponent is $\beta$. The class constant was \emph{defined} in \eqref{eq:pair-const} as exactly this exponent.
\item[(C2)] $c_-=\min\{c,c/\kappa\}\le c\le\max\{c,c/\kappa\}=c_+$, so $c\in[c_-,c_+]$.
\item[(C3)] $|c-1|\le\max(|c-1|,|c/\kappa-1|)$ and $0<s_{\max}\le s_{\max}\vee\kappa s_{\max}$. Multiplying these inequalities between nonnegative numbers and adding $a_0+\midW$ gives
\[
|c-1|\,s_{\max}+a_0+\midW\ \le\ \max\bigl(|c-1|,|c/\kappa-1|\bigr)\bigl(s_{\max}\vee\kappa s_{\max}\bigr)+a_0+\midW\ =\ \bar b .
\]
\item[(C4)] $s_R=s_{\max}\wedge\kappa s_{\max}\le s_{\max}$, and $s_0=s_{\max}\ge s_{\max}\wedge\kappa s_{\max}=\bar s_0$.
\item[(C5)] $c\neq1$ is a datum of Definition~\ref{def:pair}.
\end{itemize}

\emph{Member $1$.} By Lemma~\ref{lem:member1} it is the affine-random witness with canonical latent law, A1-window $s_0'=\kappa s_{\max}$, and parameters
\[
\bigl(\alpha,\ \kappa^{-(\alpha+1)}L,\ c/\kappa,\ a_0,\ \beta,\ c_W,\ w_-,\ w_+,\ \kappa s_{\max}\bigr).
\]
The same five conditions are checked in turn:
\begin{itemize}
\item[(C1)] The jitter is the \emph{same} variable $W$, with lower-edge exponent $\beta$.
\item[(C2)] $c_-\le c/\kappa\le c_+$ by the construction of $c_\pm$ in \eqref{eq:pair-const}.
\item[(C3)] $|c/\kappa-1|\le\max(|c-1|,|c/\kappa-1|)$ and $0<\kappa s_{\max}\le s_{\max}\vee\kappa s_{\max}$, so, as for member $0$,
\[
|c/\kappa-1|\,\kappa s_{\max}+a_0+\midW\ \le\ \bar b .
\]
\item[(C4)] $s_R=s_{\max}\wedge\kappa s_{\max}\le\kappa s_{\max}$, and $s_0'=\kappa s_{\max}\ge\bar s_0$.
\item[(C5)] $c/\kappa\neq1$, since $\kappa\neq c$ is a datum of Definition~\ref{def:pair}.
\end{itemize}
In both cases Proposition~\ref{prop:witness} applies and yields membership in one and the same class $\Mreg$, built from the shared constants \eqref{eq:pair-const} with the fixed threshold $\tau$.
\end{proof}

\begin{proof}[Proof of Theorem~\ref{thm:I1}]
Fix any data admissible for Definition~\ref{def:pair}. Such data exist, since Example~\ref{ex:numeric} below exhibits a fully explicit choice. Let $(F_S,K)$, $(F_S',K')$ be the resulting scale-conjugation pair. By Proposition~\ref{prop:constants}, the tuple $(\beta,\bar b,c_-,c_+,\tau,\bar s_0,s_R,s_A)$ of \eqref{eq:pair-const}, with any fixed $\tau\in(0,s_R)$, is an admissible tuple of shared class constants, and \emph{both} members lie in the class $\Mreg$ formed with these constants. This exhibits the constants and the two members required by the theorem. It remains to verify the displayed properties and draw the conclusions.

\emph{Identical observed law.} By Lemma~\ref{lem:laws}, $KF_S=K'F_S'$ exactly, as Borel laws on $[0,\infty]$. That is, $(F_S,K)\sim(F_S',K')$ in the sense of Definition~\ref{def:ident}.

\emph{Different estimand.} The class constant $\tau$ satisfies $0<\tau<s_R=s_{\max}\wedge\kappa s_{\max}$, which is precisely the hypothesis of Lemma~\ref{lem:estimand}(iii). That lemma therefore gives
\[
\alpha'=\alpha,
\qquad
L'=\kappa^{-(\alpha+1)}L\neq L,
\qquad
p'_\tau=\kappa^{-(\alpha+1)}p_\tau\neq p_\tau ,
\]
with gaps $\Delta_L=L|1-\kappa^{-(\alpha+1)}|>0$ and $\Delta_p=p_\tau|1-\kappa^{-(\alpha+1)}|>0$ as in \eqref{eq:pair-DL}--\eqref{eq:pair-Dp}. This proves the first display of the theorem.

\emph{Non-identifiability.} By Lemma~\ref{lem:phiwell}, each $\phi\in\{L,p_\tau\}$ is a single-valued real functional on $\Mreg$. Take $\phi=L$. The pair just constructed consists of two members of $\Mreg$ with $(F_S,K)\sim(F_S',K')$ and $\phi(F_S)=L\neq L'=\phi(F_S')$, so the implication defining identifiability in Definition~\ref{def:ident} fails, and $L$ is not identifiable over $\Mreg$. For $\phi=p_\tau$ the same two members give $\phi(F_S)=p_\tau\neq p'_\tau=\phi(F_S')$, so $p_\tau$ is not identifiable over $\Mreg$ either.

\emph{No uniformly consistent estimator.} Apply Lemma~\ref{lem:bridge} to the same pair, taken as the two witnesses $(F^0,K^0):=(F_S,K)$ and $(F^1,K^1):=(F_S',K')$, once with $\phi=L$ (gap $\Delta=\Delta_L>0$) and once with $\phi=p_\tau$ (gap $\Delta=\Delta_p>0$). Its hypotheses hold: both witnesses lie in $\Mreg$, $\phi$ is single-valued on $\Mreg$ (Lemma~\ref{lem:phiwell}), and the gaps are strictly positive. For every sample size $n$,
\[
\inf_{\hat\phi_n}\ \sup_{(F_S,K)\in\Mreg}\ \E_{Q^{\otimes n}}\bigl|\hat\phi_n-\phi(F_S)\bigr|
\ \ge\ \frac{\Delta_L}{2}\quad(\phi=L),
\qquad\text{respectively}\quad
\ge\ \frac{\Delta_p}{2}\quad(\phi=p_\tau):
\]
the minimax risk is bounded below by a strictly positive constant free of $n$, so no estimator is uniformly consistent in risk for $L$ or for $p_\tau$ over $\Mreg$. The in-probability mode fails as well. By Corollary~\ref{cor:prob}, for every $n$, every estimator, and every $\eps\in(0,\Delta/2)$ (with $\Delta=\Delta_L$, resp.\ $\Delta_p$), the error probability $\sup_{\Mreg}\PP_{Q^{\otimes n}}(|\hat\phi_n-\phi(F_S)|\ge\eps)$ is at least $\tfrac12$. Finally, by the pointwise addendum of Lemma~\ref{lem:bridge}, every estimator fails to be even pointwise consistent at one of the two exhibited witnesses. This proves all assertions of the theorem.
\end{proof}

\begin{example}[An explicit instance, in exact rationals]\label{ex:numeric}
Take the data
\[
\alpha=1,\qquad s_{\max}=3,\qquad c=\tfrac45,\qquad a_0=\tfrac{3}{10},\qquad
W\sim\mathrm{Unif}\bigl[0,\tfrac15\bigr],\qquad \kappa=\tfrac34 .
\]
The exclusions of Definition~\ref{def:pair} hold: $c=\tfrac45\neq1$, $\kappa=\tfrac34\neq1$, and $\kappa=\tfrac34\neq\tfrac45=c$. Convention~\ref{conv:W} is met by the uniform jitter, with $w_-=0$, $w_+=\tfrac15$ and density $\rho_W\equiv5$ on $[0,\tfrac15]$. Near the lower edge $\rho_W(w)=5\,(w-0)^{0}$ \emph{exactly}, so $\beta=0$ and $c_W=5$ with identically vanishing $o(1)$. Hence $\midW=\tfrac{1}{10}$, and the apparent endpoint of either member is $a'=a_0+w_-=\tfrac{3}{10}$.

\emph{The two members.} Member $0$ has the canonical latent law on $[0,3]$ and edge slope $c=\tfrac45$ with break at $s_{\max}=3$. By Lemma~\ref{lem:member1}, member $1$ has the canonical latent law on $[0,\kappa s_{\max}]=[0,\tfrac94]$ and edge slope
\[
\frac{c}{\kappa}=\frac{4/5}{3/4}=\frac{16}{15}\ (\approx1.067),
\qquad\text{with break at }\kappa s_{\max}=\tfrac94\ (=2.250).
\]
The common observed scale is
\[
c\,s_{\max}=\tfrac45\cdot3=\tfrac{12}{5}
=\tfrac{16}{15}\cdot\tfrac94=\tfrac{c}{\kappa}\,\bigl(\kappa s_{\max}\bigr)\ (=2.400),
\]
the algebraic identity underlying Lemma~\ref{lem:laws}. Here the common observed law $Q=Q'$ assigns full mass to $[a_0+w_-,\,c\,s_{\max}+a_0+w_+]=[\tfrac{3}{10},\tfrac{29}{10}]$.

\emph{Shared constants} (Proposition~\ref{prop:constants}, formula \eqref{eq:pair-const}):
\[
\beta=0,\qquad
c_-=\min\bigl\{\tfrac45,\tfrac{16}{15}\bigr\}=\tfrac45,\qquad
c_+=\tfrac{16}{15},\qquad
\bar s_0=s_R=s_A=3\wedge\tfrac94=\tfrac94,\qquad
\tau:=\tfrac12\in\bigl(0,\tfrac94\bigr),
\]
\[
\bar b=\max\bigl(\bigl|\tfrac45-1\bigr|,\bigl|\tfrac{16}{15}-1\bigr|\bigr)\cdot\bigl(3\vee\tfrac94\bigr)+\tfrac{3}{10}+\tfrac{1}{10}
=\max\bigl(\tfrac15,\tfrac1{15}\bigr)\cdot3+\tfrac25
=\tfrac35+\tfrac25=1 .
\]

\emph{Estimands} (Lemma~\ref{lem:estimand}, with $\alpha+1=2$ and $\kappa^{-2}=\tfrac{16}{9}$):
\[
L=2\cdot3^{-2}=\tfrac29\ (\approx0.2222),
\qquad
L'=\kappa^{-2}L=2\cdot\bigl(\tfrac94\bigr)^{-2}=\tfrac{32}{81}\ (\approx0.3951),
\]
\[
p_\tau=\Bigl(\frac{1/2}{3}\Bigr)^{2}=\tfrac{1}{36}\ (\approx0.02778),
\qquad
p'_\tau=\Bigl(\frac{1/2}{9/4}\Bigr)^{2}=\bigl(\tfrac29\bigr)^{2}=\tfrac{4}{81}\ (\approx0.04938),
\]
with the common ratio $L'/L=p'_\tau/p_\tau=\kappa^{-2}=\tfrac{16}{9}$ ($\approx1.778$), a relative discrepancy of $\kappa^{-2}-1=\tfrac79$ ($\approx77.78\%$) in both coordinates, while $\alpha'=\alpha=1$. The member-dependent factor of the affine lead coefficient \eqref{eq:affine} does not move: $L\,c_\ell^{-(\alpha+1)}=\tfrac29\bigl(\tfrac45\bigr)^{-2}=\tfrac{25}{72}$ for member~$0$ and $\tfrac{32}{81}\bigl(\tfrac{16}{15}\bigr)^{-2}=\tfrac{25}{72}$ for member~$1$, so the full lead coefficient $\bar c_K\,\Beta(\alpha+1,\beta+2)\,L\,c_\ell^{-(\alpha+1)}=5\cdot\tfrac16\cdot\tfrac{25}{72}=\tfrac{125}{432}$ is shared, as the coincidence $Q'=Q$ requires. The two-point risk floors of Lemma~\ref{lem:bridge} are
\[
\frac{\Delta_L}{2}=\frac{L}{2}\Bigl(\frac{16}{9}-1\Bigr)=\frac{7}{81}\ (\approx0.08642),
\qquad
\frac{\Delta_p}{2}=\frac{p_\tau}{2}\Bigl(\frac{16}{9}-1\Bigr)=\frac{7}{648}\ (\approx0.01080),
\]
valid at every sample size $n$.

The machine verification of this instance is reported in Section~\ref{app:num} below. It covers exact rational identities, quadrature and Monte Carlo agreement of the two observed laws, and membership spot checks.
\end{example}

\begin{remark}[A genuine non-additive member, not deconvolution]\label{rem:wit-aff-nonadd}
The affine-random witness of Definition~\ref{def:witness} is not a disguised additive model. With $c\neq1$ the contamination acts on the severity scale as the affine map $S\mapsto cS+a_0$ (plus bounded spread $W$), and there is \emph{no error characteristic function to divide out}. Indeed, in $\varphi_{\tilde S}(t)=e^{ia_0t}\varphi_S(ct)\varphi_W(t)$, the latent argument is dilated to $ct$, so $\varphi_{\tilde S}/\varphi_W$ is $e^{ia_0t}\varphi_S(ct)$, not $\varphi_S(t)$. The classical additive quotient $\varphi_Y/\varphi_\eps=\varphi_X$ is unavailable because the model is not a convolution $K(\cdot\mid s)=\mu(\cdot-s)$ (NA). Yet the latent endpoint $0$ maps to the apparent endpoint $a'=a_0+w_-$, and the latent shape $\alpha$ survives in the observed near-edge exponent (consistency check below). In extreme-value terms, an affine reparametrization of a power-law tail is again a power-law tail of the \emph{same} exponent (moving the endpoint and rescaling the lead constant, but not the index).
\end{remark}

\begin{remark}[Consistency check: the Lemma exponent on the affine witness]\label{rem:wit-aff-check}
We confirm that the observed law $Q=KF_S$ of this witness exhibits the exponent $\alpha+\beta+2$ of Lemma~\ref{lem:exponent}. On this exactly-affine, constant-$c_K$ instance, $Q$ also reproduces the exact Beta constant. For small $h>0$, a conditional $K(\cdot\mid s)$ charges $[a',a'+h]$ only when $\ell(s)\le a'+h$, i.e.\ $cs\le h$, i.e.\ $s\le h/c$. In that range the part of $[a',a'+h]$ below $\ell(s)$ is null, so
\[
K\bigl([a',a'+h]\mid s\bigr)
=K\bigl([\ell(s),\,a'+h]\mid s\bigr)
=\PP\bigl(W\in[w_-,\,w_-+(h-cs)]\bigr)
=\frac{c_W}{\beta+1}\,(h-cs)^{\beta+1}\bigl(1+o(1)\bigr),
\]
using \eqref{eq:set-rhoW} with depth $\delta_h(s)=(h-cs)_+$. Integration against $f_S(s)=Ls^{\alpha}(1+o(1))$ over the contributing range $s\in(0,h/c]$, together with the Beta identity $\int_0^{h/c}(h-cs)^{\beta+1}s^{\alpha}\,ds =c^{-(\alpha+1)}B(\alpha+1,\beta+2)\,h^{\alpha+\beta+2}$ (Proposition~\ref{prop:beta}), yields
\[
Q\bigl([a',a'+h]\bigr)
=\frac{c_W}{\beta+1}\,B(\alpha+1,\beta+2)\,L\,c^{-(\alpha+1)}\,
h^{\alpha+\beta+2}\bigl(1+o(1)\bigr)\qquad(h\downarrow0).
\]
This is the affine-edge sharpening \eqref{eq:affine} of Lemma~\ref{lem:exponent} with $c_\ell=c$ and $\bar c_K=c_W/(\beta+1)$. Equivalently, using $\tfrac{1}{\beta+1}B(\alpha+1,\beta+2)=\tfrac{1}{\beta+1}\cdot\tfrac{\Gamma(\alpha+1)\Gamma(\beta+2)}{\Gamma(\alpha+\beta+3)}=\tfrac{\Gamma(\alpha+1)\Gamma(\beta+1)}{\Gamma(\alpha+\beta+3)}$ since $\Gamma(\beta+2)=(\beta+1)\Gamma(\beta+1)$, the lead coefficient takes the symmetric form $c_W\,\dfrac{\Gamma(\alpha+1)\Gamma(\beta+1)}{\Gamma(\alpha+\beta+3)}\,L\,c^{-(\alpha+1)}$.
In particular $\lim_{h\downarrow0}\log Q([a',a'+h])/\log h=\alpha+\beta+2$, the
exponent of the Lemma. The computation checks the engine against an instance. No identifiability claim is made.
\end{remark}

\section{Further non-emptiness witnesses and the case split}\label{supp:wit}

The main text exhibits the affine-random witness in full and records that it settles non-emptiness wherever the corridor admits a scale $c\neq1$. Here we construct the two further witnesses that complete the picture, namely a genuine \emph{ratio} witness and a unit-slope \emph{corner} witness. The exhaustive non-emptiness case split is recorded as well. Membership is verified on every axiom, exactly as for the affine-random witness in Section~\ref{sec:membership}. As there, no claim about identifiability is made or used here.

\subsection{The ratio witness: growing conditional width (non-affine)}\label{ssec:wit-ratio}

Being a global affine image of a fixed law, the affine witness might leave the impression that $\Mreg$ contains only affine reparametrizations (a small step from the additive case). It does not. We now construct a member whose conditional law is \emph{not} a global affine image of any fixed law: a ratio whose conditional support \emph{widens} with $s$. This is the application-faithful shape (a closing-speed-normalized gap), and it carries the moving power-law edge of A6 from first principles.

\paragraph{Construction.}
Freeze a relative speed $v_0>0$. Let the gap jitter $\eps_d$ and speed jitter
$\eps_v$ be independent random variables, independent of the latent $S$, with
bounded supports
\[
\eps_d\in[e_d^-,e_d^+],\qquad \eps_v\in[e_v^-,e_v^+]\qquad(e_d^-<e_d^+,\ e_v^-<e_v^+),
\]
and corner densities
\begin{equation}\label{eq:wit-ratio-jitter}
\rho_d(x)=C_d\,(x-e_d^-)^{\gamma_d}\bigl(1+o(1)\bigr)\ \ (x\downarrow e_d^-),
\qquad
\rho_v(y)=C_v\,(e_v^+-y)^{\gamma_v}\bigl(1+o(1)\bigr)\ \ (y\uparrow e_v^+),
\end{equation}
with $C_d,C_v\in(0,\infty)$ and edge exponents $\gamma_d,\gamma_v\ge0$. Take $\delta_d,\delta_v\in\R$ to be deterministic offsets, and impose the two no-clipping floors
\begin{equation}\label{eq:wit-ratio-floors}
A_-:=\delta_d+e_d^->0\quad(\text{numerator floor}),
\qquad
B_-:=v_0+\delta_v+e_v^->0\quad(\text{denominator floor}),
\end{equation}
and set $A_+:=\delta_d+e_d^+$, $B_+:=v_0+\delta_v+e_v^+$ (so $0<A_-\le A_+$ and
$0<B_-\le B_+$). Let $S$ have a latent law $F_S$ on $[0,s_{\max}]$ exactly as in A1 (\eqref{eq:wit-canon}, with density $f_S(s)=Ls^{\alpha}(1+o(1))$ near $0$). Define
the observed score and kernel by
\begin{equation}\label{eq:wit-ratio-model}
\tilde S=\frac{s\,v_0+\delta_d+\eps_d}{v_0+\delta_v+\eps_v}\quad\text{given }S=s\in[0,s_{\max}],
\qquad
K(\cdot\mid s):=\mathrm{Law}\!\left(\frac{s\,v_0+\delta_d+\eps_d}{v_0+\delta_v+\eps_v}\right),
\end{equation}
and, off the latent support, extend by pure translation: for $s>s_{\max}$ put
$K(B\mid s):=K\bigl(B-(s-s_{\max})\,\big|\,s_{\max}\bigr)$, the law at $s_{\max}$
shifted up by $s-s_{\max}$. As in the affine witness, this extension lives on the $F_S$-null set $(s_{\max},\infty)$, so it leaves $Q=KF_S$ and every $s\downarrow0$ quantity unchanged. Its only role is to keep the offset $m$ bounded on $[0,\infty)$ (A3 below). The map $(x,y)\mapsto(sv_0+\delta_d+x)/(v_0+\delta_v+y)$ is continuous on the (compact) jitter support, where the denominator is $\ge B_->0$. Hence $K(\cdot\mid s)$ is a well-defined Borel law and $s\mapsto K(B\mid s)$ is measurable (the off-support translation preserving both). The pair $(F_S,K)$ is the
\emph{ratio witness} with parameters
$(\alpha,L,v_0,\delta_d,\delta_v,e_d^\pm,e_v^\pm,\gamma_d,\gamma_v,C_d,C_v)$. We
again tune parameters inside the class constants $0<c_-\le c_+<\infty$,
$\beta\ge0$, $\bar b$, as recorded where each is used below.

\paragraph{Edges in closed form, and the growing width.}
On the latent support $s\in[0,s_{\max}]$, the numerator $sv_0+\delta_d+\eps_d$ ranges over $[sv_0+A_-,\,sv_0+A_+]$ and the denominator $v_0+\delta_v+\eps_v$ over $[B_-,B_+]$, both positive. Since the two jitters are independent, all four corners are attained on the product support. The ratio is increasing in the numerator and decreasing in the denominator, so the conditional support is the interval
\begin{equation}\label{eq:wit-ratio-supp}
\supp K(\cdot\mid s)\cap[0,\infty)
=\left[\frac{sv_0+A_-}{B_+},\ \frac{sv_0+A_+}{B_-}\right]\qquad(0\le s\le s_{\max}).
\end{equation}
Hence the edges are
\begin{equation}\label{eq:wit-ratio-edges}
\ell(s)=\frac{v_0}{B_+}\,s+\frac{A_-}{B_+}=a+c_\ell s,
\qquad
u(s)=\frac{v_0}{B_-}\,s+\frac{A_+}{B_-}=:u_0+c_u s,
\end{equation}
with apparent endpoint, lower slope, and upper slope
\begin{equation}\label{eq:wit-ratio-coeff}
a=\frac{A_-}{B_+}\ge0,\qquad
c_\ell=\frac{v_0}{B_+}>0,\qquad
c_u=\frac{v_0}{B_-}>0.
\end{equation}
The lower edge is exactly affine, but the width is
\begin{equation}\label{eq:wit-ratio-width}
w(s)=u(s)-\ell(s)=w(0)+(c_u-c_\ell)\,s,
\qquad
c_u-c_\ell=\frac{v_0\,(e_v^+-e_v^-)}{B_-B_+}>0,
\end{equation}
with $w(0)=A_+/B_--A_-/B_+>0$. Thus $w(s)$ \emph{strictly increases} with $s$
(as soon as $e_v^-<e_v^+$, i.e.\ the speed jitter is non-degenerate). The location offset is
$m(s)=\tfrac12(\ell(s)+u(s))-s=\tfrac12(a+u_0)+\bigl(\tfrac12(c_\ell+c_u)-1\bigr)s$,
affine in $s$ on $[0,s_{\max}]$. For $s>s_{\max}$ the translation extension shifts both edges by $s-s_{\max}$. The width and the offset are therefore \emph{constant} off the support: $w(s)\equiv w(s_{\max})$ and $m(s)\equiv m(s_{\max})$.

\begin{proposition}[The ratio witness lies in $\Mreg$]\label{prop:wit-ratio}
Let $(F_S,K)$ be the ratio witness
\eqref{eq:wit-ratio-jitter}--\eqref{eq:wit-ratio-model} with the floors
\eqref{eq:wit-ratio-floors}, non-degenerate speed jitter $e_v^-<e_v^+$, and
parameters tuned so that $c_\ell=v_0/B_+\in[c_-,c_+]$ and
$\sup_{0\le s\le s_{\max}}|m(s)|\le\bar b$. Then $(F_S,K)$ satisfies A1--A6 and
NA, hence $(F_S,K)\in\Mreg\subseteq\mathcal{M}_\star$. Moreover its conditional
edge index is
\[
\beta=\gamma_d+\gamma_v+1,
\]
so the product-uniform case $\gamma_d=\gamma_v=0$ gives the canonical
$\beta=1$.
\end{proposition}

\begin{proof}
We verify the axioms, using \eqref{eq:wit-ratio-supp}--\eqref{eq:wit-ratio-width}.

\smallskip
\emph{A1.} The construction gives $F_S(\{0\})=0$ and $f_S(s)=Ls^{\alpha}(1+o(1))$ on $(0,s_0)$.

\smallskip
\emph{Well-posedness.} For $s\in[0,s_{\max}]$, by \eqref{eq:wit-ratio-floors} the denominator is $\ge B_->0$ and the numerator $\ge sv_0+A_->0$, so $K([0,\infty)\mid s)=1$ and $\ell(s)\in[0,\infty)$. When $s>s_{\max}$, the kernel is a nonnegative shift of the law at $s_{\max}$, so its support stays in $[0,\infty)$ and $K([0,\infty)\mid s)=1$. Thus $\ell(s)$ is well defined for every $s$.

\smallskip
\emph{A2.} By \eqref{eq:wit-ratio-supp} the conditional support is, for $s\in[0,s_{\max}]$, the bounded interval $[(sv_0+A_-)/B_+,(sv_0+A_+)/B_-]$. For $s>s_{\max}$ it is that interval at $s_{\max}$ shifted by $s-s_{\max}$, again bounded. Hence $w(s)<\infty$ for every $s$, and $K(\{+\infty\}\mid s)=0$: on the support the ratio is bounded above by $(sv_0+A_+)/B_-<\infty$, the denominator floor $B_->0$ forbids escape to $+\infty$, and the off-support shift is finite.

\smallskip
\emph{A3.} On the latent support $s\in[0,s_{\max}]$, $m(s)=\tfrac12(a+u_0)+(\tfrac12(c_\ell+c_u)-1)s$ is affine, hence bounded by $\bar b_0:=\tfrac12(a+u_0)+|\tfrac12(c_\ell+c_u)-1|\,s_{\max}$. Off the support, $s>s_{\max}$, the translation extension makes $m$ \emph{constant}, $m(s)\equiv m(s_{\max})$ with $|m(s)|\le\bar b_0$. The bound is therefore global, $\sup_{s\ge0}|m(s)|\le\bar b_0<\infty$, and we take $\bar b\ge\bar b_0$. Moreover $m\not\equiv0$: at $s=0$, $m(0)=\tfrac12(a+u_0)=\tfrac12\bigl(A_-/B_++A_+/B_-\bigr)>0$ since $A_\pm,B_\pm>0$. (The deterministic gap/speed offsets, carried through the ratio, thus produce a strictly positive apparent bias at the endpoint.)

\smallskip
\emph{A4.} By \eqref{eq:wit-ratio-width}, $w(0)=A_+/B_--A_-/B_+>0$ (since the
numerator jitter has positive width, $A_-<A_+$, while $B_-\le B_+$), and
$w(s)=w(0)+(c_u-c_\ell)s\to w(0)>0$ as $s\downarrow0$. Hence
$w(s)/s\to+\infty$.

\smallskip
\emph{A5.} The limit $a=\ell(0+)=A_-/B_+\in[0,\infty)$ exists. Set $\ell(0):=a$. On $0\le s\le s'\le s_R\ (\le s_{\max})$ the increment is $\ell(s')-\ell(s)=c_\ell(s'-s)$, and $c_\ell\in[c_-,c_+]$ by the standing tuning. The two-point increment condition $c_-(s'-s)\le c_\ell(s'-s)\le c_+(s'-s)$ therefore holds for all $0\le s\le s'\le s_R$. Finally $\ell$ is nondecreasing on all of $[0,\infty)$. Its slope is $c_\ell>0$ on $[0,s_{\max}]$ and $1>0$ for $s>s_{\max}$ (the translation extension), and the two pieces agree at $s_{\max}$, so $\ell$ is strictly increasing throughout. Full A5 holds (hence A5${}^-$).

\smallskip
\emph{A6 (and the derivation of $\beta$).}
Fix small $h>0$. We compute $K([\ell(s),\ell(s)+h]\mid s)$ near the lower-edge corner of the jitter. That corner is attained when $\eps_d$ is at its lower edge $e_d^-$ and $\eps_v$ at its upper edge $e_v^+$ (numerator minimal, denominator maximal). Introduce the corner coordinates
\[
x':=\eps_d-e_d^-\in[0,e_d^+-e_d^-],\qquad
y':=e_v^+-\eps_v\in[0,e_v^+-e_v^-],
\]
both $\to0^+$ at the corner. By \eqref{eq:wit-ratio-jitter} their joint density near $(0,0)$ factorizes as $\rho_d(e_d^-+x')\rho_v(e_v^+-y')=C_dC_v\,x'^{\,\gamma_d}y'^{\,\gamma_v}(1+o(1))$.
Write $N:=sv_0+A_-+x'$ and $D:=B_+-y'$ for numerator and denominator. The score is then $N/D$, and the lower edge is $\ell(s)=(sv_0+A_-)/B_+=A_-/B_+ + (v_0/B_+)s$. A
first-order expansion at the corner gives, with $r_0:=\ell(s)$ and to leading
order in $(x',y')$,
\[
\tilde S-\ell(s)
=\frac{sv_0+A_-+x'}{B_+-y'}-\frac{sv_0+A_-}{B_+}
=\frac{x'}{B_+}+\frac{(sv_0+A_-)}{B_+^2}\,y'+o(|x'|+|y'|)
=\frac{x'+r_0\,y'}{B_+}\bigl(1+o(1)\bigr),
\]
using $(sv_0+A_-)/B_+^2=r_0/B_+$. Hence, to leading order, the event
$\{\tilde S\le\ell(s)+h\}\cap\{\tilde S\ge\ell(s)\}$ is the corner simplex
$\{x',y'\ge0:\ x'+r_0y'\le B_+h\}$, and
\[
K\bigl([\ell(s),\ell(s)+h]\mid s\bigr)
=\iint_{x'+r_0y'\le B_+h,\ x',y'\ge0}C_dC_v\,x'^{\,\gamma_d}y'^{\,\gamma_v}\,dx'\,dy'\,
\bigl(1+o(1)\bigr).
\]
The two-dimensional corner (Dirichlet-type) integral evaluates in closed form:
for $t>0$, $r>0$, $g_d,g_v\ge0$,
\begin{equation}\label{eq:wit-dirichlet}
\iint_{x+ry\le t,\ x,y\ge0}x^{g_d}y^{g_v}\,dx\,dy
=r^{-(g_v+1)}\,\frac{\Gamma(g_d+1)\,\Gamma(g_v+1)}{\Gamma(g_d+g_v+3)}\,
t^{\,g_d+g_v+2}
\end{equation}
(integrate $x$ over $[0,t-ry]$ to get $(t-ry)^{g_d+1}/(g_d+1)$ and then $y$ over $[0,t/r]$ via the Beta integral). Applying
\eqref{eq:wit-dirichlet} with $t=B_+h$, $r=r_0=\ell(s)$, $g_d=\gamma_d$,
$g_v=\gamma_v$,
\[
K\bigl([\ell(s),\ell(s)+h]\mid s\bigr)
=C_dC_v\,\frac{\Gamma(\gamma_d+1)\Gamma(\gamma_v+1)}{\Gamma(\gamma_d+\gamma_v+3)}\,
\ell(s)^{-(\gamma_v+1)}\,B_+^{\,\gamma_d+\gamma_v+2}\,
h^{\,\gamma_d+\gamma_v+2}\bigl(1+o(1)\bigr).
\]
This is of the form $c_K(s)\,h^{\beta+1}(1+o(1))$ with
\begin{equation}\label{eq:wit-ratio-beta-cK}
\beta=\gamma_d+\gamma_v+1,
\qquad
c_K(s)=C_dC_v\,\frac{\Gamma(\gamma_d+1)\Gamma(\gamma_v+1)}{\Gamma(\gamma_d+\gamma_v+3)}\,
B_+^{\,\gamma_d+\gamma_v+2}\,\ell(s)^{-(\gamma_v+1)} .
\end{equation}
As $s\downarrow0$, $\ell(s)\to a=A_-/B_+>0$, so $c_K(s)\to C_dC_v\,\Gamma(\gamma_d+1)\Gamma(\gamma_v+1)\Gamma(\gamma_d+\gamma_v+3)^{-1} B_+^{\gamma_d+\gamma_v+2}(A_-/B_+)^{-(\gamma_v+1)}\in(0,\infty)$. In particular $0<\liminf_{s\downarrow0}c_K(s)\le\limsup_{s\downarrow0}c_K(s)<\infty$. The expansion is uniform over $s\in(0,s_R]$, as A6 requires. Indeed, $\ell(s)$ stays in the compact interval $[a,a+c_\ell s_R]\subset(0,\infty)$, on which $r_0=\ell(s)$ is bounded away from $0$ and $\infty$. There the corner remainder $o(1)$ in $(x',y')$ is controlled uniformly (the jitter law is fixed and $s$-independent, entering only through the bounded parameter $r_0$). More explicitly, the rescaling $x'=B_+h\,\xi$, $y'=(B_+h/r_0)\,\eta$ carries the exact exceedance region $\{x'+(r_0+h)y'\le B_+h\}$ onto $\{\xi+(1+h/r_0)\eta\le1\}$, which converges to the fixed unit simplex $\{\xi+\eta\le1,\ \xi,\eta\ge0\}$. The same rescaling factors the weight as $x'^{\,\gamma_d}y'^{\,\gamma_v}=(B_+h)^{\gamma_d}(B_+h/r_0)^{\gamma_v}\xi^{\,\gamma_d}\eta^{\,\gamma_v}$. On that fixed compact simplex the corner-linearization remainder and the density remainder of \eqref{eq:wit-ratio-jitter} are dominated by a single modulus tending to $0$, so the pointwise $o(1)$'s integrate to the single multiplicative $(1+o(1))$ written outside the integral. The strict numerator floor $A_->0$ is imposed to keep $a=A_-/B_+>0$, hence $c_K$ bounded. Were $A_-=0$, the edge would pass through the origin and $c_K(s)\asymp\ell(s)^{-(\gamma_v+1)}\to\infty$, violating A6. A6 therefore holds with the common class index $\beta=\gamma_d+\gamma_v+1$.

\smallskip
\emph{NA.}
By \eqref{eq:wit-ratio-edges} the lower and upper edges advance with slopes $c_\ell=v_0/B_+$ and $c_u=v_0/B_-$. Since $e_v^-<e_v^+$, \eqref{eq:wit-ratio-width} gives $c_u-c_\ell=v_0(e_v^+-e_v^-)/(B_-B_+)>0$. A translate family $K(\cdot\mid s)=\mu(\cdot-s)$ would advance \emph{both} edges with the common slope $1$ and keep the width $w(s)$ constant. Here the width is non-constant (strictly increasing), so no fixed $\mu$ can serve even two distinct $s$ in the support of $F_S$. A fortiori, NA holds. (Equivalently, $c_\ell\neq c_u$ already rules out a rigid translation, which would force $c_\ell=c_u=1$.)

\smallskip
All of A1--A6 and NA hold, so $(F_S,K)\in\Mreg\subseteq\mathcal{M}_\star$.
\end{proof}

\begin{remark}[The ratio witness: non-affine and faithful to the application]\label{rem:wit-ratio-nonaff}
The defining non-affine feature is \eqref{eq:wit-ratio-width}: the conditional support width $w(s)=w(0)+(c_u-c_\ell)s$ grows linearly in $s$ with strictly positive slope $c_u-c_\ell$. An affine image $S\mapsto cS+a+($fixed spread$)$, in particular the affine-random witness of Section~\ref{sec:witness}, has \emph{constant} width, so no global affine map (let alone a translation) carries a fixed law onto this kernel. The mechanism is division by the perturbed closing speed $v_0+\delta_v+\eps_v$: the same speed-jitter band $[e_v^-,e_v^+]$ produces a relative spread that scales with the numerator $sv_0+\delta_d+\eps_d$, hence with $s$. This is the construction of a closing-speed-normalized surrogate severity, a ``detector whose resolution floor does not shrink with the signal''. The fuller derivation belongs to the main text's motivating application (Section~\ref{sec:app}): it connects $(\delta_d,\delta_v)$ to a sensing bias differential, the jitter to tracking error, and $\beta=\gamma_d+\gamma_v+1$ to the gap/speed corner exponents. None of it is needed for the axiom verification above.
\end{remark}

\begin{remark}[Consistency check on the ratio witness]\label{rem:wit-ratio-check}
The ratio witness has an affine lower edge ($\ell(s)=a+c_\ell s$) but a
non-constant $c_K(s)=c_K(0+)\,(\ell(s)/a)^{-(\gamma_v+1)}=c_K(0+)(1+(c_\ell/a)s)^{-(\gamma_v+1)}$
with finite positive limit $c_K(0+)$ as $s\downarrow0$. Therefore the affine-edge sharpening of Lemma~\ref{lem:exponent} applies. Here the computation is the affine consistency check (Remark~\ref{rem:wit-aff-check}) with $c_\ell=v_0/B_+$ and $\bar c_K=c_K(0+)$ in place of the affine edge slope and edge-mass constant, giving
\[
Q\bigl([a,a+h]\bigr)=\bar c_K\,B(\alpha+1,\beta+2)\,L\,c_\ell^{-(\alpha+1)}\,
h^{\alpha+\beta+2}\bigl(1+o(1)\bigr),\qquad \beta=\gamma_d+\gamma_v+1,
\]
so $\lim_{h\downarrow0}\log Q([a,a+h])/\log h=\alpha+\beta+2=\alpha+\gamma_d+\gamma_v+3$. This is again the Lemma's exponent, now on a non-affine instance, and it carries no identifiability content (Convention~\ref{conv:wit-scope}).
\end{remark}

\subsection{The corner witness: unit slope, growing width}\label{ssec:wit-corner}

Neither witness above reaches the corner $c_-=c_+=1$ with $\beta\in[0,1)$. The affine witness breaks additivity through its tilted edge and so needs $c\neq1$, which is excluded at the corner, where A5 forces the edge slope to be exactly $1$. In contrast, the ratio witness owes its non-additivity to a strictly growing width, yet cannot realize $\beta<1$ ($\beta=\gamma_d+\gamma_v+1\ge1$). We supply a third member that does, breaking additivity through a \emph{growing width} rather than a tilted edge.

\paragraph{Construction.}
Fix class constants $\beta\ge0$ and a corridor with $c_-\le1\le c_+$ (in particular the corner $c_-=c_+=1$), together with a member shape $\alpha\in[0,\infty)$. Let
$S$ have the latent law $F_S$ on $[0,s_{\max}]$ exactly as in the affine witness
\eqref{eq:wit-canon} (density $f_S(s)=Ls^{\alpha}(1+o(1))$ near $0$,
$F_S(\{0\})=0$, $s_0=s_{\max}$). Take an intercept $a\ge0$, a strictly growing
width map
\begin{equation}\label{eq:wit-corner-width}
r(s):=r_0+\rho\,s\qquad(r_0>0,\ \rho>0),
\end{equation}
and a jitter $V\in[0,1]$ independent of $S$ with distribution $\PP(V\le b)=b^{\beta+1}$ on $[0,1]$. This is the $\mathrm{Beta}(\beta+1,1)$ law, with density $(\beta+1)b^{\beta}$ and no atom at $0$ (equivalently $V=U^{1/(\beta+1)}$, $U\sim\mathrm{Unif}[0,1]$). Define the kernel by
\begin{equation}\label{eq:wit-corner-model}
K(\cdot\mid s):=\mathrm{Law}\bigl(a+s+r(s)\,V\bigr)\qquad(0\le s\le s_{\max}),
\end{equation}
and, off the latent support, extend by pure translation as in the affine and
ratio witnesses: for $s>s_{\max}$ put $K(B\mid s):=K\bigl(B-(s-s_{\max})\mid
s_{\max}\bigr)$. The map $(s,v)\mapsto a+s+r(s)\,v$ is continuous, hence Borel, so $K(\cdot\mid s)$ is a well-defined Borel law. For each Borel $B$, the map $s\mapsto K(B\mid s)=\int_{[0,1]}\mathbf 1_B\bigl(a+s+r(s)\,v\bigr)\,P_V(dv)$ is measurable by the measurability half of Tonelli's theorem \citep{Kallenberg2002} (the off-support translation preserving both properties). Hence $K$ is a Markov kernel. Lemma~\ref{lem:kernelmeas}
does not apply here: the jitter is \emph{scaled} by $r(s)$, not merely translated.
The pair $(F_S,K)$ is the \emph{corner witness}.

Since $V\in[0,1]$, the conditional support is the interval
$[a+s,\,a+s+r(s)]$, so for every $s\ge0$
\begin{equation}\label{eq:wit-corner-edges}
\ell(s)=a+s,\qquad u(s)=a+s+r(s),\qquad w(s)=r(s),\qquad m(s)=a+\tfrac12 r(s),
\end{equation}
on $[0,s_{\max}]$. Off the support, the translation extension continues $\ell(s)=a+s$ (slope $1$, continuous at $s_{\max}$) and freezes $w(s)\equiv r(s_{\max})$, $m(s)\equiv m(s_{\max})$, just as in the affine witness. The lower edge thus has slope exactly $1$ everywhere, while the width
$w(s)=r_0+\rho s$ strictly grows.

\begin{proposition}[The corner witness lies in $\Mreg$]\label{prop:wit-corner}
Let $(F_S,K)$ be the corner witness \eqref{eq:wit-corner-width}--\eqref{eq:wit-corner-model} with $c_-\le1\le c_+$, any $\beta\ge0$, and parameters tuned so that $a+\tfrac12 r(s_{\max})\le\bar b$ (e.g.\ $a=0$ and $r_0=\rho\,s_{\max}=\bar b$, giving $\sup|m|=\bar b$, with strict slack for $r_0=\rho\,s_{\max}=\bar b/2$). Then $(F_S,K)$
satisfies A1--A6 and NA, hence $(F_S,K)\in\Mreg\subseteq\Mstar$, realizing the
prescribed $\alpha$ and $\beta$. The neighbourhood floors are met by the
construction itself, $\bar s_0:=s_R:=s_A:=s_{\max}$.
\end{proposition}

\begin{proof}
Throughout, $\ell,u,w,m$ are as in \eqref{eq:wit-corner-edges}. We reuse the affine witness's discharge of every step that depends only on the off-support translation patch (Proposition~\ref{prop:witness}).

\emph{A1} holds by construction (any $\alpha$), being the affine witness's latent
law \eqref{eq:wit-canon}. \emph{Well-posedness:} $a+s+r(s)V\ge a+s\ge0$, and
the off-support shift is upward, so $K([0,\infty)\mid s)=1$ and $\ell(s)$ is well
defined for every $s$. \emph{A2:} $w(s)=r(s)<\infty$ on $[0,s_{\max}]$ and $w(s)=r(s_{\max})<\infty$ beyond. Since $V$ is bounded and the shift finite, $K(\{+\infty\}\mid s)=0$ for all $s$. \emph{A4:}
$w(s)/s=(r_0+\rho s)/s\to+\infty$ as $s\downarrow0$ because $r_0>0$.

\emph{A3.} By \eqref{eq:wit-corner-edges}, $m(s)=a+\tfrac12 r(s)$ is affine and increasing on $[0,s_{\max}]$, hence $\sup_{[0,s_{\max}]}|m|=a+\tfrac12 r(s_{\max})=:\bar b_0$. Off the support, $m\equiv m(s_{\max})$, so $\sup_{s\ge0}|m(s)|=\bar b_0<\infty$, a global bound, as A3 requires. We take $\bar b\ge\bar b_0$, attainable for any prescribed $\bar b>0$ by the stated tuning. Moreover $m(0)=a+\tfrac12 r_0>0$ (as $r_0>0$), so $m\not\equiv0$.

\emph{A5 (full).} The limit $a=\ell(0+)$ exists. With $\ell(0):=a$, for
$0\le s\le s'\le s_R\ (=s_{\max})$ the increment is $\ell(s')-\ell(s)=s'-s$, so
\[
c_-\,(s'-s)\ \le\ s'-s\ \le\ c_+\,(s'-s)
\]
holds for all such $s\le s'$ \emph{iff} $c_-\le1\le c_+$, with equality at the corner $c_-=c_+=1$. A5 permits that corner, its inequalities \eqref{eq:A5} being non-strict. Taking $s=0$ gives the one-point corridor. Finally $\ell(s)=a+s$ is strictly increasing, hence nondecreasing, on all of $[0,\infty)$ (slope $1$ on $[0,s_{\max}]$ and beyond, the pieces agreeing at $s_{\max}$). Full A5
holds.

\emph{A6.} For $s\in(0,s_R]$ and $0\le h\le r(s)$, since $V\in[0,1]$,
\begin{equation}\label{eq:wit-corner-A6}
K\bigl([\ell(s),\ell(s)+h]\mid s\bigr)
=\PP\bigl(r(s)\,V\le h\bigr)
=\PP\bigl(V\le h/r(s)\bigr)
=r(s)^{-(\beta+1)}\,h^{\beta+1},
\end{equation}
\emph{exactly}, by the $\mathrm{Beta}(\beta+1,1)$ CDF. Thus A6 holds with edge index $\beta$ and $c_K(s)=r(s)^{-(\beta+1)}$, continuous with $c_K(s)\to r_0^{-(\beta+1)}\in(0,\infty)$ as $s\downarrow0$. In particular $0<\liminf_{s\downarrow0}c_K=\limsup_{s\downarrow0}c_K<\infty$. For the uniformity of Convention~\ref{conv:A6}, take $\eta(\eps):=r_0$ for every $\eps>0$. Since $r(s)=r_0+\rho s\ge r_0$ for all $s\in(0,s_R]$, every $\delta\in(0,r_0]$ satisfies $\delta\le r_0\le r(s)$, so \eqref{eq:wit-corner-A6} applies with $h=\delta$. Consequently, for every such $\delta$, the ratio $K([\ell(s),\ell(s)+\delta]\mid s)/(c_K(s)\delta^{\beta+1})$ equals $1$ for all $s\in(0,s_R]$ simultaneously, the $o(1)$ vanishing identically on the window $\delta\le r_0$. The exact power law holds for \emph{every} $\beta\ge0$, in
particular for $\beta\in[0,1)$. (Here $c_K$ is independent of $a$, so the choice $a=0$ is harmless, unlike in the ratio witness, whose $c_K$ blows up at $A_-=0$.)

\emph{NA.} Suppose a fixed $\mu\in\Pclass(\R)$ had $K(\cdot\mid s)=\mu(\cdot-s)$
for $F_S$-a.e.\ $s$. Equality as laws for two distinct admissible values $s_1\neq s_2$ in the (uncountable, hence non-null) support of $F_S$ forces $\supp\mu=\supp K(\cdot\mid s_i)-s_i=[a,\,a+r(s_i)]$, whence $r(s_1)=r(s_2)$. Since $\rho>0$, this gives $s_1=s_2$, a contradiction. (The edge slope is exactly $1$, identical to a translate, so non-additivity here cannot come from the edge, unlike in the affine and ratio witnesses. It comes instead from the growing width $w(s)=r(s)$, which no translate, having constant width, can match.) Hence NA
holds.

\smallskip
All of A1--A6 and NA hold, so $(F_S,K)\in\Mreg\subseteq\Mstar$.
\end{proof}

\subsection{Non-emptiness over the full parameter range}\label{supp:wit-record}

\begin{proposition}[Non-emptiness of $\mathcal{M}_\star$ and presence of non-additive, non-affine members]\label{prop:wit-nonempty}
For every choice of class constants $\beta\ge0$, $0<c_-\le c_+<\infty$, $\bar b\in(0,\infty)$, and every member shape $\alpha\in[0,\infty)$, the classes $\Mreg$ and $\mathcal{M}_\star$ are non-empty and contain a member with the given $\alpha$. In each case the witness realizes the neighbourhood floors $\bar s_0,s_R,s_A$, and one may take $\bar s_0=s_R=s_A=s_{\max}$. The genuinely shared class constants are $\beta,c_-,c_+,\bar b$. Indeed:
\begin{enumerate}
\item the affine-random witness of Proposition~\ref{prop:witness}, with any $c\in[c_-,c_+]\setminus\{1\}$ and any jitter satisfying \eqref{eq:set-rhoW} with the prescribed $\beta$, is a member that is \emph{non-additive}.
\item the ratio witness of Proposition~\ref{prop:wit-ratio}, with non-degenerate speed jitter, $c_\ell=v_0/B_+\in[c_-,c_+]$, and $\gamma_d+\gamma_v+1=\beta$ (available whenever $\beta\ge1$), is a member that is \emph{non-additive and non-affine} (its conditional width grows with $s$).
\item the corner witness of Proposition~\ref{prop:wit-corner}, with unit edge
slope, growing width, and any $\beta\ge0$, is a member that is \emph{non-additive
and non-affine} and is available for every corridor with $c_-\le1\le c_+$.
\end{enumerate}
The three witnesses cover every admissible quadruple $(\beta,c_-,c_+,\alpha)$. If $c_-=c_+=1$, use the corner witness~(3) (any $\beta\ge0$). Otherwise $[c_-,c_+]$ contains some $c\neq1$, and~(1) applies with that $c$ (any $\beta\ge0$), while~(2) supplies a non-affine alternative whenever $\beta\ge1$. In particular Lemma~\ref{lem:exponent} and Theorem~\ref{thm:I2} are not vacuous. The class is never confined to additive (translation) kernels, and whenever $\beta\ge1$ or $c_-\le1\le c_+$ it is not confined to global affine reparametrizations either.
\end{proposition}

\begin{proof}
The claim is immediate from Propositions~\ref{prop:witness}, \ref{prop:wit-ratio}, and~\ref{prop:wit-corner}. Each of them constructs, for the given constants and any prescribed $\alpha$, an explicit pair in $\Mreg\subseteq\mathcal{M}_\star$, with the floors $\bar s_0,s_R,s_A$ set by the construction (e.g.\ $s_R=s_{\max}$). The case split below exhibits a member for every admissible quadruple.

\emph{Corner $c_-=c_+=1$.} The affine witness needs $c\in[c_-,c_+]\setminus\{1\}=\varnothing$ and the ratio witness needs $\beta\ge1$, so neither applies for $\beta\in[0,1)$. Here the corner witness~(3) covers all $\beta\ge0$, its A5 corridor $c_-(s'-s)\le s'-s\le c_+(s'-s)$ holding with equality.

\emph{Non-corner $\neg(c_-=c_+=1)$.} Then $[c_-,c_+]$ contains a value $c\neq1$: either $c_-<c_+$, so that the interval is uncountable and at most one of its points equals $1$, or $c_-=c_+\neq1$, in which case that single value is $\neq1$.
Witness~(1) with this $c$ realizes any $\beta\ge0$. To realize a prescribed common $\beta$ one selects the jitter edge exponent: $\beta$ directly via \eqref{eq:set-rhoW} for~(1), or $\gamma_d+\gamma_v=\beta-1$ via \eqref{eq:wit-ratio-jitter} (e.g.\ $\gamma_d=\beta-1,\gamma_v=0$, or $\gamma_d=\gamma_v=(\beta-1)/2$) for~(2) when $\beta\ge1$. Whenever $c_-\le1\le c_+$, the corner witness~(3) is the non-affine carrier available for every $\beta\ge0$.

Witnesses~(1)--(3) fail to be translates (NA, proved in each proposition), while~(2) and~(3) also fail to be global affine images (Remark~\ref{rem:wit-ratio-nonaff} and the growing width \eqref{eq:wit-corner-edges}).
\end{proof}

\begin{remark}[Comparison of the three witnesses]\label{rem:wit-compare}
All members share the near-edge geometry the engine reads (an affine lower edge $\ell(s)=a+c_\ell s$ and a power-law edge-mass profile of index $\beta$), and all preserve the latent shape $\alpha$ in the observed exponent $\alpha+\beta+2$. For the corner witness the affine-edge sharpening of Lemma~\ref{lem:exponent} applies with $c_\ell=1$ and $\bar c_K=r_0^{-(\beta+1)}$, and gives $Q([a,a+h])=r_0^{-(\beta+1)}\,\Beta(\alpha+1,\beta+2)\,L\,h^{\alpha+\beta+2}(1+o(1))$. That computation is the counterpart of Remark~\ref{rem:wit-aff-check} for the affine witness and of Remark~\ref{rem:wit-ratio-check} for the ratio witness. Globally, the three witnesses differ. A rigid affine image of a fixed law, the affine-random witness has constant width $w(s)\equiv w_+-w_-$ and is the cleanest non-additive carrier. The ratio witness has a \emph{growing} width $w(s)=w(0)+(c_u-c_\ell)s$ with edge slope $c_\ell=v_0/B_+\in[c_-,c_+]$, and it requires $\beta\ge1$. With \emph{unit} edge slope ($c_\ell=1$) and growing width $w(s)=r_0+\rho s$, the corner witness is the member that covers the corner $c_-=c_+=1$ for \emph{every} $\beta\ge0$. Their coexistence shows that the edge-regularity axioms A5--A6 isolate exactly the near-endpoint behaviour the theorem needs while leaving the global conditional shape free. Additive, affine, and genuinely nonlinear kernels all qualify.
\end{remark}

\section{An estimation remark}\label{est:sec}

Section~\ref{supp:est} below gives a plug-in estimator for $\alpha$ and a heuristic consistency argument for it, along with a discussion of the window and endpoint choices. We prove identifiability here, not a rate. Section~\ref{supp:num} also reports the numerical self-check (the script \texttt{i2\_selfcheck.py}) corroborating the affine-case Beta constant of Lemma~\ref{lem:exponent} and the exponent $\alpha+\beta+2$ under a curved edge and an oscillating mass prefactor $c_K$. A second script, \texttt{i1\_selfcheck.py} (Section~\ref{app:num}), certifies the confounding instance of Theorem~\ref{thm:I1}.
\subsection{A plug-in estimator for \texorpdfstring{$\alpha$}{alpha} and heuristic consistency}\label{supp:est}

\begin{remark}[A natural estimator of $\alpha$ and heuristic consistency]\label{est:remark}
The recovery map is constructive and suggests a natural plug-in estimator. Recall from Theorem~\ref{thm:I2} and Lemma~\ref{lem:exponent} that, on the class,
\[
\alpha\;=\;\lim_{h\downarrow0}\frac{\log Q\bigl([\,a,\,a+h\,]\bigr)}{\log h}\;-\;\beta\;-\;2,\qquad a=\inf\supp Q,
\]
with $\beta$ the known class constant. Given an i.i.d.\ sample $\tilde S_1,\dots,\tilde S_n\sim Q$ with empirical measure $Q_n$, the estimator replaces $a$ and the local log-log slope by sample analogues:
\begin{enumerate}
\item estimate the apparent endpoint by the smallest order statistic (or a low sample quantile), $\widehat a:=\min_{1\le i\le n}\tilde S_i$.
\item over a window $h$ in a vanishing range, regress $\log Q_n\bigl([\widehat a,\widehat a+h]\bigr)$ on $\log h$ and read off the slope $\widehat m$ (equivalently, a two-point or local-polynomial log-log slope on a grid of small $h$).
\item set $\widehat\alpha:=\widehat m-\beta-2$.
\end{enumerate}

\emph{Heuristic consistency.} The population target of $\widehat m$ is the left-endpoint log-log slope of $Q$, which the two-sided bound \eqref{eq:twosided} of Lemma~\ref{lem:exponent} pins exactly. Taking logarithms confines $\log Q([a,a+h])-(\alpha+\beta+2)\log h$ to a fixed bounded interval as $h\downarrow0$, so the slope converges to $\alpha+\beta+2$ regardless of any oscillation of the lead constant (the squeeze of Corollary~\ref{cor:exponent}). The transfer from $Q$ to $Q_n$ rests on two facts. First, $\widehat a\to a$ since $a=\inf\supp Q$. Second, $\sup_h|Q_n([\widehat a,\widehat a+h])-Q([a,a+h])|\to0$ by Glivenko--Cantelli (e.g.\ \citealp{vdV1998,Kallenberg2002}) with the continuity of $h\mapsto Q([a,a+h])$ at $0$. Provided the $h$-window shrinks slowly enough that $Q_n([\widehat a,\widehat a+h])$ stays informative, $\widehat m$ inherits the limit $\alpha+\beta+2$, whence $\widehat\alpha\to\alpha$. Two points remain delicate. One is the joint choice of the vanishing $h$-window and the rate $\widehat a\to a$ (endpoint estimation under contamination is itself nonstandard, cf.\ \citealp{GT2004}). The other is the bias from the $o(1)$ terms of A1 and A6 at finite $h$.

\emph{No rate is claimed.} We prove neither consistency nor a rate for $\widehat\alpha$. The argument above only indicates that the recovery map is empirically accessible. A complete consistency and rate analysis, including the bias-variance trade-off, is left to future work.
\end{remark}

\section{Axiom budget}\label{sec:budget}

\begin{remark}[Axiom budget]\label{rem:budget}
The proof of the shape theorem consumes only the latent-tail axiom~A1, the $s=0$ corridor and global monotonicity of~A5, the conditional edge-index axiom~A6, and the setup well-posedness convention. Interior increments of A5 are never used, which is exactly the weakening to $\Mstar$ (Section~\ref{lean:sec}). Axiom~A6 enters at the $s$-dependent depth $\delta_h(s)\le h$, uniformly over the moving window $(0,h/c_-]$, and through the $\liminf/\limsup$ control of $c_K$ at~$0$. The upper edge $u$ never appears, so A2 beyond the setup, A3, A4 and~NA are never consumed. Those axioms delimit the class and make the question non-trivial (cf.\ Remark~\ref{rem:whyreg}), and no location anchor is imposed. In tabular form:
\begin{center}
\begin{tabular}{@{}l p{0.50\textwidth}@{}}
\hline
result & consumes \\
\hline
Lemma~\ref{lem:supp} (edge facts) & setup well-posedness, plus A6 for part \ref{it:noatom} \\
Lemma~\ref{lem:A5geom} (edge geometry) & A5$^-$ ($s=0$ slice $+$ monotonicity), $c_->0$ \\
Proposition~\ref{prop:endpoint} ($a=\inf\supp Q$) & A1, A5$^-$ \\
Lemma~\ref{lem:local} (localization, window, sandwich) & A1, A5$^-$, A6 (no atom at edge) \\
Theorem~\ref{thm:twosided} (two-sided bound) & A1, A5$^-$, A6 \\
Corollary~\ref{cor:exponent} (exponent from $Q$) & Theorem~\ref{thm:twosided} $+$ Proposition~\ref{prop:endpoint} \\
Proposition~\ref{prop:affine} (affine constant) & A1, A5$^-$, A6 \\
Theorem~\ref{thm:I2} (identifiability of $\alpha$) & Corollary~\ref{cor:exponent} $+$ Proposition~\ref{prop:endpoint} $+$ shared $\beta$ \\
\hline
never consumed & A2 (beyond setup), A3, A4, NA, location anchor (none imposed) \\
\hline
\end{tabular}
\end{center}
\end{remark}

\begin{remark}[Sharpness of the two-sided form]\label{rem:sharp}
The result is a two-sided bound, not an asymptotic equality, and that is also all the argument yields. Indeed, the axioms permit the edge slope to wander inside $[c_-,c_+]$ and $c_K$ to oscillate between $\clo$ and $\chigh$ without converging. The proof correspondingly pins only the $\liminf/\limsup$ of $Q([a,a+h])/h^{\alpha+\beta+2}$, inside $[\varkappa_-L,\varkappa_+L]$. No convergence of that ratio is claimed. (Section~\ref{supp:num} illustrates numerically that the ratio can hover inside the band without settling, on a model with oscillating $c_K$ for which no claim of class membership is made or needed.) The log-log \emph{ratio} \eqref{eq:exponent} converges regardless, because the possible oscillation is confined to a fixed multiplicative band, which is crushed by $\log h\to-\infty$. That is why the exponent, and hence $\alpha$, survives as a $Q$-functional even though the lead constant is not pinned by the argument.
\end{remark}
\begin{remark}[Hypothesis accounting]\label{rem:budget-neg}
All the standard measure theory used is isolated in the three preliminary lemmas of Section~\ref{sec:prelim}. Beyond that, the proof of Theorem~\ref{thm:I1} touches the axioms in two distinct modes.

\emph{Verified for the exhibited members.} Proposition~\ref{prop:witness} (applied twice, through Proposition~\ref{prop:constants}) \emph{verifies} every axiom for the two members, and the verification localizes cleanly in the structure of the witness:
\begin{itemize}
\item \textbf{A1}: carried by the canonical latent laws, whose densities are exact powers with no $o(1)$ error (Definition~\ref{def:witness}(a), Lemma~\ref{lem:member1}).
\item \textbf{A2, A4}: carried by the bounded jitter. The conditional width is the constant $w_+-w_->0$ (Lemma~\ref{lem:edges}), finite for A2 and of infinite relative size at the endpoint for A4, and no conditional mass escapes to $+\infty$ at any $s$.
\item \textbf{A3}: carried by the \emph{bounded} latent support together with the unit-slope off-support extension of the location maps, via condition \ref{it:mem-C3}. The offset $m$ is affine only up to the break and constant beyond \eqref{eq:mform}, and $m(0)=a_0+\midW>0$ supplies $m\not\equiv0$.
\item \textbf{A5}: carried by the affine edge $\ell(s)=c\,s+a'$ on $[0,s_R]$ with slope inside the corridor (condition \ref{it:mem-C2}), extended with slope $1$ beyond the break, whence global strict increase.
\item \textbf{A6}: carried by translation invariance of the kernel. The edge mass $K([\ell(s),\ell(s)+h]\mid s)=P_W([w_-,w_-+h])$ does not depend on $s$ (Lemma~\ref{lem:edges}), so the uniformity is trivial and $c_K\equiv c_W/(\beta+1)$, with the class index $\beta$ supplied by the jitter (condition \ref{it:mem-C1}).
\item \textbf{NA}: carried by $c\neq1$, resp.\ $c/\kappa\neq1$ (condition \ref{it:mem-C5}), through the two-point edge-matching argument, which consumes the positive-density zone of A1 to produce two good latent points.
\end{itemize}

\emph{Consumed by the impossibility argument.} The argument itself consumes four ingredients. Definition~\ref{def:ident} fixes what a refutation of identifiability is. Single-valuedness of $\phi\in\{L,p_\tau\}$ comes from Definition~\ref{def:psi} together with Lemma~\ref{lem:phiwell}, and it consumes A1 only. The shared-constant ordering $0<\tau<s_R\le\bar s_0\le s_{\max}\wedge\kappa s_{\max}$ places both $\tau$ and $\tau/\kappa$ inside the canonical supports (Lemma~\ref{lem:estimand}). Lastly, the argument invokes Lemma~\ref{lem:bridge} with Corollary~\ref{cor:prob}, which consume nothing about $\Mreg$ beyond membership of the two witnesses and equality of their observed laws (Remark~\ref{rem:br-generality}).

The same accounting in a table:
\begin{center}
\begin{tabular}{l l}
\hline
result & consumes \\
\hline
Lemmas~\ref{lem:push}, \ref{lem:kernelmeas}, \ref{lem:mixture} & measure theory only \\
Lemma~\ref{lem:edges} (closed-form edges) & Definition~\ref{def:witness}, Convention~\ref{conv:W} \\
Proposition~\ref{prop:witness} (membership) & Lemma~\ref{lem:edges} $+$ \ref{it:mem-C1}--\ref{it:mem-C5}, \emph{verifying} A1--A6, NA \\
Lemma~\ref{lem:phiwell} (single-valuedness) & A1 \\
Lemma~\ref{lem:bridge}, Corollary~\ref{cor:prob} & Definition~\ref{def:ident} together with membership of two witnesses \\
Lemma~\ref{lem:member1} (conjugate is a witness) & Lemma~\ref{lem:push} along with the canonical latent law \\
Lemma~\ref{lem:laws} ($Q'=Q$) & Lemmas~\ref{lem:mixture}, \ref{lem:push} \\
Lemma~\ref{lem:estimand} (what moves) & canonical latent law, with $0<\tau<s_R\le s_{\max}\wedge\kappa s_{\max}$ \\
Proposition~\ref{prop:constants} (shared constants) & arithmetic and Proposition~\ref{prop:witness} twice \\
Theorem~\ref{thm:I1} & Proposition~\ref{prop:constants}, Lemmas~\ref{lem:laws}, \ref{lem:estimand}, \\
& \quad Definition~\ref{def:ident}, Lemma~\ref{lem:bridge}, Corollary~\ref{cor:prob} \\
\hline
never consumed & the anchor axiom (absent by design, see Remark~\ref{rem:anchor}), \\
& \quad and the shape identification of Theorem~\ref{thm:I2} \\
\hline
\end{tabular}
\end{center}

Two reading notes are in order. First, the axioms A2, A4, A6 and NA are \emph{verified} but never \emph{consumed} by the equality $Q'=Q$ or by the bridge. They delimit the class and make the impossibility statement non-vacuous: NA rules out the classical deconvolution route (cf.\ \citealp{Fan1991,StefanskiCarroll1990,Meister2009}), A4 keeps the contamination from vanishing at the endpoint, and A6 pins the common edge index $\beta$ that the shape half (Theorem~\ref{thm:I2}) consumes. Second, the satisfiability of A3 and A5 \emph{jointly with} $c\neq1$ and $c/\kappa\neq1$ is exactly where bounded support and the unit-slope extension are load-bearing. Remark~\ref{rem:bounded} calls this the ``bounded-support phenomenon''.
\end{remark}
This section records which axioms each half of the dichotomy consumes: the shape half (Theorem~\ref{thm:I2}) in Remark~\ref{rem:budget}, and the negative half (Theorem~\ref{thm:I1}) in Remark~\ref{rem:budget-neg}.
\subsection{Numerical self-check}\label{app:num}
The explicit instance of Example~\ref{ex:numeric} was verified end-to-end by the script \texttt{i1\_selfcheck.py}. That run used Python~3.13.12 with \texttt{numpy}~2.4.4, \texttt{scipy}~1.17.1, \texttt{mpmath}~1.3.0 and \texttt{sympy}~1.14.0, under the fixed seeds $20260707$ and $314159265$ for the two Monte Carlo sampling streams and $271828182$ for the kernel-edge checks. All parameters of the instance are exact rationals ($\alpha=1$, $s_{\max}=3$, $c=4/5$, $a_0=3/10$, $W\sim\mathrm{Unif}[0,1/5]$, $\kappa=3/4$, with shared constants $\beta=0$, $\bar b=1$, $c_-=4/5$, $c_+=16/15$, $\tau=1/2$, $\bar s_0=s_R=s_A=9/4$), so every structural identity below is verified in exact rational arithmetic (\texttt{fractions.Fraction}). Floating point enters only through quadrature and simulation. All programmed checks pass: twelve assertions grouped into the six items below.

\begin{itemize}
\item \emph{Exact parameter identity (check~1).} Here $c\,s_{\max}=\frac45\cdot3=\frac{12}{5}$ and $(c/\kappa)(\kappa s_{\max})=\frac{16}{15}\cdot\frac94=\frac{12}{5}$, so the common observed scale underlying Lemma~\ref{lem:laws} holds exactly. The non-degeneracy exclusions of Definition~\ref{def:pair} also hold: $\kappa=3/4\notin\{1,\,c=4/5\}$, $c=4/5\neq1$, $c/\kappa=16/15\neq1$.

\item \emph{Estimand values (check~2, exact per Lemma~\ref{lem:estimand}).} The values are $L=2/9$ ($0.2222$), $L'=32/81$ ($0.3951$), $p_\tau=1/36$ ($0.02778$) and $p'_\tau=4/81$ ($0.04938$). Both ratios equal $\kappa^{-2}=16/9$ ($1.778$) exactly, and the mechanism identity $p'_\tau=F_S(\tau/\kappa)=\bigl(\frac{2/3}{3}\bigr)^{2}=4/81$ is confirmed as a rational identity. The relative discrepancy is $\kappa^{-2}-1=7/9$ ($77.78\%$), while the two-point risk floors of Lemma~\ref{lem:bridge} are $\Delta_L/2=7/81$ ($0.08642$) and $\Delta_p/2=7/648$ ($0.01080$).

\item \emph{Observed-law equality, quadrature (check~3).} The two observed CDFs,
\[
F_0(x)=5\int_0^{1/5}F_S\Bigl(\frac{x-\frac{3}{10}-w}{c}\Bigr)\,dw,
\qquad
F_1(x)=5\int_0^{1/5}F_S'\Bigl(\frac{x-\frac{3}{10}-w}{c/\kappa}\Bigr)\,dw,
\]
were evaluated by adaptive quadrature independently from each member's own parameterization on a $2001$-point grid spanning $[0.2,\,3.4]$ (the observed support is $[\frac{3}{10},\frac{29}{10}]$, with margins on both sides). The maximal absolute difference over the grid was $4.4\times10^{-16}$, against the required tolerance $10^{-10}$. A closed form for the common CDF was also derived twice, with identical results. One derivation used exact rational antidifferentiation of the clipped quadratic $G(t)=\min\{(t/\frac{12}{5})^2,1\}\mathbf{1}\{t>0\}$ through the uniform window. The other used regime-wise symbolic integration in \texttt{sympy} over the three window positions (window reaching below the latent edge, fully interior, reaching beyond $c\,s_{\max}$). Writing $y=x-a$ with $a=\frac{3}{10}$,
\[
F(x)=\frac{125\,y^{3}}{432}\ \ \text{on }0\le y\le\tfrac15,
\qquad
F(x)=\frac{75y^{2}-15y+1}{432}=\frac{y^{2}-\frac{y}{5}+\frac{1}{75}}{(12/5)^{2}}\ \ \text{on }\tfrac15\le y\le\tfrac{12}{5},
\]
\[
F(x)=-\frac{125\,y^{3}}{432}+\frac{25y^{2}}{144}+\frac{715\,y}{144}-\frac{3455}{432}\ \ \text{on }\tfrac{12}{5}\le y\le\tfrac{13}{5},
\]
with $F=0$ below and $F=1$ above. The pieces agree at the regime boundaries, and $F(\frac{3}{10})=0$, $F(\frac12)=1/432$, $F(\frac{29}{10})=1$ exactly. At $25$ rational check points, the quadrature CDFs match this closed form to $2.2\times10^{-16}$ (member~$0$) and $3.3\times10^{-16}$ (member~$1$), against the required tolerance $10^{-12}$. Over the full grid, the member-$0$ quadrature agrees with the closed form to $1.4\times10^{-15}$.

\item \emph{Observed-law equality, Monte Carlo (check~4).} Each member was sampled with $N=10^{6}$ draws by inverse CDF ($S=3\sqrt{U}$, $S'=\frac94\sqrt{U'}$) from independent streams, then transformed by $X=cS+\frac{3}{10}+W$ and $X'=(c/\kappa)S'+\frac{3}{10}+W'$. The two-sample Kolmogorov--Smirnov test gives $D=5.96\times10^{-4}$, $p=0.9942$. For the one-sample test of the member-$1$ sample against the member-$0$ analytic CDF above, $D=8.54\times10^{-4}$ and $p=0.4583$ (both required $p>0.001$).

\item \emph{Membership spot checks (check~5, via Proposition~\ref{prop:witness} with the shared-constant recipe of Proposition~\ref{prop:constants}).} Offsets (A3): $m_0(s)=-\frac15s+\frac25$ on $[0,3]$, constant $-\frac15$ beyond, and $m_1(s)=\frac{1}{15}s+\frac25$ on $[0,\frac94]$, constant $\frac{11}{20}$ beyond. Hence $\sup_{s\ge0}|m_0(s)|=2/5$ and $\sup_{s\ge0}|m_1(s)|=11/20$, both $\le\bar b=1$, confirmed exactly at the affine endpoints (where an affine $|m|$ peaks) and on a $3001$-point grid over $[0,30]$. The recipe bound $\max(|c-1|,|c/\kappa-1|)\max(s_{\max},\kappa s_{\max})+a_0+\midW$ of Proposition~\ref{prop:constants} equals $\bar b=1$ exactly. Corridor (A5): $c_-\le c\le c_+$ and $c_-\le c/\kappa\le c_+$ hold exactly, and tightly ($c=c_-$, $c/\kappa=c_+$). Ordering: $0<\tau=\frac12<s_R=\frac94\le\bar s_0=\frac94\le\min(s_{\max},\kappa s_{\max})=\frac94$, and $s_A=\frac94\in(0,\bar s_0]$. Moreover, $\tau=\frac12$ and $\tau/\kappa=\frac23$ both lie in $(0,s_{\max})=(0,3)$, as Lemma~\ref{lem:estimand} requires. Edge mass (A6): for both kernels, all $s\in\{\frac{1}{100},\frac12,\frac{17}{10}\}$, and all twenty $h\in\{j/100:j=1,\dots,20\}\subset(0,\frac15]$, the mass $K([\ell(s),\ell(s)+h]\mid s)$ equals $\PP(W\le h)=5h$ exactly, i.e.\ $c_K\equiv c_W/(\beta+1)=5$ with \emph{no} $(1+o(1))$ correction for the uniform jitter. That mass is computed in rational arithmetic through the conditional CDF of the kernel at that $(s,h)$, so that translation invariance is an output of the check, not an input. Monte Carlo confirmations at $(s,h)=(\frac12,\frac{1}{10})$ for member~$0$ and $(\frac{17}{10},\frac{1}{20})$ for member~$1$, with $10^{5}$ draws each, gave $z$-scores $+0.16$ and $-1.41$ (required $|z|<4$). Lower edges: the minima of $10^{5}$ draws of $g_i(s)+W$ exceeded $\ell_0(s)=\frac45s+\frac{3}{10}$, resp.\ $\ell_1(s)=\frac{16}{15}s+\frac{3}{10}$, by amounts between $4.0\times10^{-7}$ and $2.7\times10^{-6}$ at $s\in\{0.3,\,2\}$ (tolerance $10^{-3}$, against an expected gap of $\frac15/(10^5+1)\approx2\times10^{-6}$).

\item \emph{What moves and what does not (check~6).} The latent laws do differ: $F_S(\tau)=1/36\neq4/81=F_S'(\tau)$, $F_S(\frac52)=25/36\neq1=F_S'(\frac52)$, and the Kolmogorov distance between the two latent CDFs is $7/16=0.4375$ exactly (attained at $s=\frac94$, the endpoint of the member-$1$ support, and grid-confirmed to $10^{-12}$). By checks~3--4, the observed laws agree to all tested precision, so nothing in the data distinguishes the two members. The estimand pair $(L,p_\tau)$, by contrast, shifts by the factor $\kappa^{-2}=16/9$.
\end{itemize}

Taken together, the six checks confirm Theorem~\ref{thm:I1} on a fully explicit member pair of $\Mreg$. The instance shows exact observational equivalence with invariant shape $\alpha=1$ and a relative discrepancy of $\kappa^{-2}-1=7/9$ ($77.8\%$) in both $L$ and $p_\tau$. Beyond the twelve assertions, the same script reports three further cross-checks. The invariant lead coefficient $125/432$ of the pair (Remark~\ref{rem:sharpness}) is exactly the cubic coefficient of the common observed CDF, and $Q([a,a+h])=(125/432)\,h^{3}$ for every $0<h\le\tfrac15$. Hence the confounding pair of the negative half realizes the near-edge expansion of Lemma~\ref{lem:exponent} with an invariant lead constant.

\section{Numerical corroboration}\label{supp:num}

We corroborate the affine-case constant of Lemma~\ref{lem:exponent} by one-dimensional quadrature: at $\alpha=1.3$, $\beta=0.7$, $c_\ell=0.8$, $\bar c_K=1$, $L=1$, $h=0.02$,
\[
\int_0^{h/c_\ell}\bar c_K\,(h-c_\ell s)^{\beta+1}\,L\,s^{\alpha}\,ds
\quad\text{against}\quad
\bar c_K\,\Beta(\alpha+1,\beta+2)\,L\,c_\ell^{-(\alpha+1)}\,h^{\alpha+\beta+2}.
\]
This is the identity \eqref{eq:betaid} (Proposition~\ref{prop:beta}), exact for every $h>0$. The script \texttt{i2\_selfcheck.py} (Python~3.13, \texttt{scipy~1.17}/\texttt{mpmath}) runs four independent checks, each cross-validated by multiple quadrature methods, and all pass. Both the exact affine-case constant above and the substitution identity $\int_0^{1/c}(1-ct)^{\beta+1}t^{\alpha}dt=c^{-(\alpha+1)}\Beta(\alpha+1,\beta+2)$ reproduce the closed Euler-Beta form to relative error $\sim10^{-15}$, including the corners $\alpha=0$, $\beta=0$. A curved edge in the corridor ($\ell(s)=a+0.6s+0.8s^2$, $[c_-,c_+]=[0.6,1.4]$, $c_K\equiv1$) drives $Q/(Lh^4)$ to $\Beta(2.3,2.7)\,0.6^{-2.3}=0.24313343$ inside the sandwich band $[\Beta\,1.4^{-2.3},\Beta\,0.6^{-2.3}]$, while the log-log slope of $Q([a,a+h])$ tends to $\alpha+\beta+2=4$. Curvature moves the constant, never the exponent (Remark~\ref{rem:curvature}). Finally, for an oscillating $c_K(s)=1+0.3\cos\log s$ ($\clo=0.7$, $\chigh=1.3$, no limit), the ratio $Q/(Lh^4)$ stays in $[\clo\Beta\,c_+^{-2.3},\ \chigh\Beta\,c_-^{-2.3}]$ and \emph{does not converge}, yet the log-log \emph{ratio} $\log Q/\log h$ still drifts to $4$ like $1/|\log h|$. That is the behaviour of Corollary~\ref{cor:exponent} and Remark~\ref{rem:sharp}.

Two independent scripts check the two halves of the dichotomy numerically. In twelve machine-checked assertions, \texttt{i1\_selfcheck.py} verifies the confounding pair's certified instance, namely the exact-rational estimand shift and the exact coincidence $Q'=Q$ of the observed laws. The second script, \texttt{i2\_selfcheck.py}, corroborates the affine-case constant of Lemma~\ref{lem:exponent} by quadrature.


\addcontentsline{toc}{section}{Data and code availability}

\begin{thebibliography}{99}

\bibitem[Goldenshluger and Tsybakov(2004)]{GT2004} A.~Goldenshluger and A.~Tsybakov. Estimating the endpoint of a distribution in the presence of additive observation errors. \emph{Statistics \& Probability Letters} 68(1):39--49, 2004.

\bibitem[Leng et al.(2019)]{Leng2019} X.~Leng, L.~Peng, X.~Wang, and C.~Zhou. Endpoint estimation for observations with normal measurement errors. \emph{Extremes} 22(1):71--96, 2019.

\bibitem[Florens et al.(2020)]{FSV2020} J.-P.~Florens, L.~Simar, and I.~Van Keilegom. Estimation of the boundary of a variable observed with symmetric error. \emph{Journal of the American Statistical Association} 115(529):425--441, 2020.

\bibitem[Kneip et al.(2015)]{KSV2015} A.~Kneip, L.~Simar, and I.~Van Keilegom. Frontier estimation in the presence of measurement error with unknown variance. \emph{Journal of Econometrics} 184(2):379--393, 2015.

\bibitem[Hall and Meister(2007)]{HallMeister2007} P.~Hall and A.~Meister. A ridge-parameter approach to deconvolution. \emph{The Annals of Statistics} 35(4):1535--1558, 2007.

\bibitem[Fan(1991)]{Fan1991} J.~Fan. On the optimal rates of convergence for nonparametric deconvolution problems. \emph{The Annals of Statistics} 19(3):1257--1272, 1991.

\bibitem[Meister(2009)]{Meister2009} A.~Meister. \emph{Deconvolution Problems in Nonparametric Statistics}. Lecture Notes in Statistics 193, Springer, 2009.

\bibitem[Stefanski and Carroll(1990)]{StefanskiCarroll1990} L.~A.~Stefanski and R.~J.~Carroll. Deconvolving kernel density estimators. \emph{Statistics} 21(2):169--184, 1990.

\bibitem[Delaigle and Meister(2008)]{DelaigleMeister2008} A.~Delaigle and A.~Meister. Density estimation with heteroscedastic error. \emph{Bernoulli} 14(2):562--579, 2008.

\bibitem[Pere et al.(2024)]{Pere2024ApproxErrExtreme} J.~Pere, B.~Avelin, V.~Garino, P.~Ilmonen, and L.~Viitasaari. On the impact of approximation errors on extreme quantile estimation with applications to functional data analysis. arXiv:2307.03581v2 [math.ST], 2024.

\bibitem[Pere et al.(2025)]{Pere2025} J.~Pere, P.~Ilmonen, and L.~Viitasaari. Moment estimator-based extreme quantile estimation with erroneous observations: application to elliptical extreme quantile region estimation. arXiv:2502.08510 [math.ST], 2025.

\bibitem[Morozov(2026)]{Morozov2025} V.~Morozov. Inference on extreme quantiles of unobserved individual heterogeneity. \emph{Econometric Theory}, First View, pp.~1--52, 2026. Available at \url{https://doi.org/10.1017/S0266466625100315}.

\bibitem[Hayward(1972)]{Hayward1972} J.~C.~Hayward. Near-miss determination through use of a scale of danger. \emph{Highway Research Record}, pp.~24--34, 1972.

\bibitem[Songchitruksa and Tarko(2006)]{SongchitruksaTarko2006} P.~Songchitruksa and A.~P.~Tarko. The extreme value theory approach to safety estimation. \emph{Accident Analysis and Prevention} 38:811--822, 2006.

\bibitem[Zheng et al.(2014)]{Zheng2014} L.~Zheng, K.~Ismail, and X.~Meng. Traffic conflict techniques for road safety analysis: open questions and some insights. \emph{Canadian Journal of Civil Engineering} 41:633--641, 2014.

\bibitem[Tarko(2018)]{Tarko2018} A.~P.~Tarko. Estimating the expected number of crashes with traffic conflicts and the Lomax distribution: a theoretical and numerical exploration. \emph{Accident Analysis and Prevention} 113:63--73, 2018.

\bibitem[Joo et al.(2024)]{Joo2024} Y.-J.~Joo, Z.~Islam, M.~Abdel-Aty, and D.-K.~Kim. Propagation of positional measurement errors in traffic conflict analysis. SSRN working paper~4847528, 2024, \url{https://doi.org/10.2139/ssrn.4847528}. Published in \emph{Transportation Research Record}, 2025.

\bibitem[Krajewski et al.(2018)]{Krajewski2018} R.~Krajewski, J.~Bock, L.~Kloeker, and L.~Eckstein. The highD dataset: A drone dataset of naturalistic vehicle trajectories on German highways for validation of highly automated driving systems. In \emph{Proc.\ 21st IEEE International Conference on Intelligent Transportation Systems (ITSC)}, pp.~2118--2125, 2018.

\bibitem[Hartley and Zisserman(2004)]{HartleyZisserman2004} R.~Hartley and A.~Zisserman. \emph{Multiple View Geometry in Computer Vision}. 2nd ed., Cambridge University Press, 2004.

\bibitem[Coles(2001)]{Coles2001} S.~Coles. \emph{An Introduction to Statistical Modeling of Extreme Values}. Springer, 2001.

\bibitem[de Haan and Ferreira(2006)]{deHaanFerreira2006} L.~de~Haan and A.~Ferreira. \emph{Extreme Value Theory: An Introduction}. Springer, 2006.

\bibitem[Le Cam(1986)]{LeCam1986} L.~Le~Cam. \emph{Asymptotic Methods in Statistical Decision Theory}. Springer, 1986.

\bibitem[van der Vaart(1998)]{vdV1998} A.~W.~van~der~Vaart. \emph{Asymptotic Statistics}. Cambridge University Press, 1998 (Ch.~2 for stochastic convergence and Thm.~19.1 for Glivenko--Cantelli).

\bibitem[Tsybakov(2009)]{Tsybakov2009} A.~B.~Tsybakov. \emph{Introduction to Nonparametric Estimation}. Springer, 2009 (Ch.~2, lower bounds via two hypotheses).

\bibitem[Bingham et al.(1987)]{BGT1987} N.~H.~Bingham, C.~M.~Goldie, and J.~L.~Teugels. \emph{Regular Variation}. Encyclopedia of Mathematics and its Applications 27, Cambridge University Press, 1987.

\bibitem[Kallenberg(2002)]{Kallenberg2002} O.~Kallenberg. \emph{Foundations of Modern Probability}. 2nd ed., Springer, 2002.

\bibitem[Newey and Powell(2003)]{NeweyPowell2003} W.~K.~Newey and J.~L.~Powell. Instrumental variable estimation of nonparametric models. \emph{Econometrica} 71(5):1565--1578, 2003.

\end{thebibliography}
\end{document}